\documentclass{amsart}

\usepackage{amsmath,amsfonts,amssymb,amsthm,amscd,amsbsy}
\usepackage{latexsym,pdfsync,xcolor,graphicx}

\usepackage[utf8]{inputenc}												
\usepackage[english]{babel}			

\usepackage{hyperref}
\usepackage{enumerate}
\usepackage{longtable}
\usepackage{float}
\usepackage{multirow,array}
\usepackage{graphicx}
\usepackage{caption}
\usepackage{subcaption}
\usepackage{multicol}
\usepackage[T1]{fontenc}

\newtheorem{theorem}{Theorem}[section]
\newtheorem{corollary}[theorem]{Corollary}
\newtheorem{lemma}[theorem]{Lemma}
\newtheorem{proposition}[theorem]{Proposition}
\newtheorem{remark}[theorem]{Remark}
\numberwithin{equation}{section}

\providecommand{\abs}[1]{\left| #1 \right| }

\begin{document}

\title[A class of Kolmogorov systems]
{Phase portraits of a family of Kolmogorov systems depending on six parameters}

\author[E. Diz-Pita, J. Llibre and M.V. Otero-Espinar]
{\'Erika Diz-Pita$^1$, Jaume Llibre$^2$ and M. Victoria Otero-Espinar$^1$}

\address{$^1$ Departamento de Estat\'istica, An\'alise Matem\'atica e Optimizaci\'on, Universidade de Santiago de Compostela, 15782 Santiago de Compostela, Spain }
\email{erikadiz.pita@usc.es, mvictoria.otero@usc.es}

\address{$^2$ Departament de Matem\`atiques, Universitat Aut\`onoma de Barcelona, 08193 Bellaterra, Barcelona, Spain}
\email{jllibre@mat.uab.cat}


\date{}
\dedicatory{}

\maketitle  

\begin{abstract}
Consider a general $3$-dimensional Lotka-Volterra system with a rational first integral of degree two of the form $H=x^i y^j z^k$. The restriction of this Lotka-Volterra system to each surface $H(x,y,z)=h$ varying $h\in \mathbb{R}$ provide Kolmogorov systems. With the additional assumption that they have a Darboux invariant of the form $x^\ell y^m e^{st}$ they reduce to the Kolmogorov systems
\begin{equation*}
\begin{split}
\dot{x}&=x \left( a_0- \mu (c_1 x + c_2 z^2 + c_3 z)\right),\\
\dot{z}&=z\left( c_0+ c_1 x + c_2 z^2 + c_3 z\right).
\end{split}
\end{equation*}
In this paper we classify the phase portraits in the Poincar\'e disc of all these Kolmogorov systems which depend on six parameters.
\end{abstract}

\section{Introduction}

The Lotka-Volterra systems have been used for modelling many natural phenomena, such as the time evolution of conflicting species in biology  \cite{compesp}, chemical reactions, plasma physics \cite{plasma} or hydrodynamics \cite{hidrodinamica}, just as other problems from social science and
economics.  

These systems, which are polynomial differential equations of degree two, were initially proposed, independently, by Alfred J. Lotka in 1925 and Vito Volterra in 1926, both in the context of competing species. Later on Lotka-Volterra systems were generalized and considered in arbitrary dimension, i.e.
$$
\dot{x_i}=x_i  \left( a_{i0} + \sum_{j=1}^{n} a_{ij}x_j \right)  , \;\;\; i=1,...,n.
$$
Consequently the applications of these systems started to multiply.
Moreover Kolmogorov in \cite{Kolmogorov} extended the Lotka-Volterra systems as follows
$$
\dot{x_i}=x_i  P_i(x_1,\ldots,x_n)  , \;\;\; i=1,...,n,
$$
where $P_i$ are polynomials of degre at most $m$. These kind of systems are now known as Kolmogorov systems. They have in particular all the applications of the Lotka-Volterra systems as for instance in the study of the black holes in cosmology, see \cite{cosmologia}.

The global qualitative dynamics of the Lotka-Volterra systems in dimension two has been completely studied in \cite {LVdim2}, where all possible phase portraits on the Poincar\'e disc have been classified.

There are few results about the global dynamics of the Lotka-Volterra systems in dimension three. Our objective is to study the phase portraits of the $3$-dimensional Lotka-Volterra systems
\begin{equation}\label{systemLV3}
\begin{split}
& \dot{x} = x ( a_0 + a_1 x + a_2 y + a_3 z ),\\
& \dot{y} = y ( b_0 + b_1 x + b_2 y + b_3 z ), \\
& \dot{z} = z  (c_0 + c_1 x + c_2 y + c_3 z ),
\end{split}
\end{equation}
which have a rational first integral of degree two of the form $x^i y^j z^k$. We have used the Darboux theory of integrability to obtain a characterization of these systems. As a result, we have reduced the initial problem to a problem in dimension two, the study of the global dynamics of two families of Kolmogorov systems. In this paper we focus on the first family, which is
\begin{equation}\label{E1}
\begin{split}
\dot{x}&=x ( a_0+a_1x+a_2z^2+a_3z ) ,\\
\dot{z}&=z ( c_0+c_1x+c_2z^2+c_3z ) ,
\end{split}
\end{equation}	

Kolmogorov systems \eqref{E1} depend on eight parameters, this is a big number in order to classify all their distinct topological phase portraits. Then we require that Kolmogorov systems \eqref{E1} have a Darboux invariant of the form $x^\ell y^m e^{st}$, then these systems are reduced to study the Kolmogorov systems
\begin{equation}\label{system}
\begin{split}
\dot{x}&=x\left(a_0- \mu (c_1 x + c_2 z^2 + c_3 z)\right),\\
\dot{z}&=z\left(c_0+ c_1 x + c_2 z^2 + c_3 z\right),
\end{split}
\end{equation}
which now depend on six parameters. For these Kolmogorov systems we give the topological classification of all their phase portraits in the Poincar\'e disc. Roughly speaking the Poincar\'e disc is the closed unit disc centered at the origin of $\mathbb{R}^2$. Its interior is identified with $\mathbb{R}^2$ and the circle of its boundary is identified with the infinity of $\mathbb{R}^2$. In the plane $\mathbb{R}^2$ we can go or come from the infinity in as many directions as points have the circle. The polynomial differential systems can be extended to the closed Poincar\'e disc, i.e. they can be extended to infinity and in this way we can study their dynamics in a neighborhood of infinity. This extension is called the Poincar\'e compactification, for more details see subsection 2.1.
Thus our main result is the following.

\begin{theorem}\label{th_global}
Kolmogorov systems \eqref{system} have $102$ topologically distinct phase portraits in the Poincar\'e disc under condition $(H_2)$ given in Figure \ref{fig:global_sis2.1}.
\end{theorem}

The condition $(H_2)$ is
\begin{align*}
\left\lbrace c_2\neq 0, a_0\geq 0, c_1\geq 0, c_3\geq 0, a_0+c_0\mu\neq 0, a_0c_1\mu\neq0, \mu\neq-1 \right\rbrace.
\end{align*}
We will see that we can assume condition $(H_2)$ because Kolmogorov systems \eqref{system} can be reduced to satisfy such condition either using symmetries, or eliminating known phase portraits, or eliminating phase portraits with infinitely many finite or infinite singular points.

Other papers where some topological phase portraits have been classified in the Poincar\'e disc are, for instance, \cite{Ben-LLi,Jim-Lli-Med,Lli-Per-Pess}.  

In Section \ref{sec:reduction} using the Darboux theory of integrability we explain the reduction from the Lotka-Volterra system \eqref{systemLV3} to the Kolmogorov systems \eqref{system}. In Section \ref{sec:properties} we give some properties of the system obtained. In Section \ref{sec:finite} we study the local phase portrait of the finite singular points, and in Section \ref{sec:infinite} we do the same with the infinite singular points, applying the blow-up technique. Finally in Section \ref{sec:global} we prove Theorem \ref{th_global}.

\section{Preliminaries}\label{sec:preliminaries}

\subsection{Poincar\'e compactification} \label{subsec:Poincare}

In order to study the behavior of the trajectories of our polynomial differential systems near the infinity we will use the Poincar\'e compactification. We provide a short summary about this method, more details can be found in Chapter 5 of \cite{Libro}.

Let $X=(P(x,y),Q(x,y))$ be a polynomial vector field of degree $d$ defined in $\mathbb{R}^2$. Consider the \textit{Poincar\'e sphere} $\mathbb{S}^2=\left\lbrace  y \in \mathbb{R}^3 : y_1^2+y_2^2+y_3^2=1 \right\rbrace$ and its tangent plane at the point $(0,0,1)$ which is identified with $\mathbb{R}^2$.

We consider the central projections $f^+\colon\mathbb{R}^2\to \mathbb{S}^2$ and $f^-\colon\mathbb{R}^2\to \mathbb{S}^2$. By definition, $f^+(x)$ is the intersection of the straight line passing through the point $x$ and the origin with the northern hemisphere of $\mathbb{S}^2$, and respectively for $f^-(x)$ with the southern hemisphere. The differential $Df^+$ and respectively $Df^-$ send the vector field $X$ into a vector field $\overline{X}$ on  $\mathbb{S}^2\textbackslash \mathbb{S}^1$.  Note that the points at infinity of $\mathbb{R}^2$ are in bijective correspondence with the points of the equator $\mathbb{S}^1$ of $\mathbb{S}^2$.

The vector field $\overline{X}$ can be extended analytically to a vector field on $\mathbb{S}^2$ multiplying $\overline{X}$ by $y_3^d$. We denothe this vector field by $\rho(X)$, and it is called the \textit{Poincar\'e compactification} of the vector field $X$ on $\mathbb{R}^2$.

For studying the dynamics of $X$ in the neighborhood of the infinity, we must study the dynamics of $\rho(X)$ near $\mathbb{S}^1$. The sphere $\mathbb{S}^2$ is a 2-dimensional manifold so we need to know the expressions of the vector field $\rho(X)$ in the local charts $(U_i,\phi_i)$ and $(V_i,\psi_i)$, where $U_i= \left\lbrace y \in \mathbb{S}^2 :y_i > 0  \right\rbrace$, $V_i = \left\lbrace y \in \mathbb{S}^2 : y_i < 0  \right\rbrace$, $\phi_i: U_i \longrightarrow \mathbb{R}^2$ and $\psi_i: V_i \longrightarrow \mathbb{R}^2$ for $i=1,2,3$ with $\phi_i(y) = - \psi_i(y) = \left( y_m/y_i , y_n/y_i \right)$ for $m < n$ and $m,n \neq i$.

In the local chart $(U_1,\phi_1)$ the expression of $\rho(X)$ is
\begin{equation}\label{Poincare_comp_U1}
\dot{u}= v^d \left[ -u \: P\left( \frac{1}{v}, \frac{u}{v}\right)  + Q\left( \frac{1}{v}, \frac{u}{v}\right) \right], \;
\dot{v}= - v^{d+1} \: P\left( \frac{1}{v}, \frac{u}{v}\right).
\end{equation}
In the local chart $(U_2,\phi_2)$ the expression of $\rho(X)$ is
\begin{equation}\label{Poincare_comp_U2}
\dot{u}= v^d \left[ P\left( \frac{1}{v}, \frac{u}{v}\right) - u Q\left( \frac{1}{v}, \frac{u}{v}\right) \right], \;
\dot{v}= - v^{d+1} \: P\left( \frac{1}{v}, \frac{u}{v}\right),
\end{equation}
and in the local chart $(U_3,\phi_3)$ the expression of $\rho(X)$ is
\begin{equation}\label{Poincare_comp_U3}
\dot{u}= P(u,v), \;
\dot{v}= Q(u,v).
\end{equation}

In the charts $(V_i,\psi_i)$, with $i=1,2,3$, the expression for $\rho(X)$ is the same as in the charts $(U_i,\phi_i)$ multiplied by $(-1)^{d-1}$.

The equator $\mathbb{S}^1$ is invariant by the vector field $\rho(X)$ and all the singular points of $\rho(X)$  which lie in this equator are called the \textit{infinite singular points} of $X$. If $y\in \mathbb{S}^1$ is an infinite singular point, then $-y$ is also an infinite singular point and they have the same (respectively opposite) stability if the degree of vector field is odd (respectively even).

The image of the closed northern hemisphere of $\mathbb{S}^2$ onto the plane $y_3=0$ under the orthogonal projection $\pi$ is called the \textit{Poincar\'e disc} $\mathbb{D}^2$.
Since the orbits of $\rho(X)$  on $\mathbb{S}^2$ are symmetric with respect to the origin of $\mathbb{R}^3$, we only need to consider the flow of $\rho(X)$ in the closed northern hemisphere, and we can project the phase portrait of $\rho(X)$ on the northern hemisphere onto the Poincar\'e disc.  We shall present the phase portraits of the polynomial differential systems \eqref{system} in the Poincar\'e disc.

\subsection{Topological equivalence between two polynomial vector fields}\label{subsec:topologicaleq}

Two polynomial vector fields $X_1$ and $X_2$ on $\mathbb{R}^2$ are \textit{topologically equivalent} if there exists a homeomorphism on the Poincar\'e disc which preserves the infinity $\mathbb{S}^1$ and sends the trajectories of the flow of $\pi(\rho(X_1))$ to the trajectories of the flow of $\pi(\rho(X_2))$, preserving or reversing the orientation of all the orbits.

A \textit{separatrix} of the Poincar\'e compactification $\pi(\rho(X))$ is an orbit at the infinity $\mathbb{S}^1$, or a finite singular point, or a limit cycle, or an orbit on the boundary of a hyperbolic sector at a finite or an infinite singular point. The set of all separatrices of $\pi(\rho(X))$ is closed and we denote it by $\Sigma_X$.

An open connected component of $\mathbb{D}^2 \textbackslash \Sigma_X$ is a \textit{canonical region} of $\pi(\rho(X))$. The \textit{separatrix configuration} of  $\pi(\rho(X))$ is the union of an orbit of each canonical region with the set  $\Sigma_X$, and it is denoted by $\Sigma^{'}_X$. We denote by S (respectively R) the number of separatrices (respectively canonical regions) of a vector field  $\pi(\rho(X))$.

We say that two separatrix configurations $\Sigma^{'}_{X_1}$ and $\Sigma^{'}_{X_2}$ are \textit{topologically equi\-va\-lent} if there is a homeomorphism $h: \mathbb{D}^2 \to \mathbb{D}^2$ such that $h(\Sigma^{'}_{X_1})=\Sigma^{'}_{X_2}$.

The following theorem of Markus \cite{Markus}, Neumann \cite{Neumann} and Peixoto \cite{Peixoto} allows to investigate only the separatrix configuration of a polynomial differential system in order to determine its phase portrait in the Poincar\'e disc.

\begin{theorem}
The phase portraits in the Poincaré disc of two compactified polynomial vector fields $\pi(\rho(X_1))$ and $\pi(\rho(X_2))$ with finitely many separatrices are topologically equivalent if and only if their separatrix configurations $\Sigma^{'}_{X_1}$ and $\Sigma^{'}_{X_2}$ are topologically equivalent.
\end{theorem}

\subsection{Blow-up technique}\label{subsec:blowup}

There exist classification theorems for hyperbolic and semi-hyperbolic singular points, and also for nilpotent singular points which can be found in Chapter 2 and 3 of \cite{Libro}. The centers are more difficult to study, see for instance Chapter 4 of \cite{Libro}. Whereas to study a singular point for which the Jacobian matrix is identically zero, the only possibility is studying each singular point case by case. The main technique to perform the desingularization of a linearly zero singular point is the blow-up technique. We give a short summary about this method, more details can be found in \cite{blowup}.

Roughly speaking the idea behind the blow up technique is to explode, through a change of variables that is not a diffeomorphism, the singularity to a line. Then, for studying the original singular point, one studies the new singular points that appear on this line, and this is simpler. If some of these new singular points are linearly zero, the process is repeated. Dumortier proved that this iterative process of desingularization is finite, see \cite{blowupDumortier}.

Consider a real planar polynomial differential system of the form
\begin{equation}\label{sis_teo_blowup}
\begin{split}
\dot{x}=P(x,y)=P_m(x,y)+ \dots, \\
\dot{y}=Q(x,y)=Q_m(x,y)+ \dots ,
\end{split}
\end{equation}
where $P$ and $Q$ are coprime polynomials, $P_m$ and $Q_m$ are homogeneous polynomials of degree $m\in \mathbb{N}$ and the dots mean higher order terms in $x$ and $y$. Note that we are assuming that the origin is a singular point because $m>0$. We define the \textit{characteristic polynomial} of \eqref{sis_teo_blowup} as
\begin{equation}\label{characteristic_polynomial}
\mathcal{F}(x,y):=xQ_m(x,y)-yP_m(x,y),
\end{equation}
and we say that the origin is a \textit{nondicritical} singular point if $\mathcal{F}\not\equiv0$
and a \textit{dicritical} singular point if $\mathcal{F}\equiv0$. In this last case $P_m=xW_{m-1}$ and $Q_m=yW_{m-1}$, where $W_{m-1}\not\equiv 0$ is a homogeneous polynomial of degree $m-1$. If $y-vx$ is a factor of $W_{m-1}$ and $v=\tan \theta^{*}$, $\theta^{*}\in \left[ 0,2\pi \right)$, then $\theta^{*}$ is a \textit{singular direction}.

The \textit{homogeneous directional blow up in the vertical direction} is the mapping $(x,y) \to (x,z) = (x,y/x)$, where $z$ is a new variable. This map transforms the origin of \eqref{sis_teo_blowup} into the line $x=0$, which is called the \textit{exceptional divisor}. The expression of system \eqref{sis_teo_blowup} after the blow up in the vertical direction is
\begin{equation}\label{sis_teo_blowup2}
\dot{x}=P(x,xz), \;
\dot{z}=\frac{Q(x,xz)-zP(x,xz)}{x},
\end{equation}
that is always well-defined since we are assuming that the origin is a singularity. After the blow up, we cancel an appearing common factor $x^{m-1}$ ($x^m$ if $\mathcal{F}\equiv0$). Moreover, the mapping swaps the second and the third quadrants in the vertical directional blow up.
Propositions 2.1 and 2.2 of \cite{blowup} provide the relationship between the original singular point of system \eqref{sis_teo_blowup} and the new singularities of system  \eqref{sis_teo_blowup2}. For additional details see \cite{blowupAndronov}.

Finally, to study the behavior of the solutions around the origin of system \eqref{sis_teo_blowup}, it is necessary to study the singular points of system \eqref{sis_teo_blowup2} on the exceptional divisor. They correspond to either characteristic directions in the nondicritical case, or singular directions in the dicritical case. It may happen that some of these singular points are linearly zero, in which case we have to repeat the process. As we said before, it is proved in \cite{blowupDumortier} that this chain of blow ups is finite.

\subsection{Indices of planar singular points} \label{subsec:indices}

Given an isolated singularity $q$ of a vector field $X$, defined on an open subset of $\mathbb{R}^2$ or $\mathbb{S}^2$, we define the index of $q$ by means of the Poincaré Index Formula. We assume that $q$ has the finite sectorial decomposition property. Let $e$, $h$  and $p$ denote the number of elliptic, hyperbolic and parabolic sectors of $q$, respectively, and suppose that $e+h+p>0$. Then the \textit{index of $q$} is $i_q=1+(e-h)/2$, and it is always an integer.

We recall that the Poincaré compactification of a vector field in $\mathbb{R}^2$ introduced in Subsection \ref{subsec:Poincare} is a tangent vector field on the sphere $\mathbb{S}^2$, so the next result will be very useful in our study.

\begin{theorem}[Poincaré-Hopf Theorem]\label{th_PoincareHopf}
For every tangent vector field on $\mathbb{S}^2$ with a finite number of singular points, the sum of their indices is $2$.
\end{theorem}

\subsection{Invariants and Application of the Darboux Theory}\label{subsec:Darboux}

The Darboux Theo\-ry of Integrability provides a link between the integrability of polynomial vector fields and the number of invariant algebraic curves that they have. The basic results on dimension two can be found in Chapter 8 of \cite{Libro}, and these results have been extended to $\mathbb{R}^n$ and $\mathbb{C}^n$ in \cite{art_Darboux_1, art_Darboux_2, art_Darboux_3}.

We consider a real polynomial differential system in dimension three, that is a system of the form
\vspace{-0.3cm}
\begin{equation}\label{sis_pol_dim3}
\begin{split}
dx/dt=\dot{x}=P(x,y,z),\\
dy/dt=\dot{y}=Q(x,y,z),\\
dz/dt=\dot{z}=R(x,y,z),\\
\end{split}
\end{equation}
where $P,Q$ and $R$ are polynomials in the variables $x,y$ and $z$. We denote by $m=\max \{\deg P, \deg Q, \deg R\}$ the degree of the polynomial system, and we always assume that the polynomials $P,Q$ and $R$ are relatively prime in the ring of the real polynomials in the variables $x,y$ and $z$.

\begin{theorem}[Darboux Integrability Theorem]\label{th_Darboux}
Suppose that a polynomial system \eqref{sis_pol_dim3} of degree $m$ admits $p$ irreducible invariant algebraic surfaces $f_i=0$ with cofactors $K_i$ for $i=1,\dots,p$. Then the next statements hold.

 \vspace{0.1cm}
 
\begin{enumerate}[(a)]
\item There exist $\lambda_i\in \mathbb{C}$ not all zero such that $\sum_{i=1}^{p}\lambda_iK_i=0$ if and only if the function $f_1^{\lambda_1}\dots f_p^{\lambda_p}$ is a first integral of system \eqref{sis_pol_dim3}.

 \vspace{0.2cm}
		
\item There exist $\lambda_i\in \mathbb{C}$ not all zero such that $\sum_{i=1}^{p}\lambda_iK_i=-s$ for some $s\in \mathbb{R} \textbackslash\left\lbrace 0 \right\rbrace $ if and only if the function $f_1^{\lambda_1}\dots f_p^{\lambda_p} exp(st)$ is a Darboux  invariant of system \eqref{sis_pol_dim3}.
\end{enumerate}
\end{theorem}

\section{Reduction of the Lotka-Volterra systems in $\mathbb{R}^3$ to the Kolmogorov systems in $\mathbb{R}^2$}\label{sec:reduction}

As we said our objective is to study the global dynamics of the Lotka-Volterra systems \eqref{systemLV3} in dimension three, which have a rational first integral of degree two of the form $x^i y^j z^k$. The Darboux theory of integrability allow us to obtain a characterization of these systems.

We consider the irreducible invariant algebraic surfaces
$f_1(x,y,z)=x=0$, $f_2(x,y,z)=y=0$ and $f_3(x,y,z)=z=0$ of system \eqref{systemLV3}, with cofactors $K_1$, $K_2$ and $K_3$, respectively. As $K_i$ is the cofactor of $f_i$ we have that
$$
Xf_i = P \dfrac{\partial f_i}{\partial x} + Q \dfrac{\partial f_i}{\partial y} + R \dfrac{\partial f_i}{\partial z} = K_i f_i.
$$
Then for the invariant algebraic surfaces considered we get the cofactors
$ K_1 =   a_0 + a_1 x + a_2 y + a_3 z$, $ K_2 =  b_0 + b_1 x + b_2 y + b_3 z$ and $ K_3 = c_0 + c_1 x + c_2 y + c_3 z$, respectively.

Applying Theorem \ref{th_Darboux}, since we assume that  $x^{\lambda_1}y^{\lambda_2}z^{\lambda_3}$ is a first integral of system \eqref{systemLV3}, we get that there exist $\lambda_i \in \mathbb{C}$, with $i\in \left\lbrace 1,2,3 \right\rbrace $, not all zero, such that $\sum_{i=1}^{3}\lambda_i K_i = 0.$
Apart from the trivial solution $ \left\lbrace  \lambda_1 = 0,\: \lambda_2 = 0 ,\: \lambda_3 = 0 \right\rbrace$, there are the following three solutions of this equation:

\small
\begin{equation*}
\begin{split}
&S_1 = \left\lbrace    c_0 = 0,\: c_1 = 0 ,\: c_2 = 0,\: c_3 = 0,\: \lambda_2=0,\: \lambda_1=0       \right\rbrace,\\
&S_2 = \left\lbrace   b_0 =-\dfrac{c_0 \lambda_3}{\lambda_2},\: b_1=-\dfrac{c_1\lambda_3}{\lambda_2},\: b_2= - \dfrac{c_2 \lambda_3}{\lambda_2},\: b_3=-\dfrac{c_3\lambda_3}{\lambda_2},\: \lambda_1=0  \right\rbrace, \: \text{\normalsize{and}} \\
&S_3 = \left\lbrace   a_0=\dfrac{-b_0\lambda_2-c_0\lambda_3}{\lambda_1},\: a_1=\dfrac{-b_1\lambda_2-c_1\lambda_3}{\lambda_1},\: a_2=\dfrac{-b_2\lambda_2-c_2\lambda_3}{\lambda_1},\: a_3=\dfrac{-b_3\lambda_2-c_3\lambda_3}{\lambda_1}      \right\rbrace,\\
\end{split}
\end{equation*}
\normalsize
which give rise to three families of Lotka-Volterra polynomial differential systems of degree two in $\mathbb{R}^3$, with a first integral of the form $x^{\lambda_1}y^{\lambda_2}z^{\lambda_3}$. %Now we will characterise each one of these families.

If we consider the family given by solution $S_1$, as the parameters $c_i, \; i=0,...,3$, are zero, we have that $\dot{z}=0$ and the Lotka-Volterra system is reduced to:
\begin{equation*}
\begin{split}
& \dot{x} = x \: ( \:  a_0 + a_1 x + a_2 y + a_3 z \: ),\\
& \dot{y} = y \: ( \: b_0 + b_1 x + b_2 y + b_3 z \: ), \\
& \dot{z} = 0.
\end{split}
\end{equation*}
As $\dot{z}=0$, $z$ is constant and this system has $H=z$ as a first integral. Note that if we consider the first integral $H=x^{\lambda_1}y^{\lambda_2}z^{\lambda_3}$,
and apply the conditions given by $S_1$, it is $\lambda_1=\lambda_2=0$, we obtain $H=z^{\lambda_3}$, with $\lambda_3=2$ for getting the degree two, but in this case we will consider the simplest first integral. In each invariant plane with $z$ constant, we have a Lotka-Volterra polinomial differential system in $\mathbb{R}^2$. The phase portrait of these systems has been studied in \cite{LVdim2}, so we are not going to deal with this case.

In this paper we study the family given by the solution $S_2$.
This solution provides the values of parameters $b_i$ as a function of the parameters $\lambda_2$, $\lambda_3$ and $c_i$, with  $i=0,..,3$, so we can replace them in the expression of $\dot{y}$ obtaining
$$
\dot{y}= y \: \left(  \: -\dfrac{c_0 \lambda_3}{\lambda_2} -\dfrac{c_1 \lambda_3}{\lambda_2} x -\dfrac{c_2 \lambda_3} {\lambda_2} y -\dfrac{c_3 \lambda_3}{\lambda_2} z \: \right).  $$
If we denote $\lambda= - \lambda_3/\lambda_2$, then the original Lotka-Volterra system becomes
\begin{equation*}
\begin{split}
& \dot{x} = x  (  a_0 + a_1 x + a_2 y + a_3 z  ),\\
&\dot{y} = \lambda  y  (  c_0 + c_1 x + c_2 y + c_3 z  ), \\
& \dot{z} = z (  c_0 + c_1 x + c_2 y + c_3 z  ).
\end{split}
\end{equation*}
Given that $\lambda_1 = 0$, the first integral $H=x^{\lambda_1}y^{\lambda_2}z^{\lambda_3}$  is reduced to $H=y^{\lambda_2}z^{\lambda_3}$, but if this is a first integral, also
$H=\left( y^{\lambda_2}z^{\lambda_3}\right)^{-\frac{1}{\lambda_2}} = y^{-1} z^{-\frac{\lambda_3}{\lambda_2}}= y^{-1} z^\lambda = z^\lambda/y$
is a first integral.
If we want $H$ to be rational of degree two, we must take $\lambda = 2 $. In each level $H=1/h$, with $h \neq 0$, we will have
$ 1/h = z^2/y$, so $y = h z^2$
and then, for each $h$, the initial Lotka-Volterra system on dimension three reduces to the system on dimension two
\begin{equation*}
\begin{split}
& \dot{x} = x  (   a_0 + a_1 x + a_2 h  z^2 + a_3 z  ),\\
& \dot{z} = z  (  c_0 + c_1 x + c_2 h z^2 + c_3 z  ).
\end{split}
\end{equation*}

We must study the phase portrait of the systems of this family, but it is equivalent to study the phase portraits of the family of Kolmogorov systems in dimension two
\begin{equation}\label{systemKolmogorov1}
\begin{split}
& \dot{x} = x  (  a_0 + a_1 x + a_2 z^2 + a_3 z  ),\\
& \dot{z}= z  (  c_0 + c_1 x + c_2 z^2 + c_3 z  ).
\end{split}
\end{equation}

In the particular cases in which $H$ is zero or infinity, the differential system on dimension three is reduced to a Lotka-Volterra system on dimension two, having in each case  $z=0$ and $y=0$, respectively. We recall that these systems had already been studied in \cite{LVdim2}.

Systems \eqref{systemKolmogorov1} depend on eight parameters and the classification of all their distinct topological phase protraits is huge. For this reason we study the subclass of them having a Darboux invariant of the form $x^{\lambda_1} z^{\lambda_2} e^{st}$. By statement (b) of Theorem \ref{th_Darboux} the expression $\lambda_1 K_x+\lambda_2 K_z + s$ must be zero, where $K_x$ and $K_y$ are the cofactors of the invariant planes $x=0$ and $z=0$, respectively. Note that $s$ and $\lambda_1^2+ \lambda_2^2$ cannot be zero. We obtain the cofactors $K_x=a_0+a_1x+a_2 z^2 + a_3 z$ and $K_z=c_0+c_1 x + c_2 z^2 + c_3 z$ and then, solving the equation  $\lambda_1 K_x+\lambda_2 K_z + s =0$, we get the following two non-trivial solutions
\begin{equation*}
\begin{split}
&\tilde{S}_1 = \left\lbrace s=-a_0\lambda_1 - c_0 \lambda_2, \: a_1=-\frac{c_1\lambda_2}{\lambda_1}, \: a_2=-\dfrac{c_2\lambda_2}{\lambda_1}, \: a_3= - \dfrac{c_3\lambda_2}{\lambda_1}\right\rbrace,  \: \text{and}\\
&\tilde{S}_2 = \left\lbrace s=-c_0\lambda_2, \: c_1=0, \: c_2=0, \:  c_3=0, \:  \lambda_1=0    \right\rbrace.
\end{split}
\end{equation*}
So we have two subsystems from the initial system \eqref{systemKolmogorov1}. According to the conditions given by solution $\tilde{S_1}$ the first subsystem is
\begin{equation*}
\begin{split}
\dot{x}&=x\left( a_0 - \frac{c_1 \lambda_2}{\lambda_1}x - \frac{c_2\lambda_2}{\lambda_1} z^2- \frac{c_3 \lambda_2}{\lambda_1}z\right) , \\
\dot{z}&=z\left( c_0+c_1 x + c_2 z^2 + c_3 z \right).
\end{split}
\end{equation*}
If we denote $\lambda_2/\lambda_1=\mu$ and $\lambda_1=\lambda$, then this subsystem becomes
\begin{equation}\label{ss}
\begin{split}
\dot{x}&=x \left( a_0- \mu (c_1 x + c_2 z^2 + c_3 z)\right),\\
\dot{z}&=z\left( c_0+ c_1 x + c_2 z^2 + c_3 z\right),
\end{split}
\end{equation}
and its Darboux invariant is $x^\lambda z^{\lambda\mu} e^{-t\lambda(a_0+c_0\mu)}$. But if this is a Daboux invariant, also it is $x z^{\mu}e^{-t(a_0+c_0\mu)}$. Note that in order that we have a Darboux invariant $a_0+c_0\mu$ cannot be zero.

If we consider now the solution $\tilde{S_2}$ we get the subsystem
\begin{equation*}
\begin{split}
\dot{x}&= x \left( a_0 + a_1 x + a_2 z^2 + a_3 z \right),\\
\dot{z}&= c_0 z,
\end{split}
\end{equation*}
which is equivalent to the previous one, taking $\mu=0$ and interchanging the variables $x$ and $z$, so it is sufficient to study the Kolmogorov systems \eqref{ss} depending on six parameters.

\section{Properties of system \eqref{system}} \label{sec:properties}

In this section we state some results that will be used on the classification in order to reduce the number of phase portraits appearing. Note that if $c_2=0$, then the system \eqref{system} is a Lotka-Volterra system in dimension 2. A global topological classification of these systems has been completed in \cite{LVdim2}, so we limit our study to the case $c_2\neq0$.

We recall that for obtaining system \eqref{system} we have supposed that system \eqref{systemKolmogorov1} has the Darboux invariant $I=xz^\mu e^{-t(a_0+c_0\mu)}$, so it is required that $a_0+c_0\mu \neq0$.

\begin{proposition}\label{th_c1c3}
Consider system \eqref{system} and suppose that $(\tilde{x}(t), \tilde{z}(t))$ is a solution of this system. If we change $c_1$ by $-c_1$ (respectively $c_3$ by $-c_3$ ), then  $(-\tilde{x}(t), \tilde{z}(t))$ (respectively $(\tilde{x}(t),-\tilde{z}(t))$) is other solution of the obtained system.
\end{proposition}
	
\begin{remark}
By Proposition \ref{th_c1c3} we can limit our study to Kolmogorov systems \eqref{system} with $c_1$ and $c_3$ non-negatives. In the cases with these parameters negatives, we will obtain phase portraits symmetric to the ones obtained in the positive cases, with respect to the $z$-axis when we change the sign of $c_1$, and with respect to the $x$-axis when we change the sign of $c_3$.
\end{remark}

\begin{corollary}\label{th_simmetry}
Consider system \eqref{system} and suppose $(\tilde{x}(t), \tilde{z}(t))$ is a solution. If $c_1=0$ (respectively $c_3=0$),  then $(-\tilde{x}(t),\tilde{z}(t))$ (respectively $(\tilde{x}(t), -\tilde{z}(t))$) is also a solution.
\end{corollary}

\begin{remark}
Corollary \ref{th_simmetry} simplifies the study of the cases with $c_1=0$ or $c_3=0$, because it proves that the phase portraits have to be symmetric with respect to the $z$-axis and $x$-axis respectively, and this fact will be useful in obtaining the global phase portraits from the local results.
\end{remark}

\begin{proposition}
Let $(\tilde{x}(t),\tilde{z}(t))$ be a solution of system \eqref{system}. In the next cases we obtain another system with solution  $(-\tilde{x}(-t),-\tilde{z}(-t))$.
\begin{enumerate}
\item If $a_0$, $c_0$ and $c_2$ are not zero, and we change the sign of all of them.

\item If $a_0=0$ and we change the sign of $c_0$ and $c_2$, which are not zero.

\item If $c_0=0$ and we change the sign of $a_0$ and $c_2$, which are not zero.
\end{enumerate}
\end{proposition}

\begin{remark}\label{remark_a0c0c2}
In order to classify all the phase portraits of the Kolmogorov systems \eqref{system}, according with the previous results, it is sufficient to consider $a_0\geq 0$. And when $a_0=0$ we will consider also $c_0>0$.
\end{remark}

\begin{remark}
In short, according with the previous results and considerations, from now on it will be sufficient to study the Kolmogorov systems \eqref{system} with the their parameters satisfying
\begin{align*}
(H)= \left\lbrace c_2\neq 0, a_0\geq 0, c_1\geq 0, c_3\geq 0, a_0+c_0 \mu\neq 0 \right\rbrace.
\end{align*}
\end{remark}

\begin{theorem}\label{th_contactpoints}
For system \eqref{system} the next statements hold.
\begin{enumerate}
\item  If $c_1\neq0$, then on any straight line $z=cte\neq0$, there exists only one contact point.

\item If $c_1=0$, then there exist two invariant straight lines $z=(\sqrt{c_3^2-4c_0c_2}-c_3)/(2c_2)$ and $z=-(\sqrt{c_3^2-4c_0c_2}+c_3)/(2c_2)$ if $c_3^2>4c_0c_2$, and one invariant straight line $z=-c_3/(2c_2)$ if $c_3^2=4c_0c_2$. There are not contact points on any other straight line $z=cte\neq0$.
\end{enumerate}
\end{theorem}

\begin{proof} First we suppose $c_1\neq0$ and consider a straight line $z=z_0\neq0$. Then the contact points on this straight line are those on which $\dot{z}=0$ and, as $z_0\neq0$, the only possible contact point is the one that satisfies $c_0+c_1x+c_2 z_0^2+c_3z_0=0$, i.e. the point such that its first coordinate is $x=-(c_2z_0^2+c_3z_0+c_0)/c_1$.
	
We consider now the case with $c_1=0$. Then looking for the points on the straight line $z=z_0\neq0$ satisfying $\dot{z}=0$, we obtain that they must verify the condition $c_0+c_2z_0^2+c_3z_0=0$, and solving this equation we get that either there are no contact points, or a full straight line of contact points , or two straight line of contact points, depending on the solutions $z_0$ of that equation.
\end{proof}

\section{Local study of finite singular points}\label{sec:finite}

System \eqref{system} has the following finite singularities:
		
\begin{enumerate}[$\bullet$]
	\vspace{0.1cm}
\item $P_0 = (0,0)$,
	\vspace{0.2cm}		
\item $P_1= \left(0, \dfrac{R_c - c_3}{2c_2} \right) $ and $P_2 = \left( 0, -\dfrac{ R_c + c_3}{2c_2} \right) $ if $c_3^2>4 c_0 c_2$,
		\vspace{0.2cm}	
\item  $P_3 = \left( 0,-\dfrac{c_3}{2c_2}\right) $ if $c_3^2=4 c_0 c_2$,
	\vspace{0.2cm}		
\item $P_4=\left( \dfrac{a_0}{c_1 \mu},0\right) $ if $ c_1 \mu \neq 0$.		
\end{enumerate}	
		\vspace{0.2cm}
We use the notation $R_c=\sqrt{c_3^2-4c_0c_2}$ in order to simplify the expressions which will appear. Moreover if $a_0=0$ and $c_1\mu=0$, all the points on the $z$-axis are singular points, and the system can be reduced to a Lotka-Volterra system in dimension 2. Therefore from now on we will consider the hypothesis
\begin{align*}
(H_1)= \left\lbrace c_2\neq 0, a_0\geq 0, c_1\geq 0, c_3\geq 0, a_0+c_0\mu\neq 0, a_0^2 + (c_1\mu)^2\neq0 \right\rbrace .
\end{align*}

Assuming $(H_1)$ there are $6$ different cases according to the finite singular points existing for system \eqref{system}, which are given in Table \ref{tab:cases1-6}. Then we study the possible local phase portraits in each one of the finite singular points under the hypothesis $(H_1)$.

\begin{table}[H]
\begin{center}
\begin{tabular}{|cll|}
\hline
\textbf{Case} & \textbf{Conditions} & \textbf{Finite singular points} \\
\hline
\hline
1&  $c_3^2>4c_0c_2$, $c_1\mu\neq0$. & $P_0$, $P_1$, $P_2$, $P_4$. \\	\hline
2 &  $c_3^2>4c_0c_2$, $c_1\mu=0$, $a_0\neq0$. & $P_0$, $P_1$, $P_2$. \\
\hline
3 &  $c_3^2=4c_0c_2$, $c_1\mu\neq0$. & $P_0$, $P_3$, $P_4$. \\
\hline
4 &  $c_3^2=4c_0c_2$, $c_1\mu=0$, $a_0\neq0$. & $P_0$, $P_3$.  \\
\hline
5 &  $c_3^2<4c_0c_2$, $c_1\mu\neq0$. & $P_0$, $P_4$. \\
\hline
6 &  $c_3^2<4c_0c_2$, $c_1\mu=0$, $a_0\neq0$. & $P_0$. \\
\hline
\end{tabular}
\caption{The different cases for the finite singular points.} \label{tab:cases1-6}
\end{center}
\end{table}

\vspace{-0.5cm}

The origin is always an isolated singular point  for system \eqref{system}, and we have the next classification for its phase portraits: if $a_0 c_0\neq0$ the singularity is hyperbolic and two cases are possible, the origin is a saddle point if $c_0 <0$, and it is an unstable node if $c_0 >0$. If $a_0 \neq 0$ and $c_0=0$ the singularity is semi-hyperbolic and it has two possibilities: if $c_3\neq0$ then the origin is a saddle-node, if $c_3=0$ and $c_2 <0$ it is a topological saddle, and if $c_3=0$ and $c_2 >0$ it is a topological unstable node. Finally if $a_0=0$ the origin is a semi-hyperbolic saddle-node.
	
When $P_1$ is a singular point of system \eqref{system}, it can present different phase portraits. If $c_0\neq0$ then $P_1$ is hyperbolic and it can present the following phase portraits: if $c_2(a_0+c_0\mu) (R_c-c_3)<0$ then $P_1$ is a saddle, if $a_0+c_0\mu<0 $ and $c_2(R_c-c_3)<0 $ it is a stable node, and finally if $a_0+c_0 \mu>0 $ and $c_2(R_c-c_3)>0$ it is an unstable node. The singular point $P_1$ collides with the origin if $c_1=0$.

When $P_2$ is a singular point of system \eqref{system}, it can present three different phase portraits: if $c_2(a_0+c_0\mu)<0$ then $P_2$ is a saddle, if $a_0+c_0\mu<0$ and $c_2 <0 $ then it is a stable node, and if $a_0+c_0\mu>0$ and $c_2 >0 $ it is an unstable node.

When $P_3$ is a singularity of system \eqref{system} it is  a semi-hyperbolic saddle-node if $c_3 \neq 0$, and it collides with the origin if $c_3 = 0$.

When $P_4$ is a singularity of system \eqref{system} it is hyperbolic if $a_0\neq0$ and can present two different phase portraits: if $(a_0+ c_0\mu) \mu > 0$ then $P_4$ is a saddle, and if $\mu(a_0+ c_0\mu) <0$ it is an stable node. If $a_0=0$ the singularity $P_4$ collides with the origin.
	
\begin{lemma}
Asumming hypothesis $(H_1)$ there are $50$ different cases according to the local phase portrait of the finite singular points of system \eqref{system}, which are given in Tables \ref{tab:cases_fin_local_1} - \ref{tab:cases_fin_local_6}.
\end{lemma}

\begin{proof}	
We have to analyse cases 1 to 6 in Table \ref{tab:cases1-6} and determine the local phase portraits of the singular points existing in each one of them, according to their individual classification. We start with the first one, in which the conditions, $c_3^2>4c_0c_2$ and $c_1\mu\neq0$ hold. The singular points are $P_0$, $P_1$, $P_2$ and $P_4$. We shall consider three subcases: $a_0=0$, $c_0=0$ and $a_0c_0\neq0$.
		
Consider case $c_0=0$ in which the origin is a saddle-node and $P_1$ collides with the origin. Since $c_0=0$ and $a_0>0$, the singular point $P_2$ is a saddle if $c_2<0$, and an unstable node if $c_2>0$. In these two cases $P_4$ can be either a saddle if $\mu>0$, or a stable node if $\mu<0$. This leads to cases 1.1 to 1.4 in Table \ref{tab:cases_fin_local_1}.
				
We continue with the case $a_0=0$ in which $P_0$ is again a saddle-node, but in this case it coincides with $P_4$. Suppose that $P_1$ is an unstable node, then we have $c_0\mu>0$ and $c_2(R_c-c_3)>0$. By Remark \ref{remark_a0c0c2} we will only consider the case $c_0>0$. Then if $c_2>0$ and $R_c-c_3>0$ taking into account the expression of $R_c$ and squaring both terms, we get that $c_3^2-4c_0c_2>c_3^2$, so $c_0c_2<0$, which leads to a contradiction. The same occurs if we suppose $c_2<0$. Therefore $P_1$ cannot be an unstable node. If $P_1$ is a saddle, then $P_2$ can be a saddle or an unstable node, but not a stable node, which is only possible if $c_0<0$, by an analogous reasoning to the previous one. If $P_1$ is a stable node then $c_0\mu<0$, so $P_2$ can be a saddle or stable node, but not an unstable node because it requires that $c_0\mu>0$. This leads to cases 1.5 to 1.8.
			
The last case is $a_0c_0\neq0$ in which the origin is a hyperbolic singular point. We start when $P_0$ is a saddle, then $c_0<0$. First we consider that $P_1$ is also a saddle, and so $c_2(a_0+c_0\mu)(R_c-c_3)<0$. If $P_2$ is a saddle then $c_2(a_0+c_0\mu)<0$, and we get $R_c-c_3>0$. From this we deduce like in previous cases that $c_0c_2<0$, but we are supposing $c_0<0$ and so $c_2>0$ and $a_0+c_0\mu<0$.  From the last inequality $a_0<-c_0\mu$, and so $\mu$ has to be positive. In short $\mu(a_0+c_0\mu)<0$, and consequently $P_4$ can only be a stable node. If $P_0$ and $P_1$ are saddles, but $P_2$ is a stable node, reasoning in an analogous way we get that $P_4$ is again a stable node. This leads to cases 1.9 and 1.10. Note that if $P_0$ and $P_1$ are saddles it is impossible for $P_2$ to be an unstable node. In that case we would have that $c_0<0$, $c_2>0$, $a_0+c_0\mu>0$ and $R_c-c_3<0$. From this last inequality we get that $c_0c_2>0$, which is a contradiction. We consider now the cases where $P_0$ is a saddle and $P_1$ an unstable node, in which the conditions $c_0<0$, $a_0+c_0\mu>0$ and $c_2(R_c-c_3)>0$ hold. It is obvious that $P_2$ cannot be a stable node because it requires that $a_0+c_0\mu<0$, so $P_2$ is a saddle if $c_2<0$ and an unstable node if $c_2>0$. In both cases $P_4$ can be either a saddle if $\mu>0$, or a stable node if $\mu<0$. This leads to cases 1.11 to 1.14. Note that the case with $P_0$ a saddle and $P_1$ a stable node is not possible, because we would have $c_0<0$, $a_0+c_0\mu<0$ and $c_2(R_c-c_3)<0$. If $c_2>0$ then $R_c-c_3<0$, whence we deduce $c_0c_2>0$ and get a contradiction. The same argument is valid if $c_2<0$. With analogous reasoning as in the case where $P_0$ is an unstable node we get the subcases 1.15 to 1.20.
				
Now we study case 2 of Table \ref{tab:cases1-6}, in which $c_3^2>4c_0c_2$, $c_1\mu=0$ and $a_0\neq0$. We shall consider three cases: $c_0<0$, $c_0>0$ and $c_0=0$.
				
We start with case $c_0<0$ in which $P_0$ is a saddle. If $P_1$ is a saddle, then $P_2$ can be a saddle or a stable node. If $P_2$ is an unstable node, then we have the conditions $c_2(a_0+c_0\mu) (R_c-c_3)<0$, $c_2>0$ and $a_0+c_0\mu>0$, so $R_c-c_3<0$, and we deduce $c_0c_2>0$ which is a contradiction. $P_1$ cannot be a stable node, because in that case we would have the conditions $c_0<0$, $a_0+c_0\mu<0$ and $c_2(R_c-c_3)<0$ which lead to a contradiction in the following way: if $c_2>0$ then $R_c-c_3<0$ and squaring we deduce $c_0c_2>0$ which is not possible because $c_0<0$ and we are supposing $c_2>0$. An analogous reasoning works in the case $c_2<0$. If $P_1$ is an unstable node then $P_2$ can be either a saddle or an unstable node, but not a stable node because it requires $a_0+c_0\mu$ to be negative, but we already know that this expression is positive because it is a condition in order that $P_1$ be an unstable node. This leads to cases 2.1 to 2.4 of Table \ref{tab:cases_fin_local_2}. 		
		
We continue with the case $c_0>0$, in which $P_0$ is an unstable node. If $P_1$  is a saddle, then $P_2$ can be a saddle or an unstable node. If $P_2$ is a stable node then we have the conditions $c_0>0$ , $c_2(a_0+c_0\mu)(R_c-c_3)<0$, $a_0+c_0\mu<0$ and $c_2<0$. Thus we have $R_c-c_3<0$ and squaring we obtain $c_0c_2>0$ which is a contradiction. This leads to cases 2.5 and 2.6. If $P_1$ is a stable node it can be proved similarly to previous cases, that $P_2$ cannot be an unstable node. This leads to cases 2.7 and 2.8. $P_1$ cannot be an unstable node, because in that case we would have the conditions $c_0>0$, $a_0+c_0\mu>0$ and $c_2(R_c-c_3)>0$ which lead to a contradiction in the following way: if $c_2>0$ then $R_c-c_3>0$, and squaring we deduce $c_0c_2<0$ which is not possible because $c_0>0$ and we are supposing $c_2>0$. An analogous reasoning works in the case $c_2<0$. 		
		
Al last we have the case $c_0=0$. Necessarily $c_3\neq0$ so the origin is a saddle-node. Also we have that $P_1$ coincides with the origin. For the singular point $P_2$ we have that it is a saddle if $c_2<0$, and an unstable node if $c_2>0$. This leads to cases 2.9 and 2.10.
		
We study case 3 of Table \ref{tab:cases1-6} in which $c_3^2=4c_0c_2$ and $c_1\mu\neq0$. Then $c_0=0$ if and only if $c_3=0$. We consider $a_0>0$ and $c_0<0$, then the origin is a saddle and $P_3$ a saddle-node (as $c_3\neq0$). The singular point $P_4$ is either a saddle or a stable node, depending on the sign of $\mu(a_0+c_0\mu)$. The same is valid in the case $a_0>0$ and $c_0<0$, except for the origin which is now an unstable node. We get the cases 3.1 to 3.4 of Table \ref{tab:cases_fin_local_3}. We continue with the case in which $a_0=0$ and so $P_0$ is a saddle-node, $P_4$ coincides with $P_0$ and $P_3$ is a saddle-node. This correspond with case 3.5. At last we have the condition $c_0=0$, under which $P_3$ coincides with $P_0$. If $c_2<0$ then it is a topological saddle, and if $c_2>0$ it is a topological unstable node. In any case $P_4$ can be either a saddle or a stable node. This leads to cases 3.6 to 3.9.
		
Now we address the case 4 of Table \ref{tab:cases1-6} in which $c_3^2=4c_0c_2$, $c_1\mu=0$ and $a_0\neq0$. The origin is a saddle if $c_0<0$, and an unstable node if $c_0>0$. If $c_0=0$ then $c_3=0$, so we distinguish two semi-hyperbolic possibilities for the origin: if $c_2<0$ it is a topological saddle, and if $c_2>0$ it is a topological unstable node. The classification of $P_3$ is totally determined by the one of $P_0$, because it only depends on whether $c_3$ is zero or not. We get cases 4.1 to 4.4.
		
In case 5 of Table \ref{tab:cases1-6} the conditions $c_3^2<4c_0c_2$ and $c_1\mu\neq0$ hold. The singular points are $P_0$ and $P_4$. From condition $c_3^2<4c_0c_2$ we get that $c_0\neq0$. If $a_0=0$ then the origin is a saddle-node and $P_4$ coincides with the origin. If $a_0\neq0$ then both singular points are hyperbolic, and it leads to cases 5.2 to 5.5.
		
Finally in case 6 of Table \ref{tab:cases1-6} we have the conditions  $c_3^2<4c_0c_2$, $c_1\mu=0$ and $a_0\neq0$. The unique singular point is the origin and as $c_0$ cannot be zero, it is either a saddle or an unstable node.
\end{proof}

\begin{table}
\begin{center}
\begin{tabular}{|cp{5.5cm}p{5cm}|}
\multicolumn{3}{l}{ \textbf{Case 1: } $\: \boldsymbol{c_3^2>4c_0c_2}$, $\boldsymbol{c_1\mu\neq0}$.}\\
\hline
\hline
\small\textbf{Sub.}& \small\textbf{Conditions} & \small\textbf{Classification} \\
\hline
\hline

1.1
&
\small $a_0>0$, $c_0=0$, $\mu>0$, $c_2<0$.
&
\small $P_0\equiv P_1$  saddle-node, $P_2$ saddle, \newline $P_4$ saddle.\\
\hline
			
1.2
&
\small $a_0>0$, $c_0=0$, $\mu>0$,  $c_2>0$.
&
\small $P_0\equiv P_1$  saddle-node, \newline  $P_2$ unstable node, $P_4$ saddle.\\
\hline
			
1.3
&
\small $a_0>0$, $c_0=0$, $\mu<0$,  $c_2<0$.
&
\small $P_0\equiv P_1$  saddle-node, $P_2$ saddle, \newline  $P_4$ stable node.\\
\hline

1.4
&
\small $a_0>0$, $c_0=0$, $\mu<0$,  $c_2>0$.
&
\small $P_0\equiv P_1$  saddle-node, \newline  $P_2$ unstable node, $P_4$ stable node.\\
\hline
			
1.5
&
\small $a_0=0$, $c_0>0$, $c_2\mu<0$, $R_c-c_3>0$.
&
\small $P_0\equiv P_4$  saddle-node, $P_1$ saddle, \newline $P_2$ saddle.\\
\hline

1.6
&
\small $a_0=0$, $c_0>0$, $R_c-c_3<0$, $\mu>0$, $c_2>0$.
&
\small $P_0\equiv P_4$  saddle-node, $P_1$ saddle, \newline $P_2$ unstable node.\\
\hline
			
1.7
&
\small $a_0=0$, $c_0>0$, $\mu<0$, $R_c-c_3<0$, $c_2>0$.
&
\small $P_0\equiv P_4$  saddle-node, \newline  $P_1$ stable node, $P_2$ saddle.\\
\hline
			
1.8
&
\small $a_0=0$, $c_0>0$, $\mu<0$, $c_2<0$, \newline $R_c-c_3>0$.
&
\small $P_0\equiv P_4$  saddle-node, \newline  $P_1$ stable node, $P_2$ stable node.\\
\hline
			
1.9
&
\small $a_0>0$, $c_0<0$, $\mu>0$, $(a_0+c_0\mu)<0$, $c_2>0$, $R_c-c_3>0$.
&
\small $P_0$ saddle, $P_1$ saddle, $P_2$ saddle, \newline  $P_4$ stable node.\\
\hline
			
1.10
&
\small $a_0>0$, $c_0<0$, $c_2<0$, $\mu>0$, \newline $a_0+c_0\mu<0$, $R_c-c_3<0$.
&
\small $P_0$ saddle, $P_1$ saddle, \newline  $P_2$ stable node, $P_4$ stable node.\\
\hline
				
1.11
&
\small $a_0>0$, $c_0<0$, $c_2<0$, $\mu>0$, \newline $a_0+c_0\mu>0$, $(R_c-c_3)<0$.
&
\small $P_0$ saddle, $P_1$ unstable node, \newline  $P_2$ saddle, $P_4$ saddle.\\
\hline
			
1.12
&
\small $a_0>0$, $c_0<0$, $c_2<0$, $\mu<0$, \newline $a_0+c_0\mu>0$, $(R_c-c_3)<0$.
&
\small $P_0$ saddle, $P_1$ unstable node, \newline  $P_2$ saddle, $P_4$ stable node.\\
\hline
			
1.13
&
\small $a_0>0$, $c_0<0$, $\mu>0$, $a_0+c_0\mu>0$, $c_2>0$, $R_c-c_3>0$.
&
\small $P_0$ saddle, $P_1$ unstable node, \newline  $P_2$ unstable node, $P_4$ saddle.\\
\hline
			
1.14
&
\small $a_0>0$, $c_0<0$, $\mu<0$, $a_0+c_0\mu>0$, $c_2>0$, $R_c-c_3>0$.
&
\small $P_0$ saddle, $P_1$ unstable node, \newline  $P_2$ unstable node, $P_4$ stable node.\\
\hline
			
1.15
&
\small $a_0>0$, $c_0>0$, $\mu(a_0+c_0\mu)>0$, \newline $c_2(a_0+c_0\mu)<0$, $R_c-c_3>0$.
&
\small $P_0$ unstable node, $P_1$ saddle, \newline $P_2$ saddle, $P_4$ saddle.\\
\hline

1.16
&
\small $a_0>0$, $c_0>0$, $\mu(a_0+c_0\mu)<0$, \newline $c_2(a_0+c_0\mu)<0$, $R_c-c_3>0$.
&
\small $P_0$ unstable node, $P_1$ saddle, \newline  $P_2$ saddle, $P_4$ stable node.\\
\hline
			
1.17
&
\small $a_0>0$, $c_0>0$, $\mu>0$, $R_c-c_3<0$, $a_0+c_0\mu>0$, $c_2>0$.
&
\small $P_0$ unstable node, $P_1$ saddle, \newline $P_2$ unstable node, $P_4$ saddle.\\
\hline
		
1.18
&
\small $a_0>0$, $c_0>0$, $\mu<0$, $R_c-c_3<0$, $a_0+c_0\mu>0$, $c_2>0$.
&
\small $P_0$ unstable node, $P_1$ saddle, \newline $P_2$ unstable node, $P_4$ stable node.\\
\hline
			
1.19
&
\small $a_0>0$, $c_0>0$, $\mu<0$, $a_0+c_0\mu<0$, $R_c-c_3<0$, $c_2>0$.
&
\small $P_0$ unstable node, $P_1$ stable node, $P_2$ saddle, $P_4$ saddle.\\
\hline
		
1.20
&
\small $a_0>0$, $c_0>0$, $\mu<0$, $a_0+c_0\mu<0$, $c_2<0$, $R_c-c_3>0$.
&
\small $P_0$ unstable node, $P_1$ stable node, $P_2$ stable node, $P_4$ saddle.\\
\hline

\end{tabular}
\caption{Classification in case 1 of Table \ref{tab:cases1-6} according with the local phase portraits of  finite singular points.}
\label{tab:cases_fin_local_1}
\end{center}
\end{table}	

\begin{table}
\begin{center}
\begin{tabular}{|c|p{5.5cm}p{5cm}|}
\multicolumn{3}{l}{\textbf{Case 2: } $\: \boldsymbol{c_3^2>4c_0c_2}$, $\boldsymbol{c_1\mu=0}$, $\boldsymbol{a_0>0}$.}\\
\hline
\hline
\small\textbf{Sub.}& \small\textbf{Conditions} & \small\textbf{Classification} \\
\hline
\hline

2.1
&
\small $c_0<0$, $c_2(a_0+c_0\mu)<0$, $R_c-c_3>0$.
&
\small $P_0$ saddle, $P_1$ saddle, $P_2$ saddle.\\
\hline

2.2
&
\small  $c_0<0$, $R_c-c_3<0$, $a_0+c_0\mu<0$, $c_2<0$.
&
\small $P_0$ saddle, $P_1$ saddle, \newline $P_2$ stable node.\\
\hline
			
2.3
&
\small $c_0<0$,  $a_0+c_0\mu>0$, $R_c-c_3<0$, $c_2<0$.
&
\small $P_0$ saddle, $P_1$ unstable node, \newline $P_2$ saddle.\\
\hline
			
2.4
&
\small $c_0<0$, $a_0+c_0\mu>0$, $c_2>0$,\newline $R_c-c_3>0$.
&
\small $P_0$ saddle, $P_1$ unstable node, \newline$P_2$ unstable node.\\
\hline
			
2.5
&
\small $c_0>0$, $c_2(a_0+c_0\mu)<0$, $R_c-c_3>0$.
&
\small $P_0$ unstable node, $P_1$ saddle, \newline $P_2$ saddle.\\
\hline
			
2.6
&
\small  $c_0>0$, $R_c-c_3<0$, $a_0+c_0\mu>0$, $c_2>0$.
&
\small $P_0$ unstable node, $P_1$ saddle, \newline $P_2$ unstable node.\\
\hline
	
2.7
&
\small $c_0>0$, $a_0+c_0\mu<0$, $R_c-c_3<0$, $c_2>0$.
&
\small $P_0$ unstable node, $P_1$ stable node, $P_2$ saddle.\\
\hline
			
2.8
&
\small $c_0>0$, $a_0+c_0\mu<0$, $c_2<0$, \newline $R_c-c_3>0$.
&
\small $P_0$ unstable node, $P_1$ stable node, $P_2$ stable node.\\
\hline
			
2.9
&
\small $c_0=0$, $a_0>0$, $c_2<0$.
&
\small $P_0\equiv P_1$ saddle-node, $P_2$ saddle.\\
\hline
			
2.10
&
\small $c_0=0$, $a_0>0$, $c_2>0$.
&
\small $P_0\equiv P_1$ saddle-node, \newline $P_2$ unstable node.\\
\hline
			
\end{tabular}
\caption{Classification in case 2 of Table \ref{tab:cases1-6} according with the local phase portraits of finite singular points.}
\label{tab:cases_fin_local_2}
\end{center}
\end{table}	
		
\vspace{-1cm}	

\begin{table}
\begin{center}
\begin{tabular}{|c|p{5.5cm}p{5cm}|}
\multicolumn{3}{l}{ \textbf{Case 3: } $\: \boldsymbol{c_3^2=4c_0c_2}$, $\boldsymbol{c_1\mu\neq0}$.}\\
\hline
\hline
\small\textbf{Sub.}& \small\textbf{Conditions} & \small\textbf{Classification} \\
\hline
\hline
		
3.1
&
\small $a_0>0$, $c_0<0$,  $\mu(a_0+c_0\mu)>0$.
&
\small $P_0$ saddle, $P_3$ saddle-node, \newline$P_4$ saddle.\\
\hline
	
3.2
&
\small $a_0>0$, $c_0<0$,  $\mu(a_0+c_0\mu)<0$.
&
\small $P_0$ saddle, $P_3$ saddle-node, \newline $P_4$ stable node.  \\
\hline
		
3.3
&
\small $a_0>0$, $c_0>0$,  $\mu(a_0+c_0\mu)>0$.
&
\small $P_0$ unstable node, $P_3$ saddle-node, $P_4$ saddle.\\
\hline
			
3.4
&
\small $a_0>0$, $c_0>0$,  $\mu(a_0+c_0\mu)<0$.
&
\small $P_0$ unstable node, $P_3$ saddle-node, $P_4$ stable node.\\
\hline
			
3.5
&
\small $a_0=0$, $c_0>0$.
&
\small $P_0\equiv P_4$ saddle-node, \newline $P_3$ saddle-node.\\
\hline
		
3.6
&
\small $c_0=0$, $a_0>0$, $c_2<0$, $\mu>0$.
&
\small $P_0\equiv P_3$ topological saddle,  \newline $P_4$ saddle.\\
\hline
		
3.7
&
\small $c_0=0$, $a_0>0$, $c_2<0$, $\mu <0$.
&
\small $P_0\equiv P_3$ topological saddle,  \newline $P_4$ stable node.\\
\hline
		
3.8
&
\small $c_0=0$, $a_0>0$, $c_2>0$, $\mu>0$.
&
\small $P_0\equiv P_3$ topological unstable node, $P_4$ saddle.\\
\hline

3.9
&
\small $c_0=0$, $a_0>0$, $c_2>0$, $\mu<0$.
&
\small $P_0\equiv P_3$ topological unstable node, $P_4$ stable node.\\
\hline
			
\end{tabular}
\caption{Classification in case 3 of Table \ref{tab:cases1-6} according with the local phase portraits of finite singular points.}
\label{tab:cases_fin_local_3}
\end{center}
\end{table}	

\vspace{-1cm}

\begin{table}
\begin{center}
\begin{tabular}{|cp{4.7cm}p{5.8cm}|}
\multicolumn{3}{l}{ \textbf{Case 4: } $\: \boldsymbol{c_3^2=4c_0c_2}$, $\boldsymbol{c_1\mu=0}$, $\boldsymbol{a_0>0}$.}\\
\hline
\hline
\small\textbf{Sub.}& \small\textbf{Conditions} & \small\textbf{Classification} \\
\hline
\hline
	
4.1
&
\small $c_0<0$.
&
\small  $P_0$ saddle, $P_3$ saddle-node.\\
\hline
			
4.2
&
\small $c_0>0$.
&
\small $P_0$ unstable node, $P_3$ saddle-node.\\
\hline
		
4.3
&
\small $c_0=0$, $c_2<0$.
&
\small $P_0\equiv P_3$ topological saddle. \\
\hline
		
4.4
&
\small $c_0=0$, $c_2>0$.
&
\small $P_0\equiv P_3$ topological unstable node. \\
\hline
			
\end{tabular}
\caption{Classification in case 4 of Table \ref{tab:cases1-6} according with the local phase portraits of finite singular points.}
\label{tab:cases_fin_local_4}
\end{center}
\end{table}	

\vspace{-2cm}

\begin{table}
\begin{center}
\begin{tabular}{|cp{4.7cm}p{5.8cm}|}
\multicolumn{3}{l}{ \textbf{Case 5: } $\: \boldsymbol{c_3^2<4c_0c_2}$, $\boldsymbol{c_1\mu\neq0}$.}\\
\hline
\hline
\small\textbf{Sub.}& \small\textbf{Conditions} & \small\textbf{Classification} \\
\hline
\hline
	
5.1
&
\small $a_0=0$.
&
\small  $P_0\equiv P_4$ saddle-node.\\
\hline

5.2
&
\small $a_0>0$, $c_0<0$, $\mu(a_0+c_0\mu)>0$.
&
\small $P_0$ saddle,  $P_4$ saddle.\\
\hline
	
5.3
&
\small $a_0>0$, $c_0<0$, $\mu(a_0+c_0\mu)<0$.
&
\small $P_0$ saddle,  $P_4$ stable node.\\
\hline

5.4
&
\small $a_0>0$, $c_0>0$, $\mu(a_0+c_0\mu)>0$.
&
\small $P_0$ unstable node,  $P_4$ saddle.\\
\hline
	
5.5
&
\small $a_0>0$, $c_0>0$, $\mu(a_0+c_0\mu)<0$.
&
\small $P_0$ unstable node,  $P_4$ stable node.\\
\hline
			
\end{tabular}
\caption{Classification in case 5 of Table \ref{tab:cases1-6} according with the local phase portraits of finite singular points.}
\label{tab:cases_fin_local_5}
\end{center}
\end{table}

\begin{table}
\begin{center}
\begin{tabular}{|cp{4.7cm}p{5.8cm}|}
\multicolumn{3}{l}{ \textbf{Case 6: } $\: \boldsymbol{c_3^2<4c_0c_2}$, $\boldsymbol{c_1\mu=0}$, $\boldsymbol{a_0>0}$.}\\
\hline
\hline
\small\textbf{Sub.}& \small\textbf{Conditions} & \small\textbf{Classification} \\
\hline
\hline
		
6.1
&
\small $c_0<0$. \textcolor{white}{Problemas de espacio}
&
\small $P_0$ saddle.\\
\hline
		
6.2
&
\small $c_0>0$.\textcolor{white}{Problemas de espacio}
&
\small $P_0$ unstable node.\\
\hline
				
\end{tabular}
\caption{Classification in case 6 of Table \ref{tab:cases1-6} according with the local phase portraits of finite singular points.}
\label{tab:cases_fin_local_6}
\end{center}
\end{table}	

\newpage

\section{Local study of infinite singular points}\label{sec:infinite}

In order to study the behavior of the trajectories of system \eqref{system} near infinity we consider its Poincar\'e compactification. For the moment we assume the same hypothesis $(H_1)$ than in previous sections. According to equations \eqref{Poincare_comp_U1} and \eqref{Poincare_comp_U2},  we get the compactification in the local charts $U_1$ and $U_2$ respectively. From Section \ref{sec:preliminaries} it is enough to study the singular points over $v=0$ in the chart $U_1$ and the origin of the chart $U_2$.

In chart $U_1$ system \eqref{system} writes
\begin{equation}\label{systemU1}
\begin{split}
\dot{u}&= c_2(\mu+1)u^3+c_3(\mu+1)u^2v+(c_0-a_0)uv^2+c_1(\mu+1)uv,\\
\dot{v}&=c_2\mu u^2v+c_3\mu u v^2 - a_0 v^3 + c_1 \mu v^2.
\end{split}
\end{equation}
Taking $v=0$ we get $\dot{u}\mid_{v=0}=c_2(\mu+1)u^3$ and $\dot{v} \mid_{v=0} =0$. Therefore if $\mu=-1$ all points at infinity are singular points, and we will not deal with this situation in this paper. In other case, if $\mu\neq-1$ the only singular point is the origin of $U_1$, which we denote by $O_1$. The linear part of system \eqref{systemU1} at the origin is identically zero, so we must use the blow-up technique to study it, leading to the following result which is proved in Subsections \ref{subsec:O1_c1nonzero} and \ref{subsec:O1_c1zero}.
From now on we include the condition $\mu\neq-1$ in our hypothesis, so we will work under the conditions
\begin{align*}
(H_2)= \left\lbrace c_2\neq 0, a_0\geq 0, c_1\geq 0, c_3\geq 0, a_0+c_0\mu\neq 0, a_0c_1\mu\neq0, \mu\neq-1 \right\rbrace.
\end{align*}

\begin{lemma}\label{lemma_O1}
Asumming hypothesis $(H_2)$ the origin of the chart $U_1$ is an infinite singular point of system \eqref{system}, and it has $47$ distinct local phase portraits described in Figure \ref{fig:localppO1}.
\end{lemma}

\begin{figure}[H]
\centering
\begin{subfigure}[h]{2.4cm}
\centering
\includegraphics[width=2.4cm]{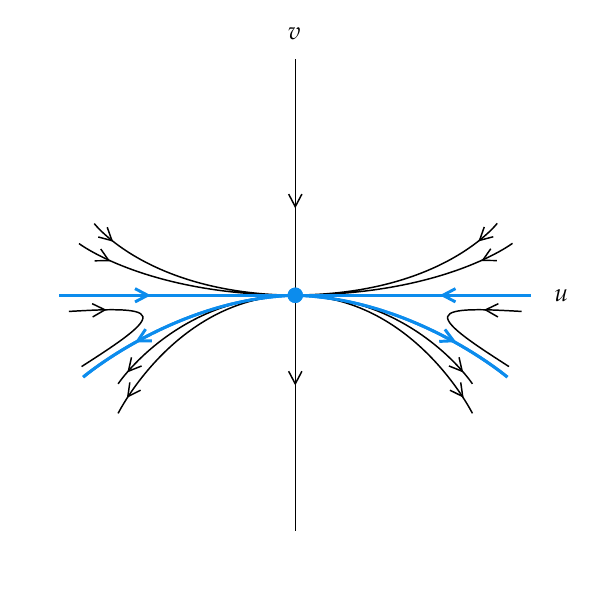}
\caption*{(L1)}
\end{subfigure}
\begin{subfigure}[h]{2.4cm}
\centering
\includegraphics[width=2.4cm]{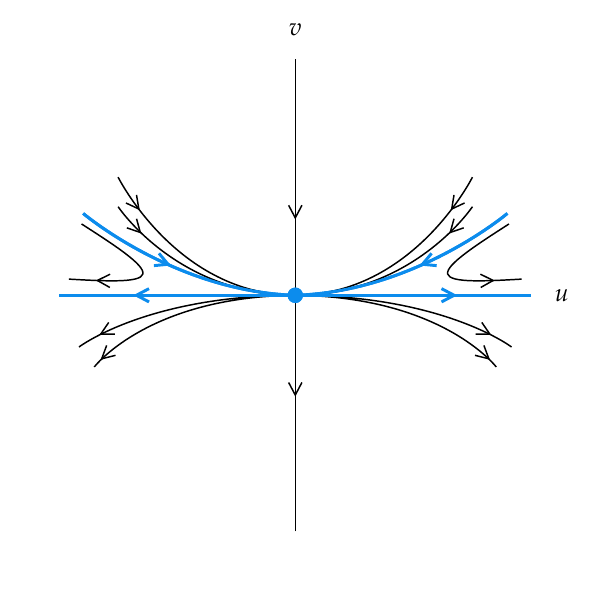}\\
\caption*{(L2)}
\end{subfigure}
\begin{subfigure}[h]{2.4cm}
\centering
\includegraphics[width=2.4cm]{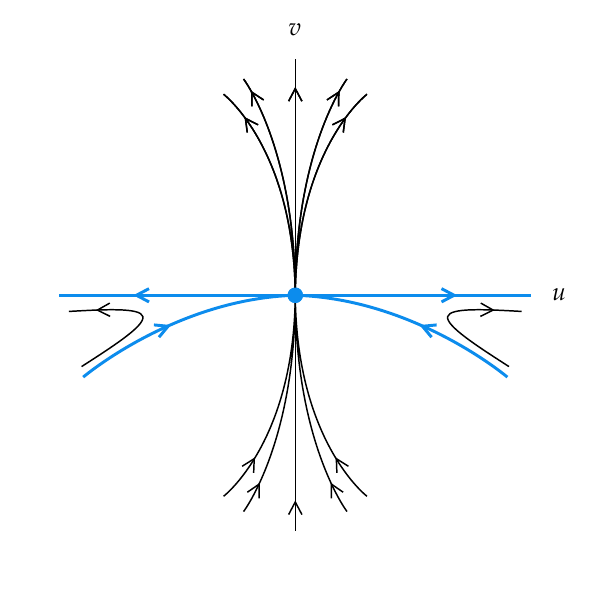}
\caption*{(L3)}
\end{subfigure}
\begin{subfigure}[h]{2.4cm}
\centering
\includegraphics[width=2.4cm]{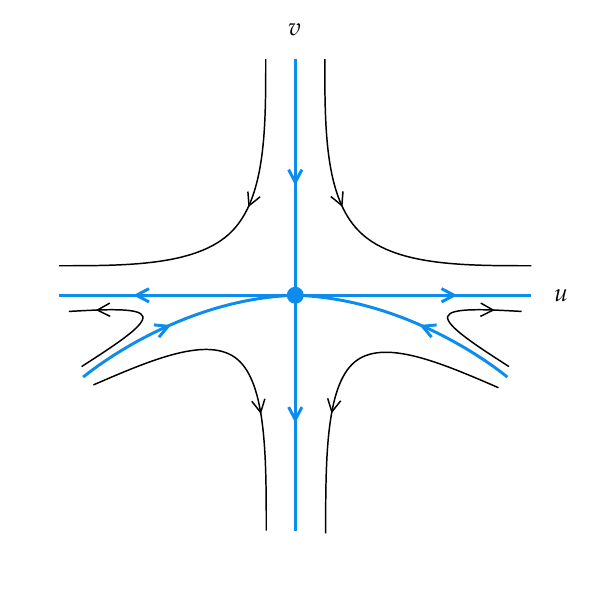}
\caption*{(L4)}
\end{subfigure}
\begin{subfigure}[h]{2.4cm}
\centering
\includegraphics[width=2.4cm]{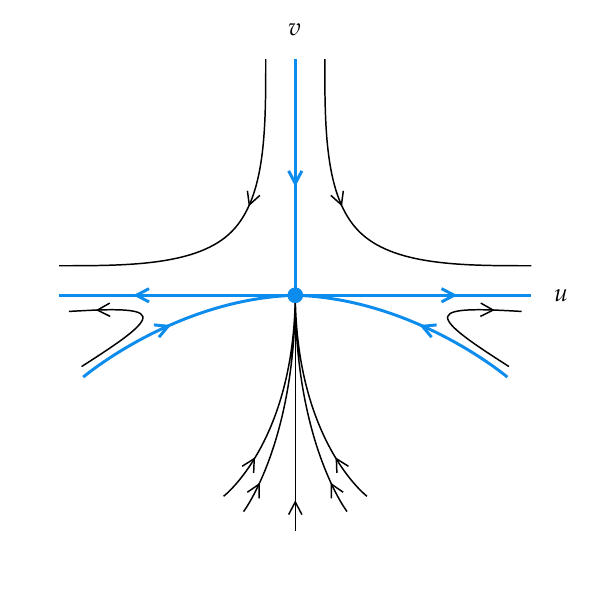}
\caption*{(L5)}
\end{subfigure}
\begin{subfigure}[h]{2.4cm}
\centering
\includegraphics[width=2.4cm]{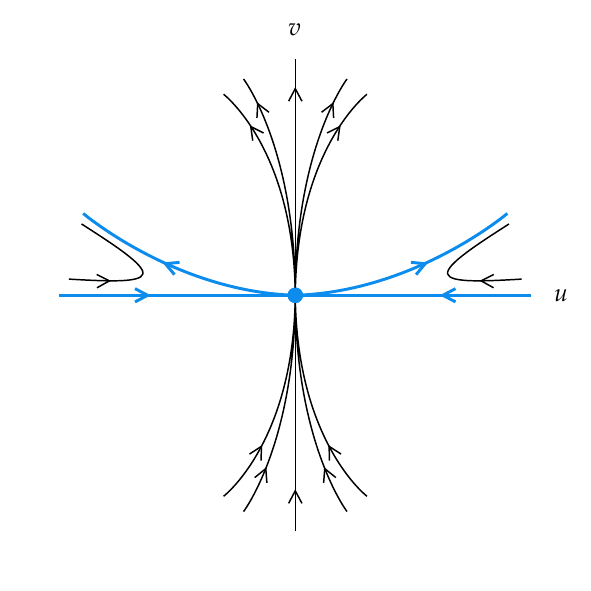}
\caption*{(L6)}
\end{subfigure}
\begin{subfigure}[h]{2.4cm}
\centering
\includegraphics[width=2.4cm]{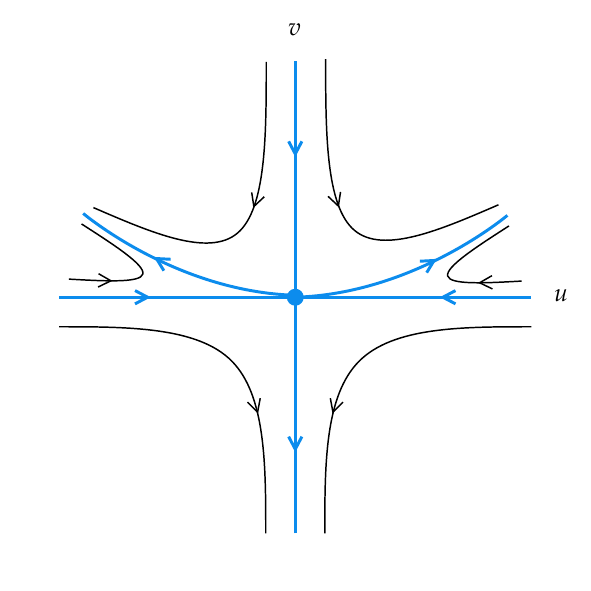}
\caption*{(L7)}
\end{subfigure}
\begin{subfigure}[h]{2.4cm}
\centering
\includegraphics[width=2.4cm]{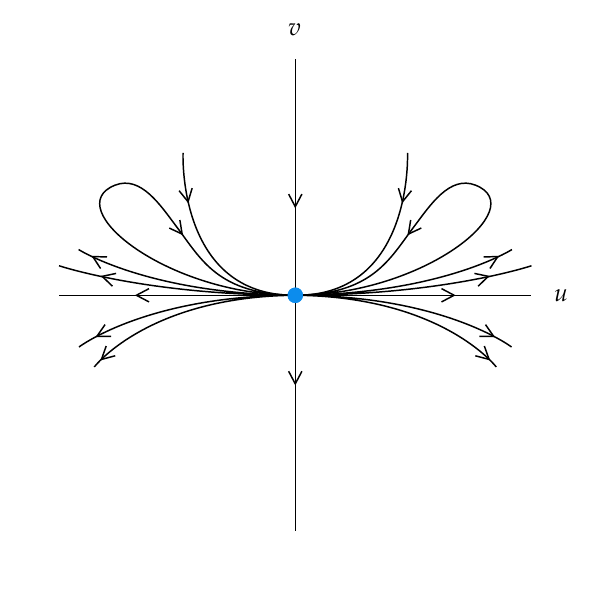}
\caption*{(L8)}
\end{subfigure}
\begin{subfigure}[h]{2.4cm}
\centering
\includegraphics[width=2.4cm]{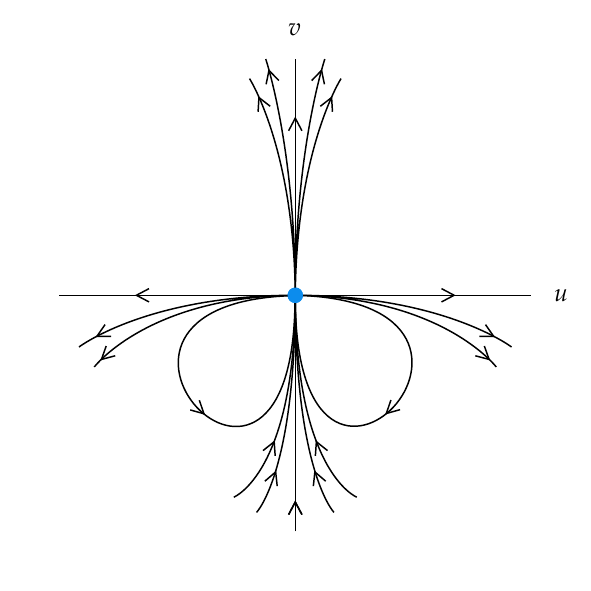}
\caption*{(L9)}
\end{subfigure}
\begin{subfigure}[h]{2.4cm}
\centering
\includegraphics[width=2.4cm]{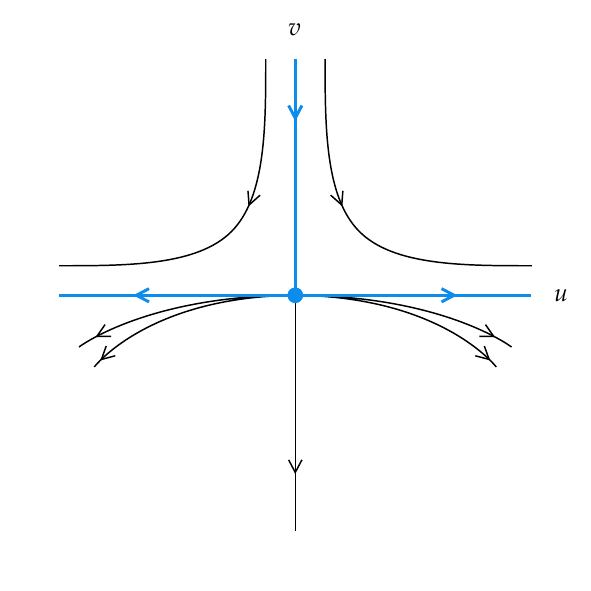}
\caption*{(L10)}
\end{subfigure}
\begin{subfigure}[h]{2.4cm}
\centering
\includegraphics[width=2.4cm]{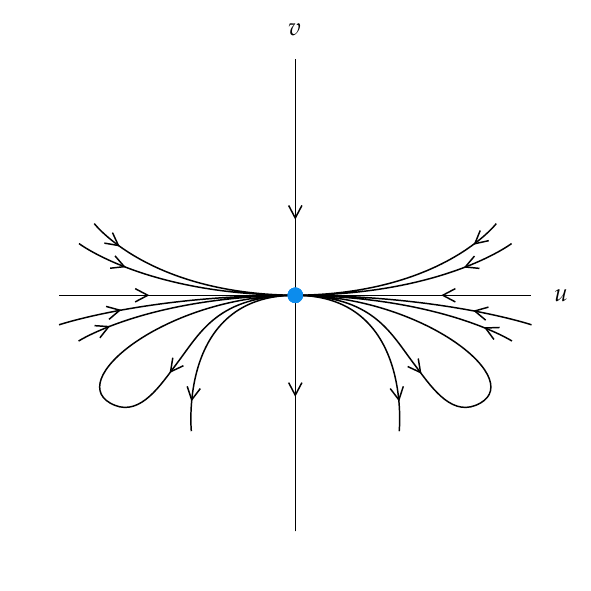}
\caption*{(L11)}
\end{subfigure}
\begin{subfigure}[h]{2.4cm}
\centering
\includegraphics[width=2.4cm]{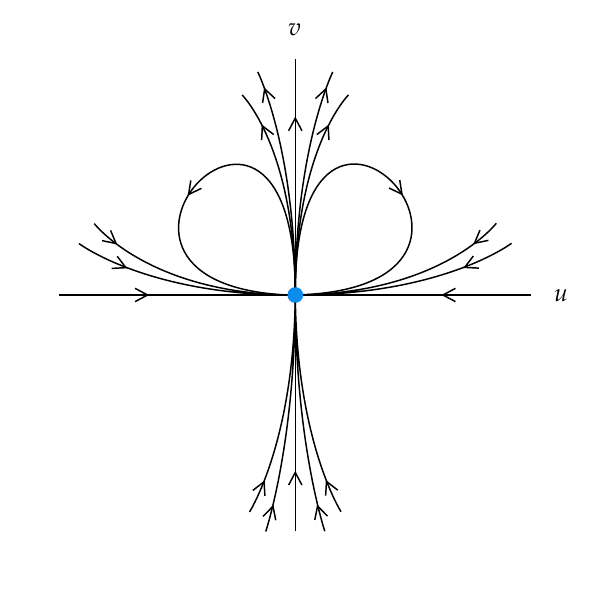}
\caption*{(L12)}
\end{subfigure}
\begin{subfigure}[h]{2.4cm}
\centering
\includegraphics[width=2.4cm]{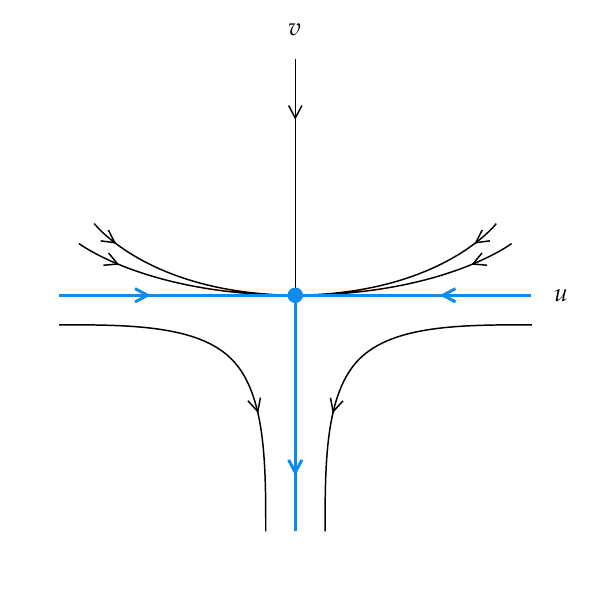}
\caption*{(L13)}
\end{subfigure}
\begin{subfigure}[h]{2.4cm}
\centering
\includegraphics[width=2.4cm]{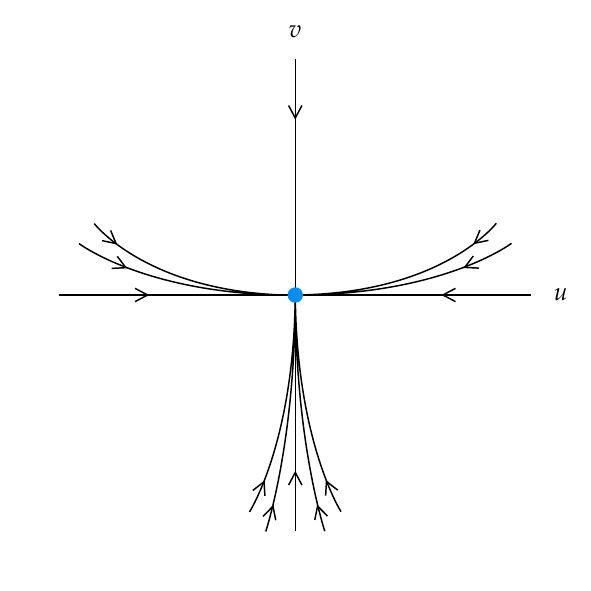}
\caption*{(L14)}
\end{subfigure}
\begin{subfigure}[h]{2.4cm}
\centering
\includegraphics[width=2.4cm]{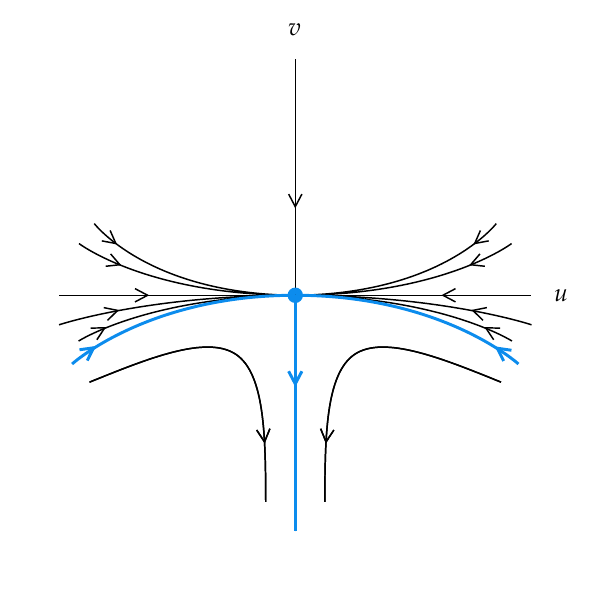}
\caption*{(L15)}
\end{subfigure}
\begin{subfigure}[h]{2.4cm}
\centering
\includegraphics[width=2.4cm]{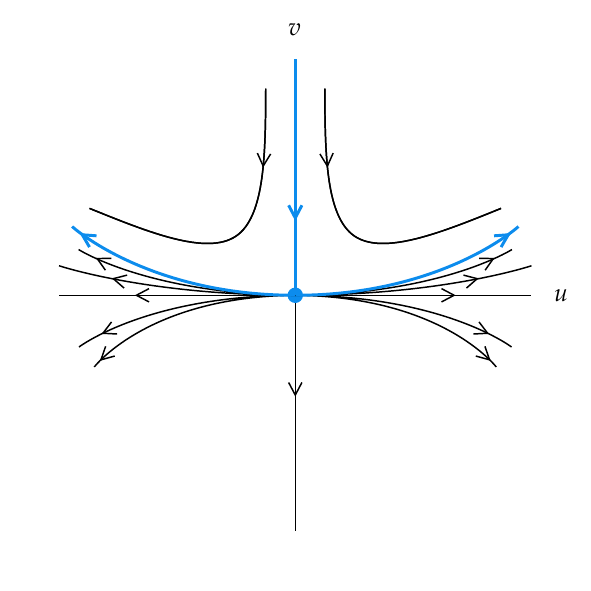}
\caption*{(L16)}
\end{subfigure}	
\begin{subfigure}[h]{2.4cm}
\centering
\includegraphics[width=2.4cm]{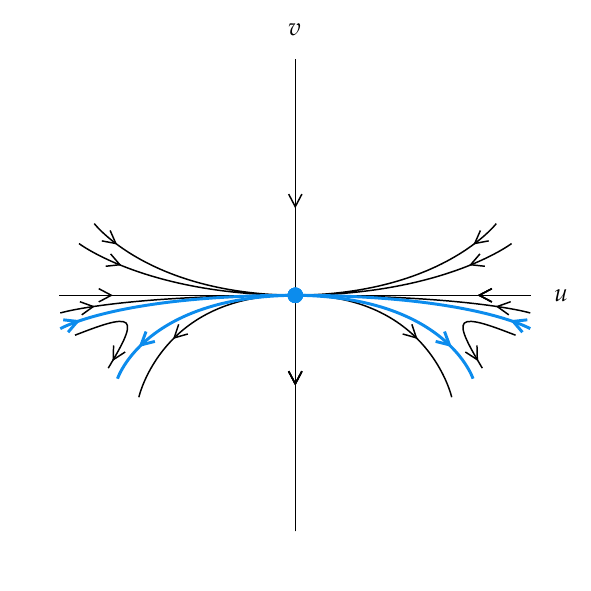}
\caption*{(L17)}
\end{subfigure}
\begin{subfigure}[h]{2.4cm}
\centering
\includegraphics[width=2.4cm]{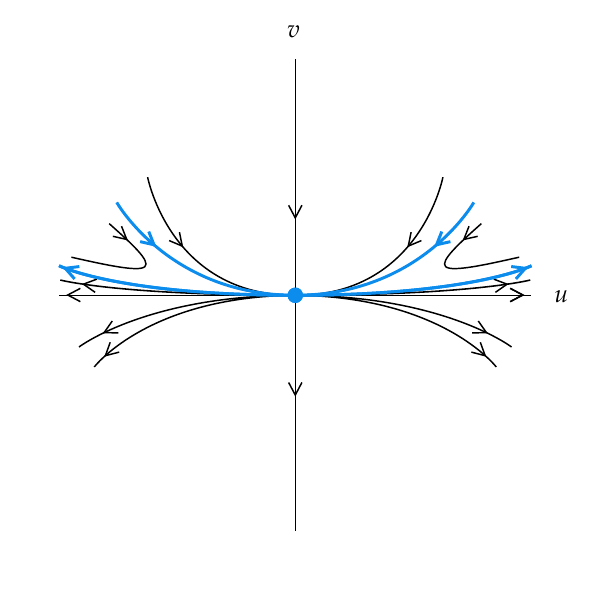}
\caption*{(L18)}
\end{subfigure}
\begin{subfigure}[h]{2.4cm}
\centering
\includegraphics[width=2.4cm]{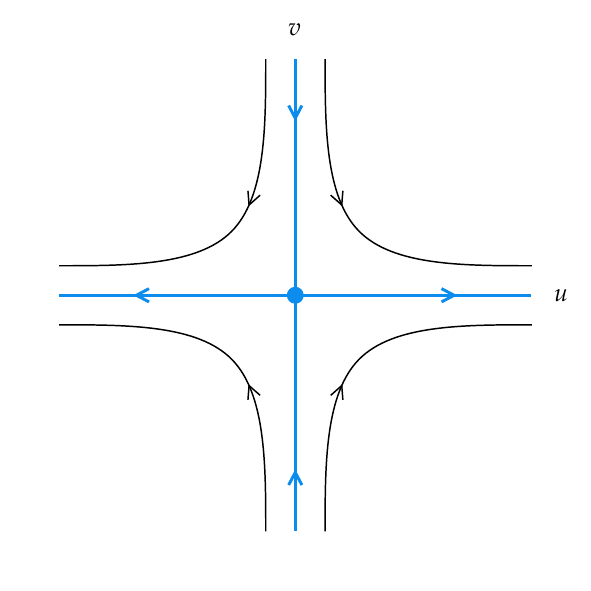}
\caption*{(L19)}
\end{subfigure}
\begin{subfigure}[h]{2.4cm}
\centering
\includegraphics[width=2.4cm]{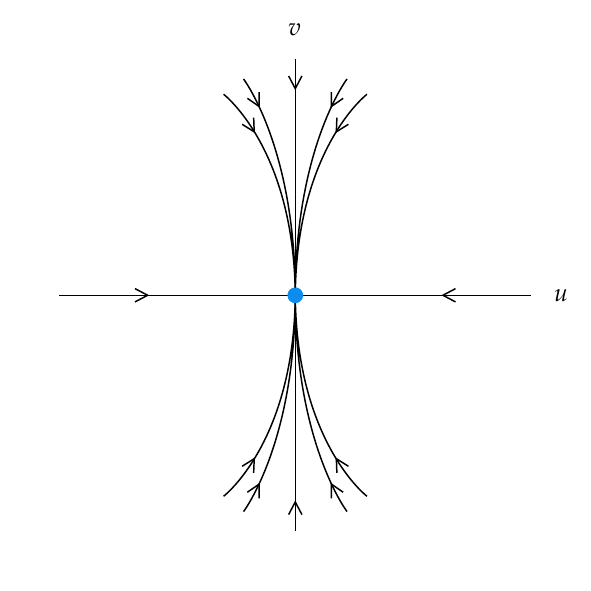}
\caption*{(L20)}
\end{subfigure}
\begin{subfigure}[h]{2.4cm}
\centering
\includegraphics[width=2.4cm]{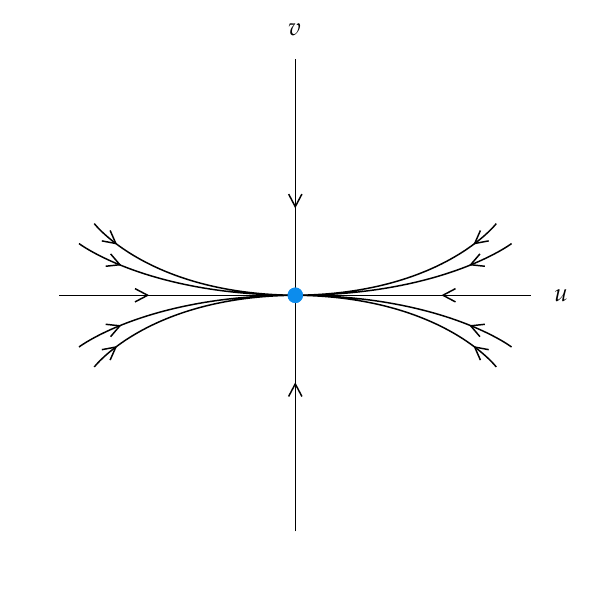}
\caption*{(L21)}
\end{subfigure}
\begin{subfigure}[h]{2.4cm}
\centering
\includegraphics[width=2.4cm]{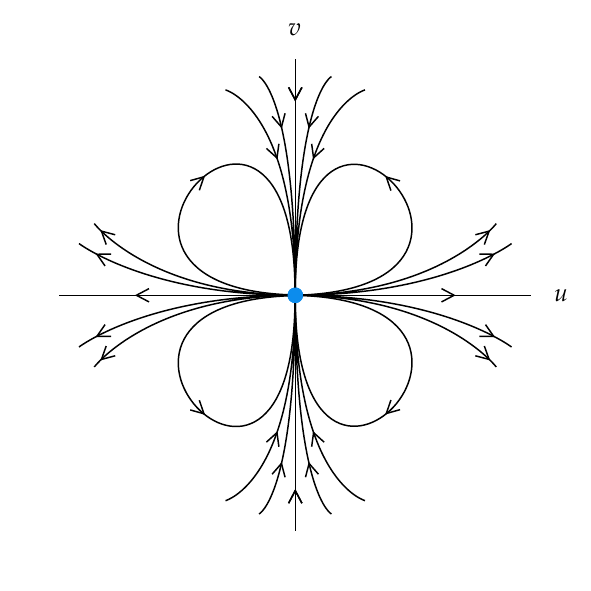}
\caption*{(L22)}
\end{subfigure}
\begin{subfigure}[h]{2.4cm}
\centering
\includegraphics[width=2.4cm]{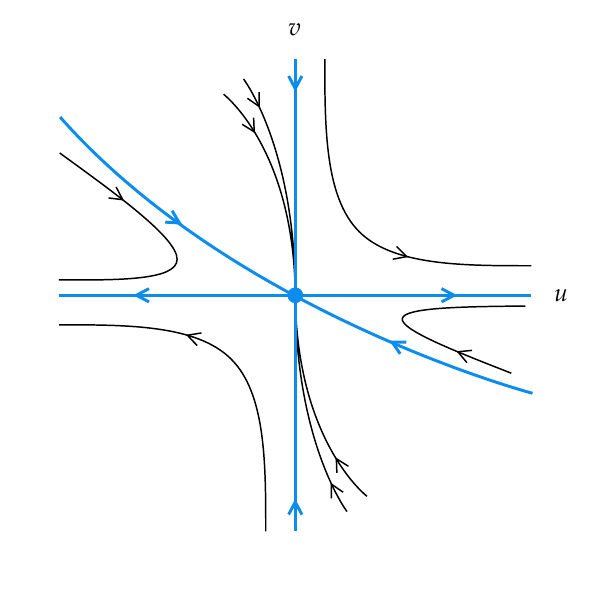}
\caption*{(L23)}
\end{subfigure}
\begin{subfigure}[h]{2.4cm}
\centering
\includegraphics[width=2.4cm]{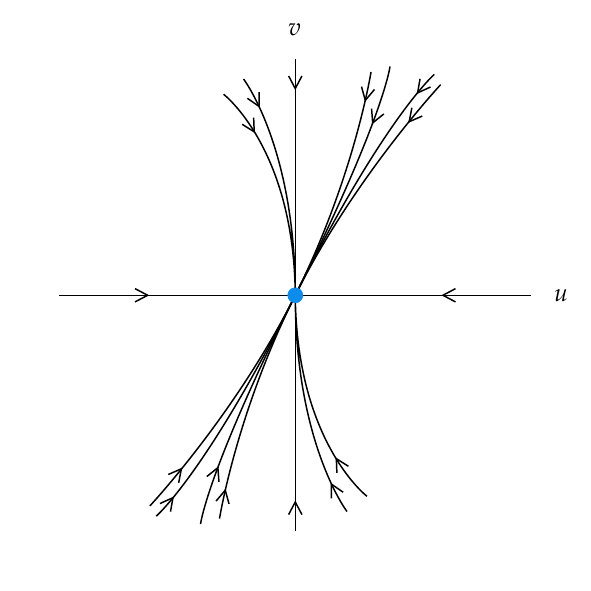}
\caption*{(L24)}
\end{subfigure}
\begin{subfigure}[h]{2.4cm}
\centering
\includegraphics[width=2.4cm]{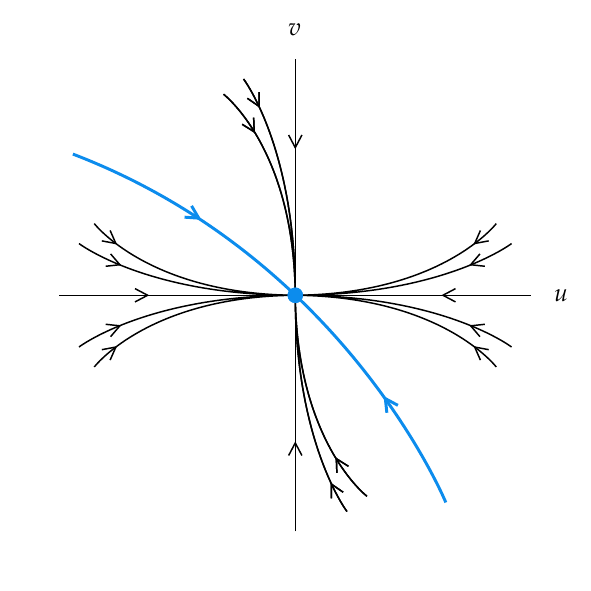}
\caption*{(L25)}
\end{subfigure}
\end{figure}

\begin{figure}[H]
\centering
\ContinuedFloat
\begin{subfigure}[h]{2.4cm}
	\centering
	\includegraphics[width=2.4cm]{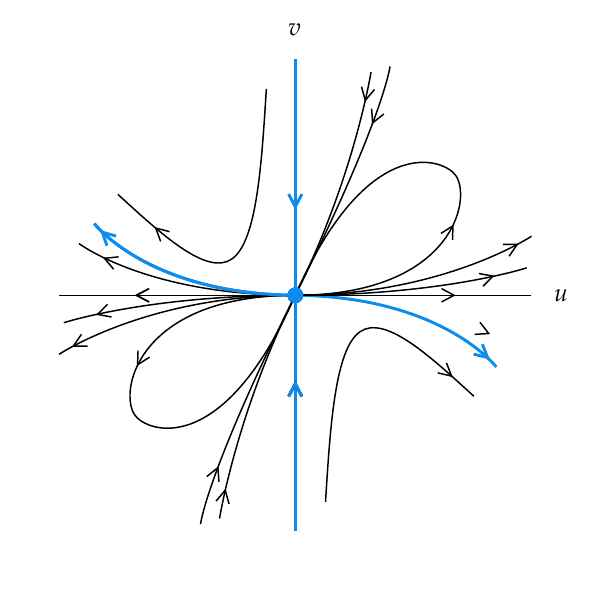}
	\caption*{(L26)}
\end{subfigure}
\begin{subfigure}[h]{2.4cm}
	\centering
	\includegraphics[width=2.4cm]{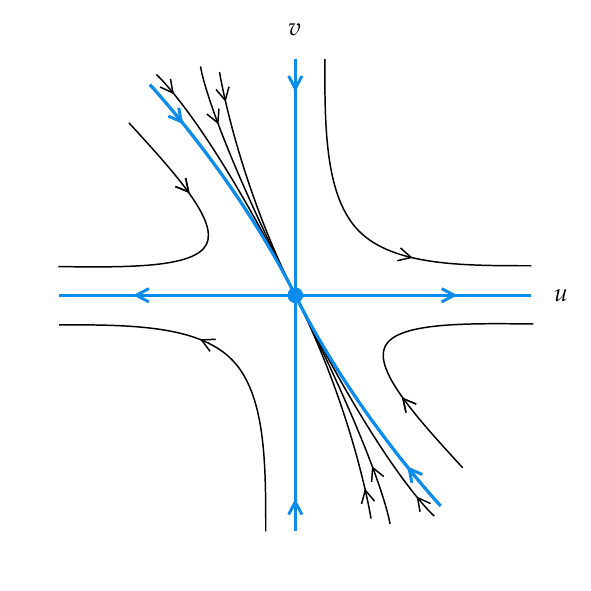}
	\caption*{(L27)}
\end{subfigure}
\begin{subfigure}[h]{2.4cm}
	\centering
	\includegraphics[width=2.4cm]{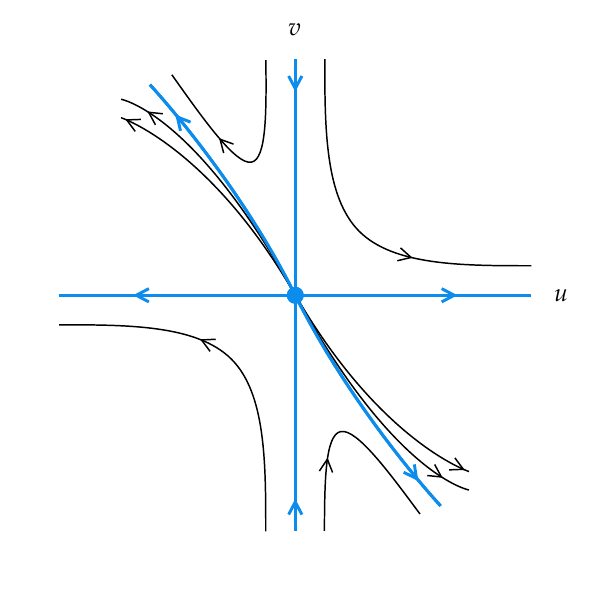}
	\caption*{(L28)}
\end{subfigure}
\begin{subfigure}[h]{2.4cm}
	\centering
	\includegraphics[width=2.4cm]{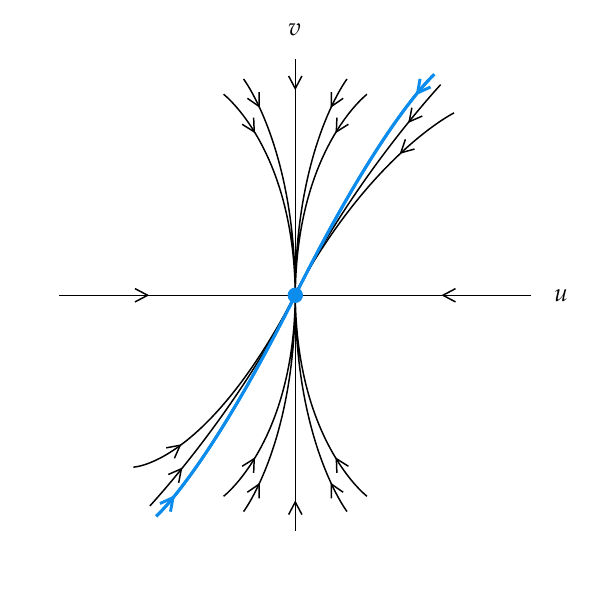}
	\caption*{(L29)}
\end{subfigure}
\begin{subfigure}[h]{2.4cm}
	\centering
	\includegraphics[width=2.4cm]{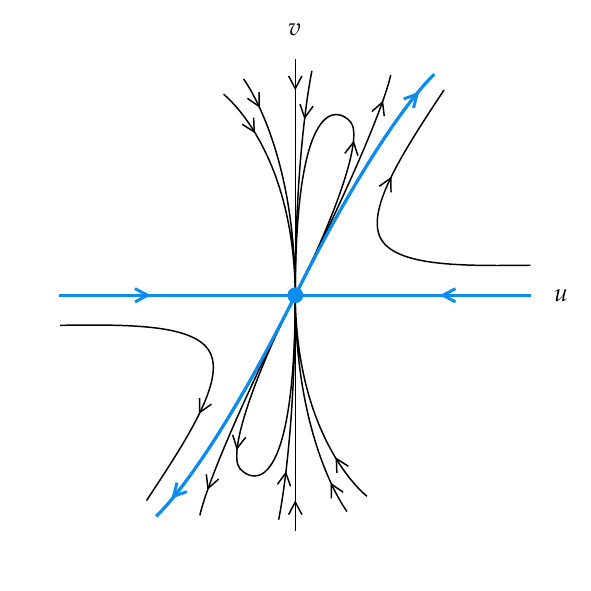}
	\caption*{(L30)}
\end{subfigure}
\begin{subfigure}[h]{2.4cm}
\centering
\includegraphics[width=2.4cm]{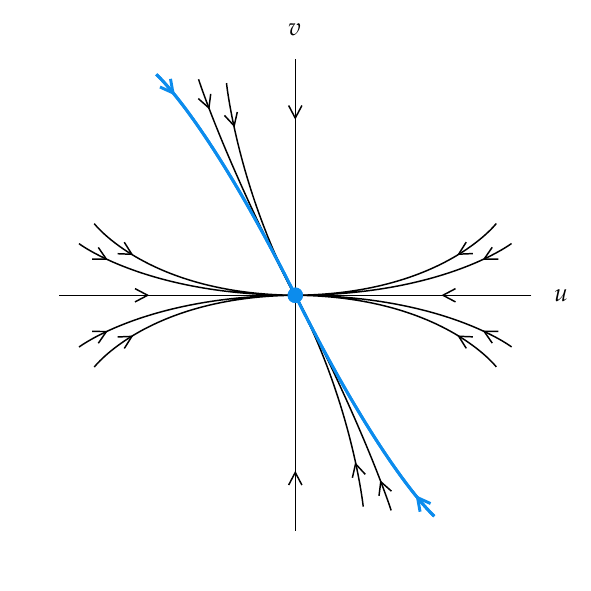}
\caption*{(L31)}
\end{subfigure}
\begin{subfigure}[h]{2.4cm}
\centering
\includegraphics[width=2.4cm]{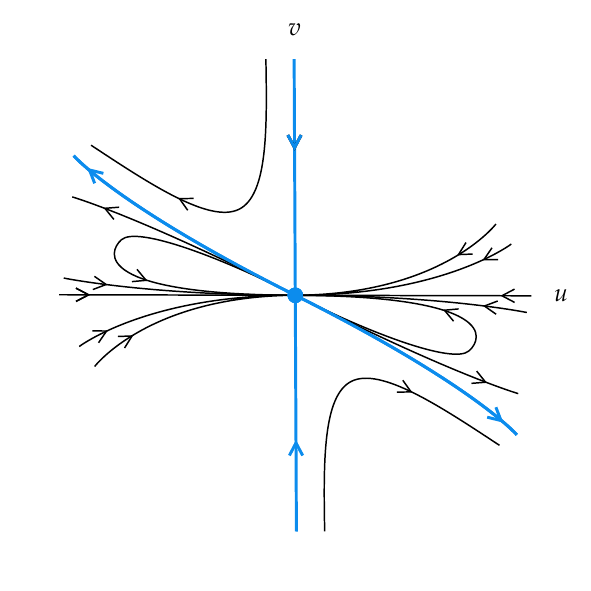}
\caption*{(L32)}
\end{subfigure}
\begin{subfigure}[h]{2.4cm}
\centering
\includegraphics[width=2.4cm]{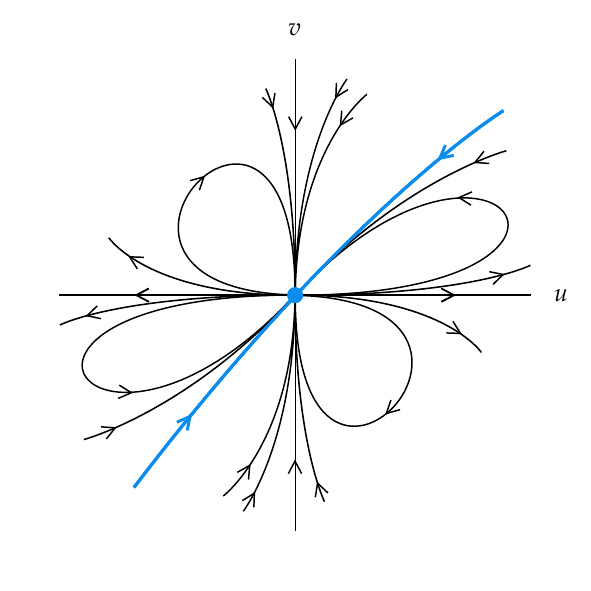}
\caption*{(L33)}
\end{subfigure}
\begin{subfigure}[h]{2.4cm}
\centering
\includegraphics[width=2.4cm]{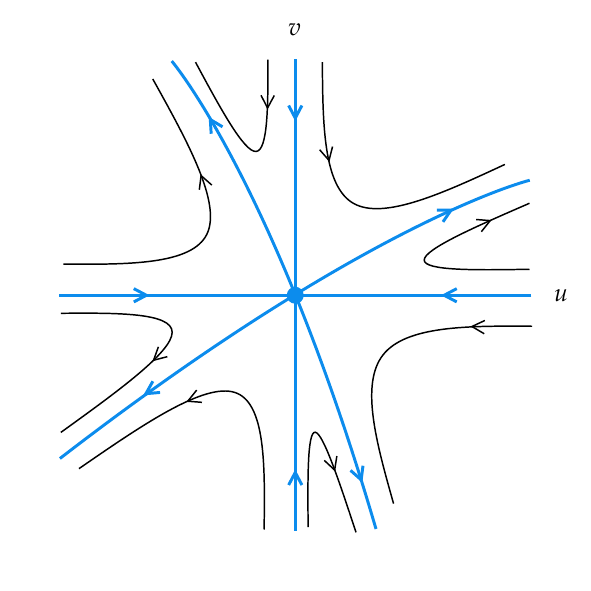}
\caption*{(L34)}
\end{subfigure}
\begin{subfigure}[h]{2.4cm}
\centering
\includegraphics[width=2.4cm]{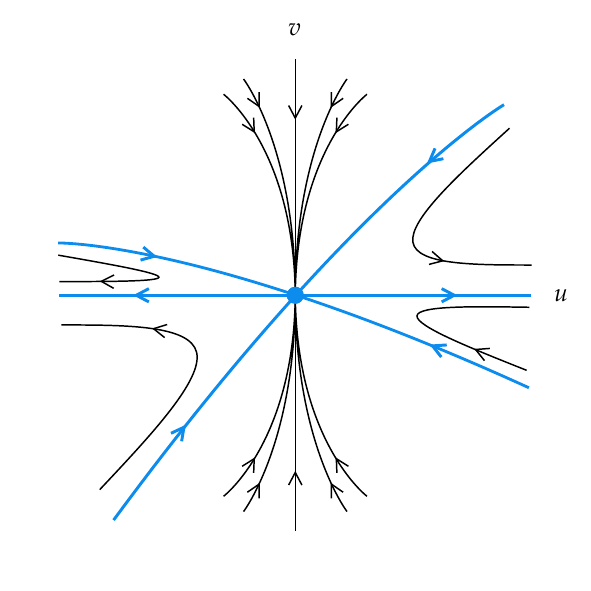}
\caption*{(L35)}
\end{subfigure}
\begin{subfigure}[h]{2.4cm}
\centering
\includegraphics[width=2.4cm]{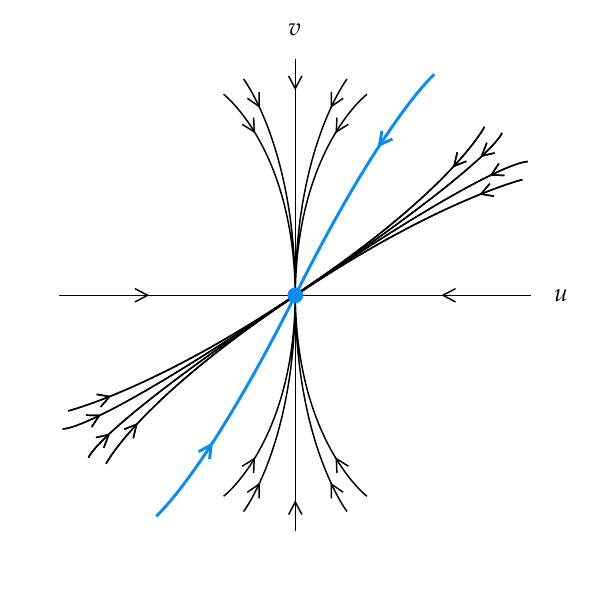}
\caption*{(L36)}
\end{subfigure}
\begin{subfigure}[h]{2.4cm}
\centering
\includegraphics[width=2.4cm]{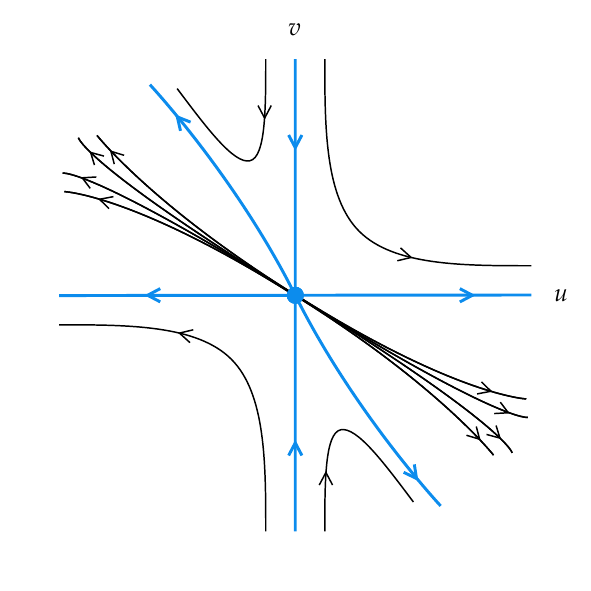}
\caption*{(L37)}
\end{subfigure}
\begin{subfigure}[h]{2.4cm}
\centering
\includegraphics[width=2.4cm]{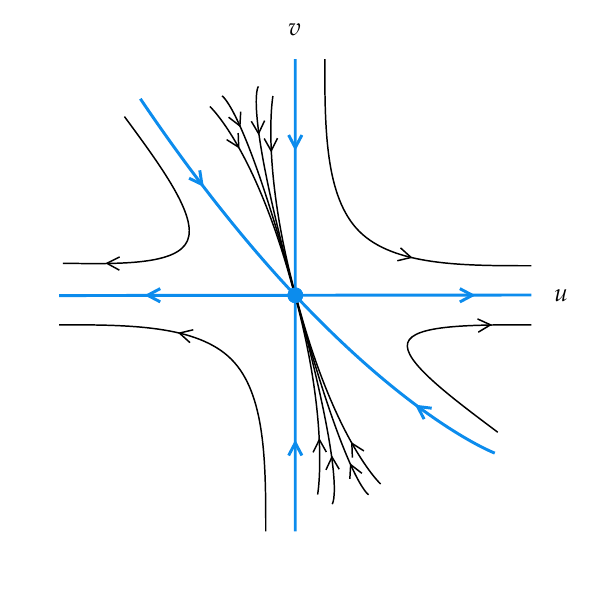}
\caption*{(L38)}
\end{subfigure}
\begin{subfigure}[h]{2.4cm}
\centering
\includegraphics[width=2.4cm]{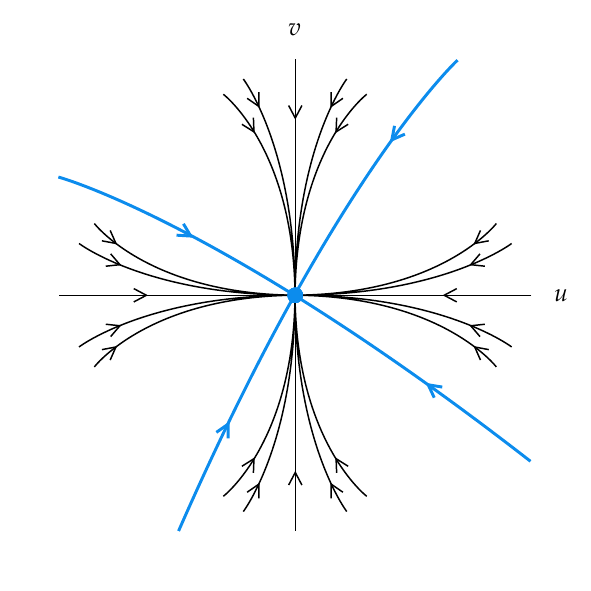}
\caption*{(L39)}
\end{subfigure}
\begin{subfigure}[h]{2.4cm}
\centering
\includegraphics[width=2.4cm]{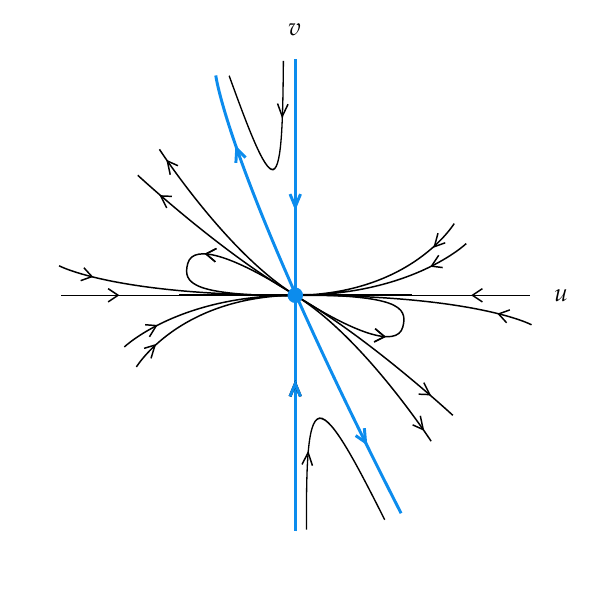}
\caption*{(L40)}
\end{subfigure}
\begin{subfigure}[h]{2.4cm}
\centering
\includegraphics[width=2.4cm]{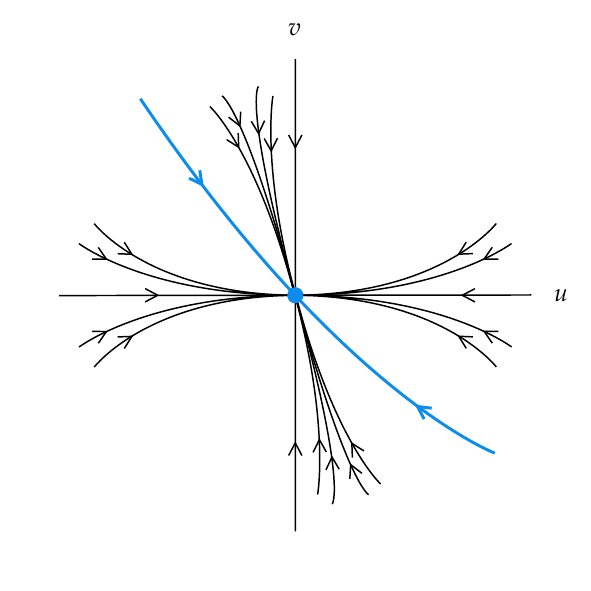}
\caption*{(L41)}
\end{subfigure}
\begin{subfigure}[h]{2.4cm}
\centering
\includegraphics[width=2.4cm]{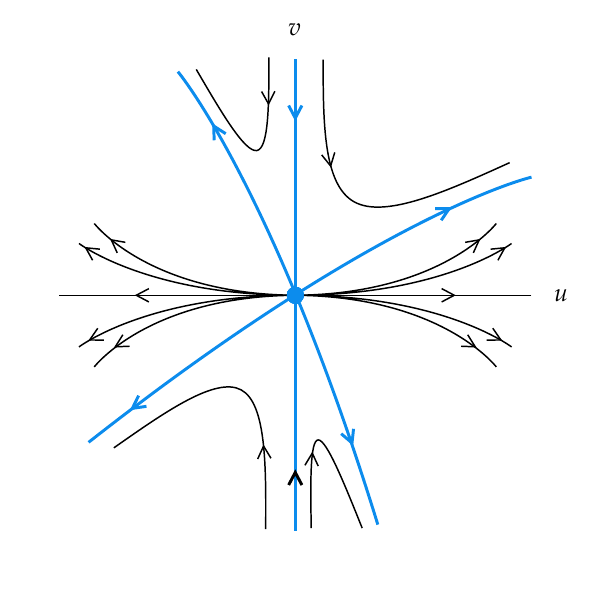}
\caption*{(L42)}
\end{subfigure}
\begin{subfigure}[h]{2.4cm}
\centering
\includegraphics[width=2.4cm]{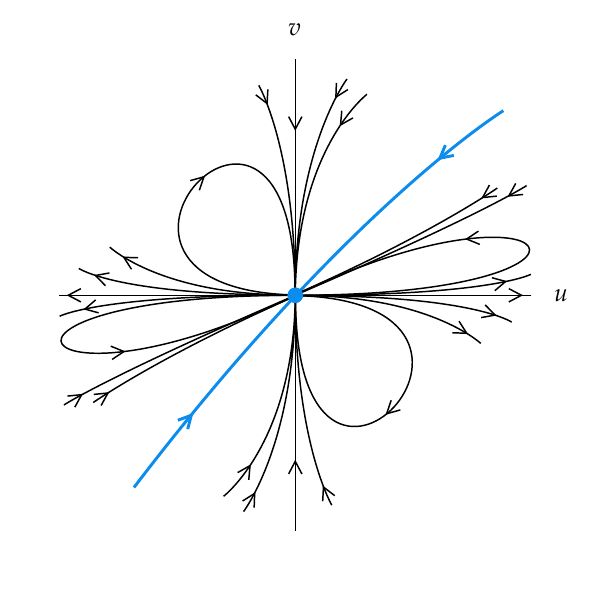}
\caption*{(L43)}
\end{subfigure}
\begin{subfigure}[h]{2.4cm}
\centering
\includegraphics[width=2.4cm]{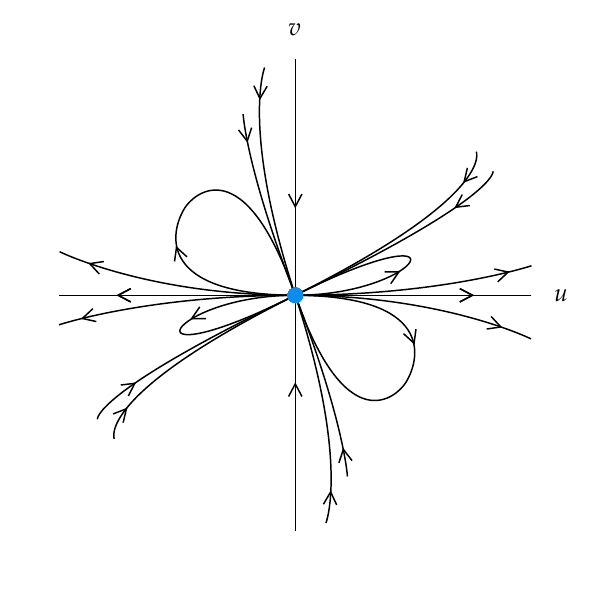}
\caption*{(L44)}
\end{subfigure}
\begin{subfigure}[h]{2.4cm}
\centering
\includegraphics[width=2.4cm]{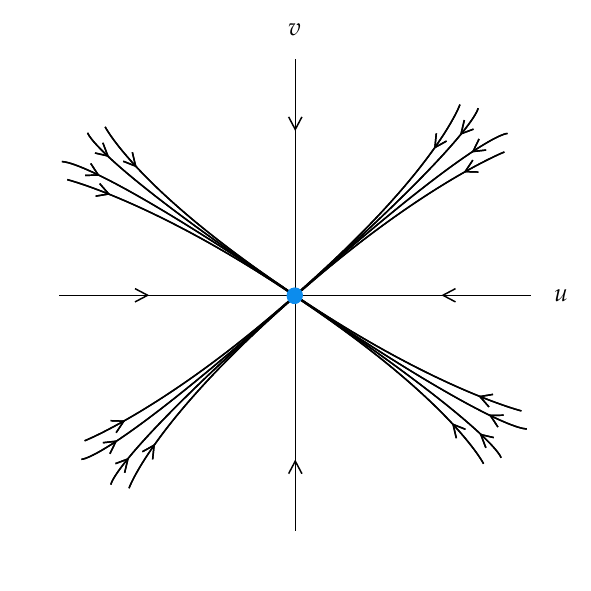}
\caption*{(L45)}
\end{subfigure}
\begin{subfigure}[h]{2.4cm}
\centering
\includegraphics[width=2.4cm]{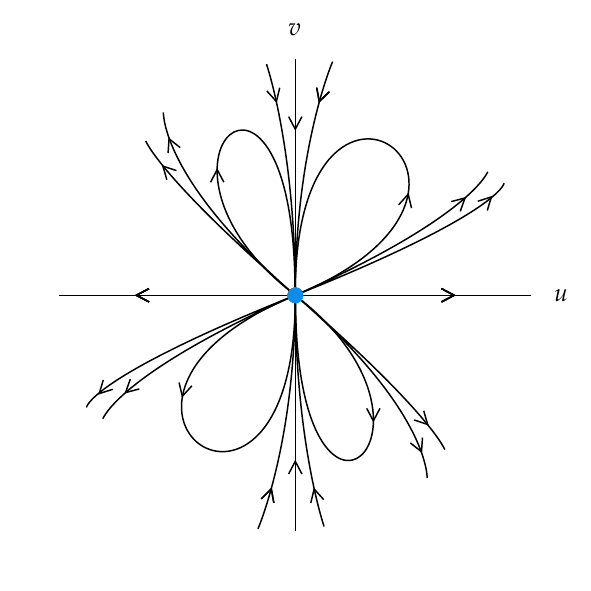}
\caption*{(L46)}
\end{subfigure}
\begin{subfigure}[h]{2.4cm}
\centering
\includegraphics[width=2.4cm]{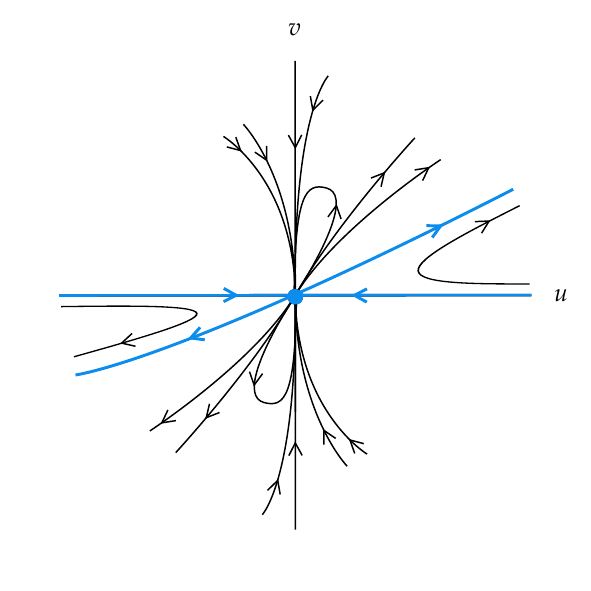}
\caption*{(L47)}
\end{subfigure}
\caption{Local phase portraits of the infinite singular point $O_1$.}
\label{fig:localppO1}
\end{figure}

For system \eqref{systemU1}, if $c_1\neq0$ the characteristic polynomial is $\mathcal{F}=-c_1uv^2\not \equiv 0$, so the origin is a nondicrital singular point. If $c_1=0$ the characteristic polynomial is $\mathcal{F}=-c_3u^2v-c_2u^3v-c_0uv^3$, which cannot be identically zero because $c_2\neq0$. We will study this two cases separately.

\subsection{Case $c_1$ non-zero}\label{subsec:O1_c1nonzero}

Consider $c_1\neq0$. We introduce the new variable $w_1$ by means of the variable change $uw_1=v$, and get the system
\begin{equation}\label{sis_blowup2}
\begin{split}
\dot{u}&=(c_0-a_0)u^3w_1^2+c_3(\mu+1)u^3w_1+c_2(\mu+1)u^3+c_1(\mu+1)u^2w_1,\\
\dot{w_1}&=-c_0u^2w_1^3-c_3u^2w_1^2-c_2u^2w_1-c_1uw_1^2.
\end{split}
\end{equation}
Now we cancel the common factor $u$, getting the system
\begin{equation}\label{sis_blowup3}
\begin{split}
\dot{u}&=(c_0-a_0)u^2w_1^2+c_3(\mu+1)u^2w_1+c_2(\mu+1)u^2+c_1(\mu+1)uw_1,\\
\dot{w_1}&=-c_0uw_1^3-c_3uw_1^2-c_2uw_1-c_1w_1^2,
\end{split}
\end{equation}
for which the only singular point on the exceptional divisor is the origin, and it is linearly zero, so we have to repeat the process. Now the characteristic polynomial is $\mathcal{F}=-c_2(\mu+2)u^2w_1- c_1(\mu+2)uw_1^2$, so the origin is a nondicritical singular point if $\mu\neq-2$, and it is dicritical if $\mu=-2$. In both cases we introduce the new variable $w_2$ doing the change $uw_2=w_1$, obtaining the system
\begin{equation}\label{sis_blowup4}
\begin{split}
\dot{u}&=(c_0-a_0)u^4w_2^2+c_3(\mu+1)u^3w_2+c_2(\mu+1)u^2+c_1(\mu+1)u^2w_2,\\
\dot{w_2}&=(a_0-2c_0)u^3w_2^3-c_3(\mu+2)u^2w_2^2-c_1(\mu+2)uw_2^2-c_2(\mu+2)uw_2.
\end{split}
\end{equation}
In the nondicritical case we have to cancel the common factor $u$ obtaining
\begin{equation}\label{sis_blowup5}
\begin{split}
\dot{u}&=(c_0-a_0)u^3w_2^2+c_3(\mu+1)u^2w_2+c_2(\mu+1)u+c_1(\mu+1)uw_2,\\
\dot{w_2}&=(a_0-2c_0)u^2w_2^3-c_3(\mu+2)uw_2^2-c_1(\mu+2)w_2^2-c_2(\mu+2)w_2.
\end{split}
\end{equation}
But in the dicritical case, when $\mu=-2$, we must cancel the common factor $u^2$ from system \eqref{sis_blowup4}, and we obtain the system
\begin{equation}\label{sis_blowup6}
\begin{split}
\dot{u}&=(c_0-a_0)u^2w_2^2-c_3uw_2-c_2-c_1w_2,\\
\dot{w_2}&=(a_0-2c_0)uw_2^3.
\end{split}
\end{equation}

\subsubsection{Nondicritical case}

In this case it is necessary to study the singular points of system \eqref{sis_blowup5} on the exceptional divisor. The origin is always a singular point, and we denote it by $Q_0$. As $\mu+2\neq0$ there is another singular point, $Q_1=\left( 0,-c_2/c_1\right)$ and we  determine their local phase portraits.

The origin $Q_0$ is always hyperbolic. It is a saddle if $\mu\in(-\infty, -2)\cup (-1, +\infty)$, a stable node if $c_2>0$ and $\mu\in (-2,-1)$, and an unstable node if $c_2<0$ and $\mu\in (-2,-1)$.

The singular point $Q_1$ is semi-hyperbolic. If $c_2(a_0+c_0\mu)>0$ then $Q_1$ is a topological saddle, if $c_2(\mu+2)>0$ and $(\mu+2)(a_0+c_0\mu)<0$ then it is a topological unstable node, and if $c_2(\mu+2)<0$ and $(\mu+2)(a_0+c_0\mu)>0$ it is a topological stable node. These conditions come together in the next 7 subcases.

\begin{enumerate}[(1)]
\item If $\mu \in (-\infty,-2)\cup (-1,+\infty)$ and $c_2(a_0+c_0\mu)>0$, then $Q_0$ is a saddle and $Q_1$ a topological saddle.
In order to determine the phase portrait around the $w_2$-axis for system \eqref{sis_blowup5}, we must fix the sign of $c_2$, which determines the position of the singular point $Q_1$, and also the sign of $\mu+1$, which determines the sense of the flow along the $x$-axis. Thus we deal with the following subcases.
	
\vspace{0.1cm}
Subcase (1.1). Let $\mu<-2$ (so $\mu+1<0$) and $c_2>0$.  Then the singular point $Q_1$ is on the negative part of the $w_2$-axis and the expression $\dot{u}\mid_{w_2=0}=c_2(\mu+1)u$ determines the sense of the flow, so the phase portrait is the one in Figure \ref{fig:blowup_c1_1.1}(a).
	
To return to system \eqref{sis_blowup4} we multiply by $u$, thus the orbits in the second and third quadrants change their orientation. Moreover all the points on the $w_2$-axis become singular points. The resultant phase portrait is given in Figure \ref{fig:blowup_c1_1.1}(b).
	
When going back to the $(u,w_1)$-plane the second and the third quadrants swap from the $(u,w_2)$-plane, and the exceptional divisor shrinks to a point, and hence the orbits are slightly modified. Attending to the expresions of $\dot{u}\mid_{w_1=0}=c_2(\mu+1)u^2$ and $\dot{w_1}\mid_{u=0}=-c_1w_1^2$, we know the sense of the flow along the axes. Following the results mentioned in Subsection \ref{subsec:blowup}, the separatrix of the singular point $Q_1=\left( 0,-c_2/c_1\right)$ in the $(u,w_2)$-plane, becomes the separatrix with slope $-c_2/c_1$ in the $(u,w_1)$-plane. We get the phase portrait given in Figure \ref{fig:blowup_c1_1.1}(c), and multiplying again by $u$, the one given in Figure \ref{fig:blowup_c1_1.1}(d).
	
\begin{figure}[H]
\centering
\begin{subfigure}[h]{3cm}
\centering
\includegraphics[width=3cm]{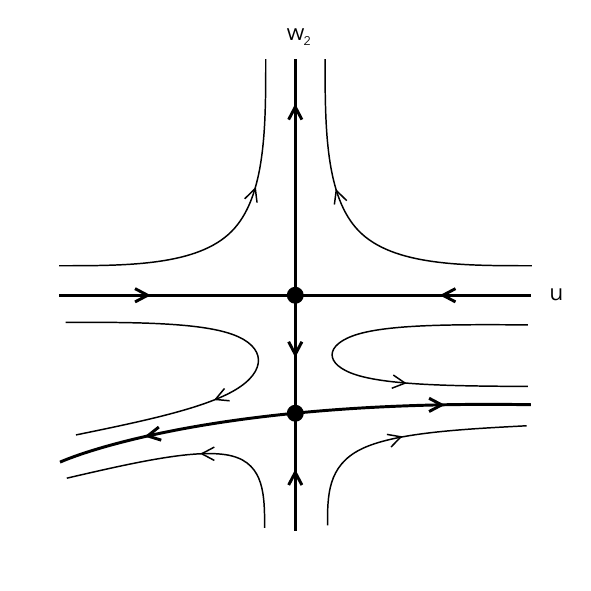}
\caption*{(a)}
\end{subfigure}
\begin{subfigure}[h]{3cm}
\centering
\includegraphics[width=3cm]{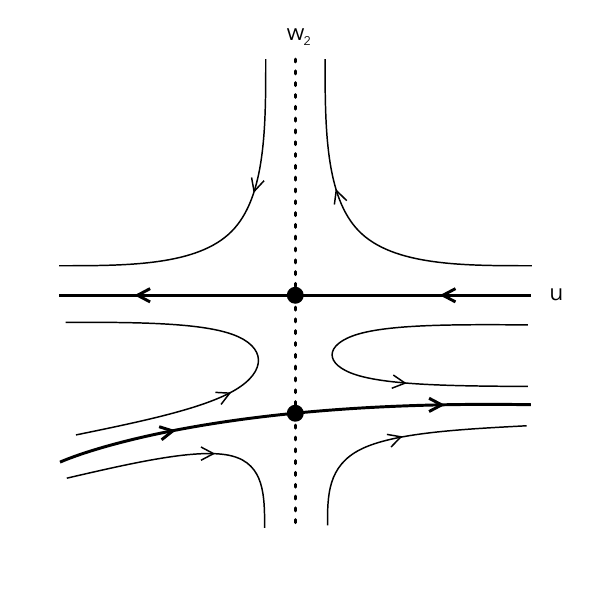}\\
\caption*{(b)}
\end{subfigure}
\begin{subfigure}[h]{3cm}
\centering
\includegraphics[width=3cm]{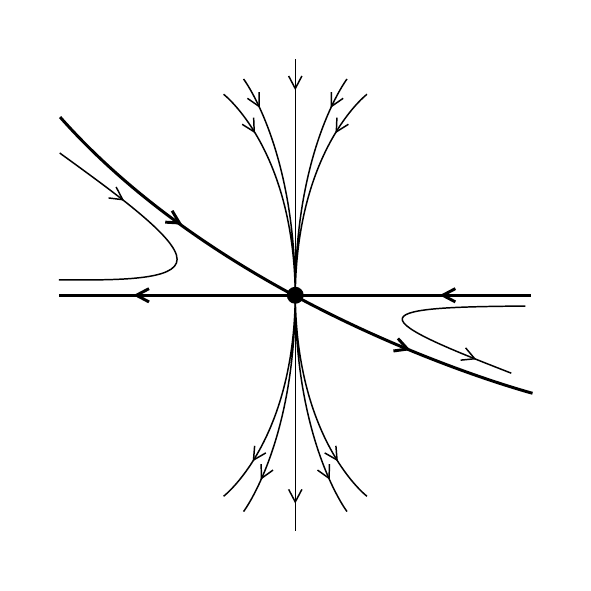}
\caption*{(c)}
\end{subfigure}
\begin{subfigure}[h]{3cm}
\centering
\includegraphics[width=3cm]{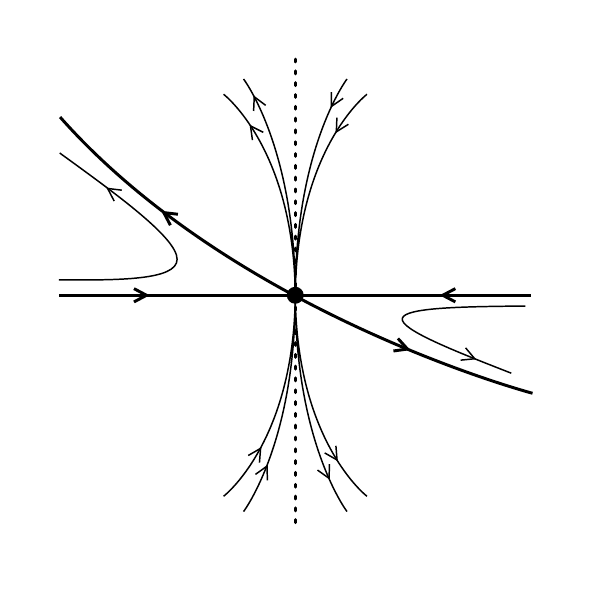}
\caption*{(d)}
\end{subfigure}
\caption{Desingularization of the origin of system \eqref{systemU1} with $c_1\neq0$. Nondicritical case (1.1).}
\label{fig:blowup_c1_1.1}
\end{figure}
	
Finally we must go back to the $(u,v)$-plane, swapping the second and the third quadrants and contracting the exceptional divisor to the origin. The orbits tending to the origin in forward or backward time, became orbits tending to the origin in forward or backward time with slope zero, i.e. tangent to the $u$-axis. According to the expressions $\dot{u}\mid_{v=0}=c_1\mu v^2 -a_0 v^3$ and $\dot{u}\mid_{v=0}=c_2(\mu+1)u^3$, which determine the sense of the flow along the axes, we get the local phase portrait at the origin for system \eqref{systemU1} given in Figure \ref{fig:localppO1}(L1).

\vspace{0.1cm}
	
Subcase (1.2). If we maintain $\mu<-2$ but take $c_2<0$, the reasoning is essentially similar to the one we have given in the previous case, and we obtain the phase portrait (L2) of Figure \ref{fig:localppO1}.

\vspace{0.1cm}

Subcase (1.3). Let $\mu>-1$  and $c_2>0$. This determines the position of the singular point $Q_1$ and the sense of the flow along the axes, so around the $w_2$-axis we obtain the phase portrait given in Figure \ref{fig:blowup_c1_1.3}(a).
	
As in the previous subcase we multiply by $u$ obtaining the phase portrait given in Figure \ref{fig:blowup_c1_1.3}(b), as the orbits in the second and third quadrants change their orientation and all the point in the $w_2$-axis become singular points.
	
In order to undo the variable change we analyze the sense of the flow along the axes according to the expression $\dot{u}\mid_{w_1=0}= c_2(\mu+1)u^2$, which determines that the flow goes in the positive sense of the $u$-axis, and $\dot{w_1}\mid_{u=0}=-c_1w_1^2$ which determines that the flow goes in the negative sense of the $w_1$-axis. Moreover we swap the second and third quadrants, and press the exceptional divisor into the origin, modifying the orbits. We obtain the phase portrait given in Figure \ref{fig:blowup_c1_1.3}(c). Multiplying again by $u$ we obtain the phase portrait \ref{fig:blowup_c1_1.3}(d).
	
Now we have to undo de second variable change. We note that $\dot{u}\mid_{v=0}=c_2(\mu+1)u^3$, so the flow gets away from the origin along the $u$-axis, nevertheless the sense of the flow along the $v$-axis is not determined by $\dot{v}\mid_{u=0}=c_1\mu v^2-a_0v^3$, it depends on the constant $\mu$. If $\mu>0$ the flow goes in the positive sense of the $v$-axis, if $\mu<0$ in the opposite sense and, if $\mu=0$ the flow goes to the origin. Thus we must distinguish three subcases and in each of them, modifying the orbits properly, we obtain, respectively, the phase portraits given in Figure \ref{fig:localppO1}(L3), (L4) and (L5).
	
	\vspace{0.1cm}
Subcase (1.4). Let $\mu>-1$ and $c_2<0$. By a similar reasoning to the previous one, we obtain the phase portraits (L6) and (L7) of Figure \ref{fig:localppO1}.
	
\begin{figure}[H]
\centering
\begin{subfigure}[h]{3cm}
\centering
\includegraphics[width=3cm]{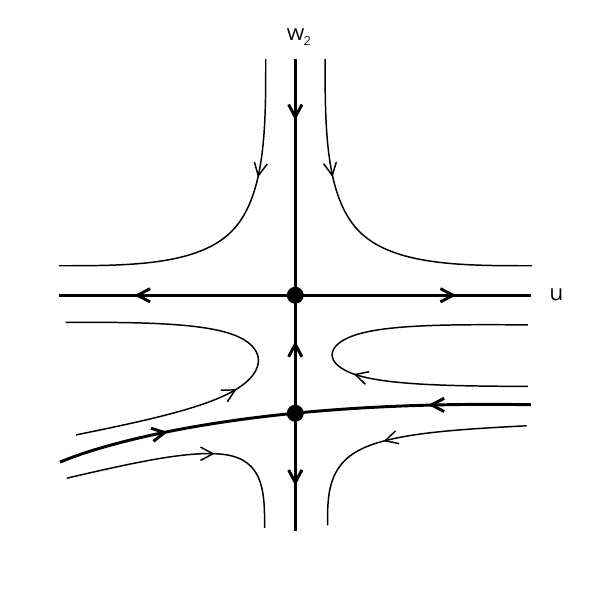}
\caption*{(a)}
\end{subfigure}
\begin{subfigure}[h]{3cm}
\centering
\includegraphics[width=3cm]{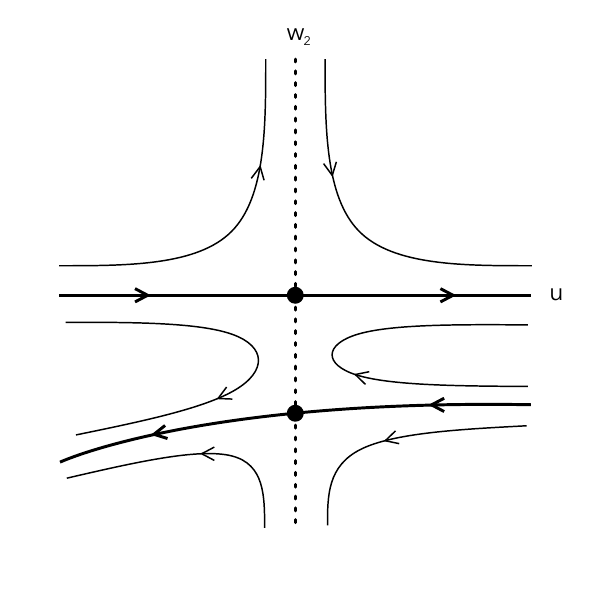}\\
\caption*{(b)}
\end{subfigure}
\begin{subfigure}[h]{3cm}
\centering
\includegraphics[width=3cm]{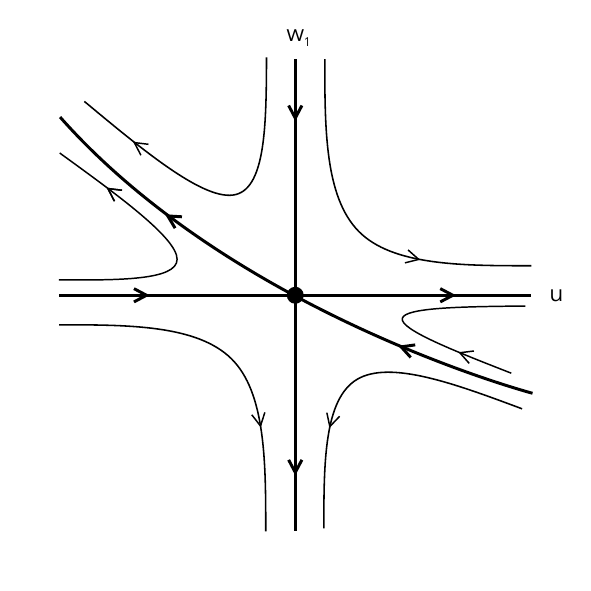}
\caption*{(c)}
\end{subfigure}
\begin{subfigure}[h]{3cm}
\centering
\includegraphics[width=3cm]{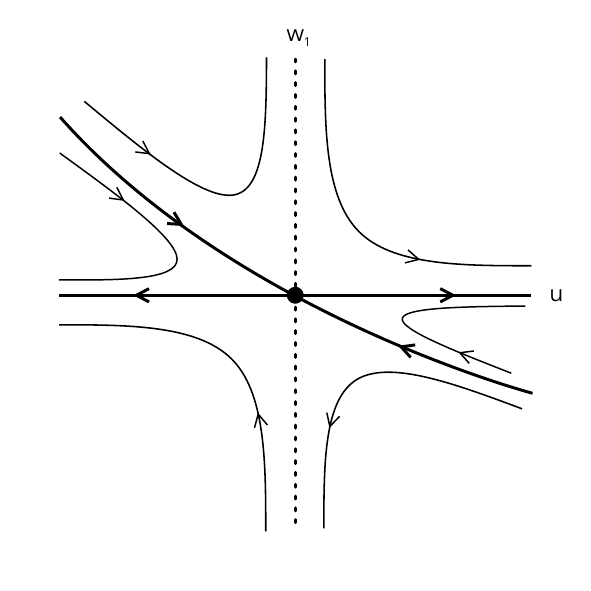}
\caption*{(d)}
\end{subfigure}
\caption{Desingularization of the origin of system \eqref{systemU1} with $c_1\neq0$. Nondicritical case (1.3).}
\label{fig:blowup_c1_1.3}
\end{figure}

\item  If $\mu \in (-\infty,-2)\cup (-1,+\infty)$, $c_2(\mu+2)>0$ and $(\mu+2)(a_0+c_0\mu)<0$, then $Q_0$ is a saddle and $Q_1$ a topological unstable node. We must distinguish two cases according with the sign of $c_2$.
	
\vspace{0.1cm}
Subcase (2.1). We consider $c_2<0$ so $\mu<-2$ and $a_0+c_0\mu>0$. The singular point $Q_1$ is on the positive $w_2$-axis and it is an unstable node, so the sense of the flow along the axes is determined, and we obtain the phase portrait given in Figure \ref{fig:blowup_c1_2_1}(a). Multiplying by $u$ we obtain Figure \ref{fig:blowup_c1_2_1}(b).
	
We see that for system \eqref{sis_blowup3} the flow goes in the negative sense along the $w_1$-axis and in the positive sense along the $u$-axis, according to the expressions $\dot{u}\mid_{w_1=0}=c_2(\mu+1)u^2 $ and $\dot{w_1}\mid_{u=0}=-c_1w_1^2$. We undo the variable change modifying the orbits properly, and we note that it must exist an hyperbolic or elliptic sector in both first and third quadrants, thus it can appear the configuration given in Figure \ref{fig:blowup_c1_2_1}(c) or the one given in Figure \ref{fig:blowup_c1_2_1_b}(a). From the first of them multiplying by $u$ we obtain \ref{fig:blowup_c1_2_1}(d), and if we undo the variable change in a similarly way than in the previous cases, we get the phase portrait in Figure \ref{fig:localppO1}(L8).
	
If we consider hyperbolic sectors we continue undoing the blow up from  Figure \ref{fig:blowup_c1_2_1_b}(a), obtaining successively the phase portraits \ref{fig:blowup_c1_2_1_b}(b) and \ref{fig:blowup_c1_2_1_b}(c).
However in our study we have proved, by means of the index theory, that only the case with elliptic sectors is feasible in the global phase portraits obtained. More detailed explanations will be given in Section \ref{sec:global} but, roughly speaking we know that the index of the vector field on the sphere must be 2, and this index is the sum of the indices of all singularities, which depend on the sectors that they have, so if the index is 2 considering two elliptic sectors in a particular singular point, it cannot be 2 if we change those sectors for hyperbolic ones. In conclusion the only phase portrait that will appear in this case is (L8) of Figure \ref{fig:localppO1}.
	
\begin{figure}
\centering
\begin{subfigure}[h]{3cm}
\centering
\includegraphics[width=3cm]{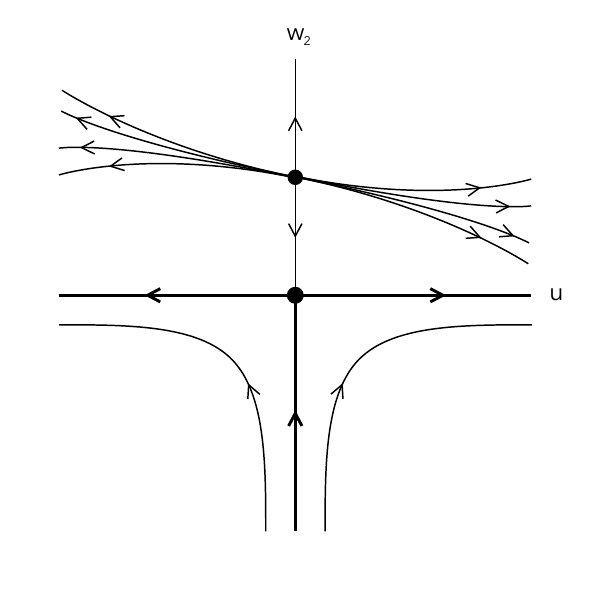}
\caption*{(a)}
\end{subfigure}
\begin{subfigure}[h]{3cm}
\centering
\includegraphics[width=3cm]{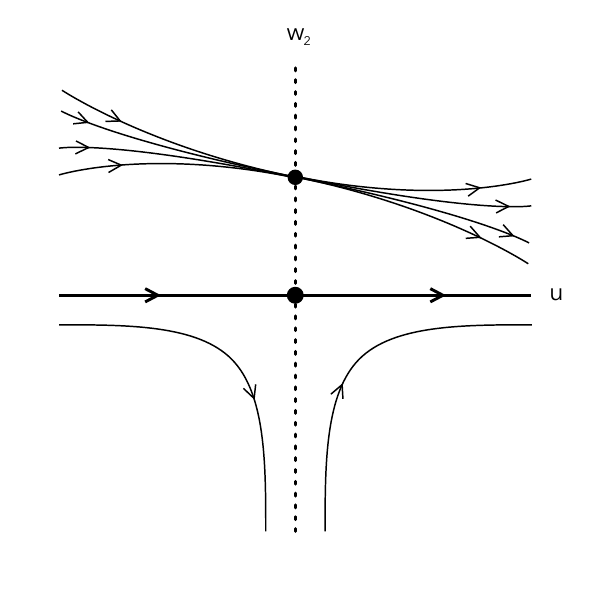}\\
\caption*{(b)}
\end{subfigure}
\begin{subfigure}[h]{3cm}
\centering
\includegraphics[width=3cm]{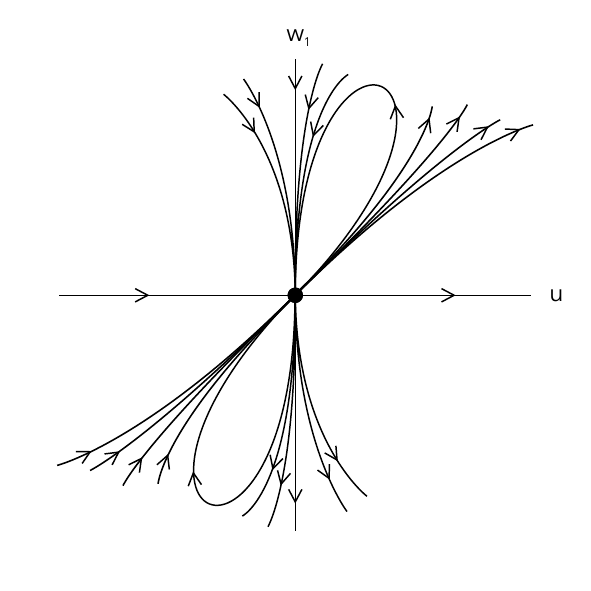}
\caption*{(c)}
\end{subfigure}
\begin{subfigure}[h]{3cm}
\centering
\includegraphics[width=3cm]{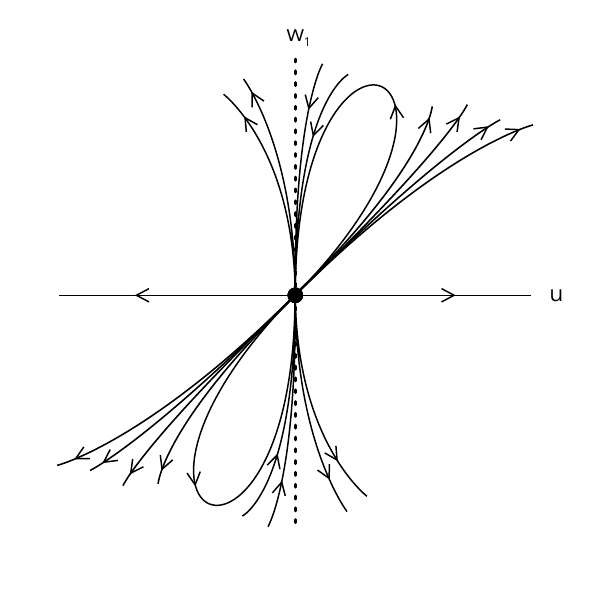}
\caption*{(d)}
\end{subfigure}
\caption{Desingularization of the origin of system \eqref{systemU1} with $c_1\neq0$. Nondicritical case (2.1).}
\label{fig:blowup_c1_2_1}
\end{figure}
	
\begin{figure}
\centering
\begin{subfigure}[h]{3cm}
\centering
\includegraphics[width=3cm]{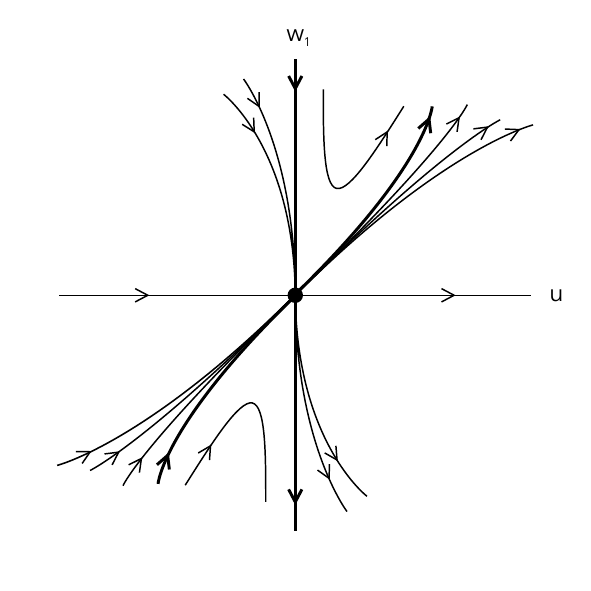}
\caption*{(c)}
\end{subfigure}
\begin{subfigure}[h]{3cm}
\centering
\includegraphics[width=3cm]{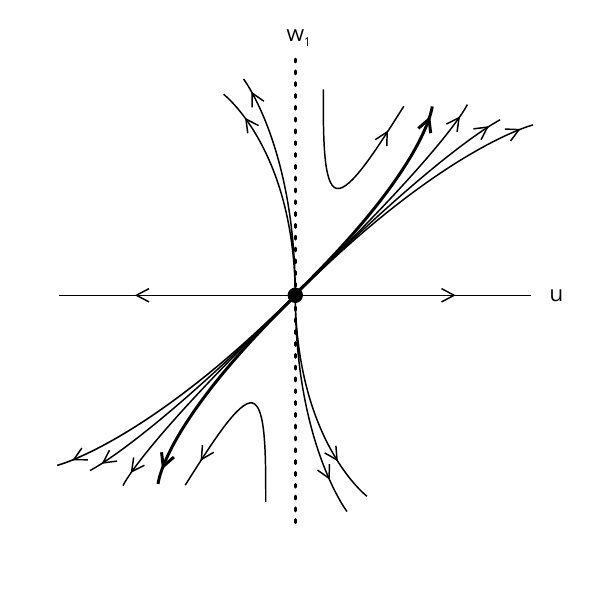}
\caption*{(d)}
\end{subfigure}
\begin{subfigure}[h]{3cm}
\centering
\includegraphics[width=3cm]{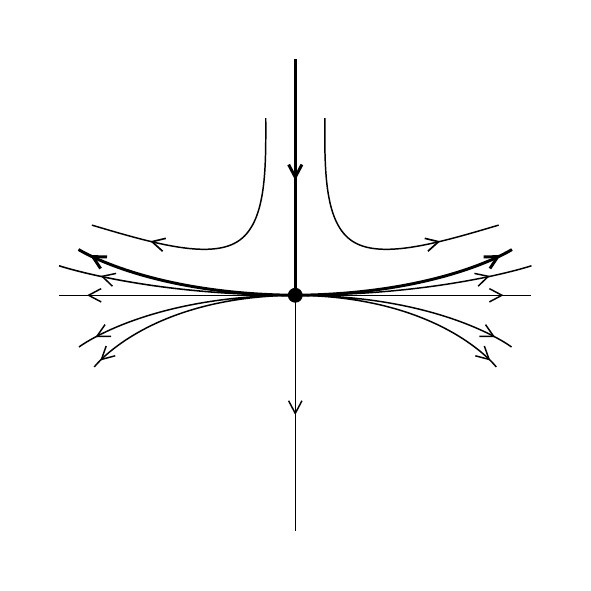}\\
\caption*{(b)}
\end{subfigure}
\caption{Desingularization of the origin of system \eqref{systemU1} with $c_1\neq0$. Alternative to nondicritical case (2.1).}
\label{fig:blowup_c1_2_1_b}
\end{figure}
	
From now on we will omit the reasonings about how to undo the variable changes for obtaining the final phase portrait, because they are similar to the ones of the previous cases. The results obtained are the following.

\vspace{0.1cm}
Subcase (2.2). Let $c_2>0$ so $\mu>-1$ and $a_0+c_0\mu<0$. We obtain the phase portraits (L9) and (L10) of Figure \ref{fig:localppO1}.  In (L9) it is possible to consider hyperbolic sectors instead of the elliptic ones, but applying index theory to the global phase portraits obtained in our study, we note that only the phase portrait with elliptic sectors is feasible.

\vspace{0.1cm}
\item  If $\mu \in (-\infty,-2)\cup (-1,+\infty)$, $c_2(\mu+2)<0$ and $(\mu+2)(a_0+c_0\mu)>0$, then $Q_0$ is a saddle and $Q_1$ a topological stable node. If $c_2>0$ we obtain the phase portrait (L11) of Figure \ref{fig:localppO1}, and if $c_2<0$ we obtain the phase portraits (L12), (L13) and (L14) of Figure \ref{fig:localppO1}. In (L11) and (L12) it would be possible that the elliptic sectors appearing were hyperbolic sectors, but again we have proved that only the elliptic option is feasible according to the index theory.

\vspace{0.1cm}
\item If $c_2>0$, $\mu\in (-2,-1)$ and $a_0+c_0\mu>0$, then $Q_0$ is a stable node and $Q_1$ a topological saddle.  We obtain the phase portrait (L15) of Figure \ref{fig:localppO1}.

\vspace{0.1cm}	
\item If $c_2>0$, $\mu\in (-2,-1)$ and $a_0+c_0\mu<0$, then $Q_0$ is a stable node and $Q_1$ a topological unstable node. We obtain the phase portrait (L11) of Figure \ref{fig:localppO1}.

\vspace{0.1cm}
\item  If $c_2<0$, $\mu\in (-2,-1)$ and $a_0+c_0\mu<0$, then $Q_0$ is an  unstable node and $Q_1$ a topological saddle. We obtain the phase portrait (L16) of Figure \ref{fig:localppO1}.
	
\vspace{0.1cm}
\item If $c_2<0$, $\mu\in (-2,-1)$ and $a_0+c_0\mu>0$, then $Q_0$ is an unstable node and $Q_1$ a topological stable node.  We obtain the phase portrait (L8) of Figure \ref{fig:localppO1}.
\end{enumerate}

\subsubsection{Dicritical case}

Now we must study the singular points on the exceptional divisor of system \eqref{sis_blowup6}. In this case there is only one singular point, $R=\left(0, -c_2/c_1\right) $ which is non-degenerated. We shall distinguish several subcases.

\begin{enumerate}[$\bullet$]

\item If $c_3^2<-4c_2(a_0-2c_0)$ and $c_2c_3<0$, then $P$ is a stable focus. We shall distinguish two cases depending on the sign of the parameter $ c_2$, because it determines if the singular point is on the positive or negative $u$-axis. We consider $c_2>0$. In Figure \ref{fig:blowup_c1_dicritical_1} the blowing-down process is represented.  The phase portrait around the $u$-axis is the one given in Figure 	\ref{fig:blowup_c1_dicritical_1}(a), multiplying by $u^2$ we obtain (b), undoing the second variable change we obtain (c), multiplying by $u$ we get (d) and finally, undoing the first variable change we get the phase portrait (L11) of Figure \ref{fig:localppO1}. Taking $c_2<0$ and by the same method we obtain the phase portrait (L8) of Figure \ref{fig:localppO1}.
	
	\begin{figure}[h]
		\centering
		\begin{subfigure}[h]{3cm}
			\centering
			\includegraphics[width=3cm]{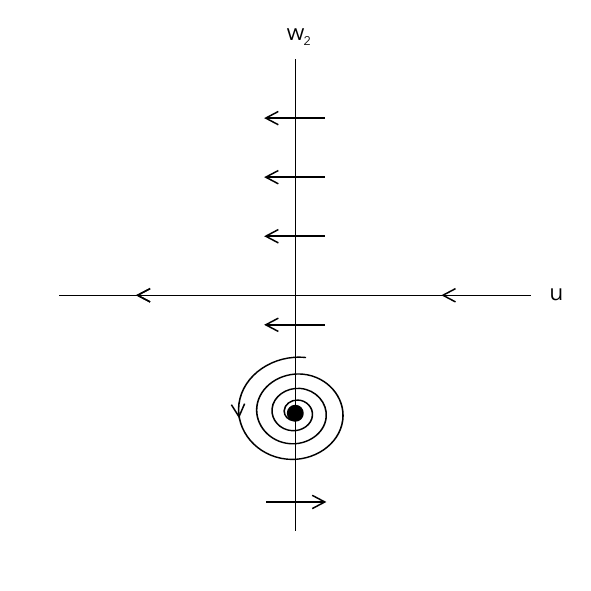}
			\caption*{(a)}
		\end{subfigure}
		\begin{subfigure}[h]{3cm}
			\centering
			\includegraphics[width=3cm]{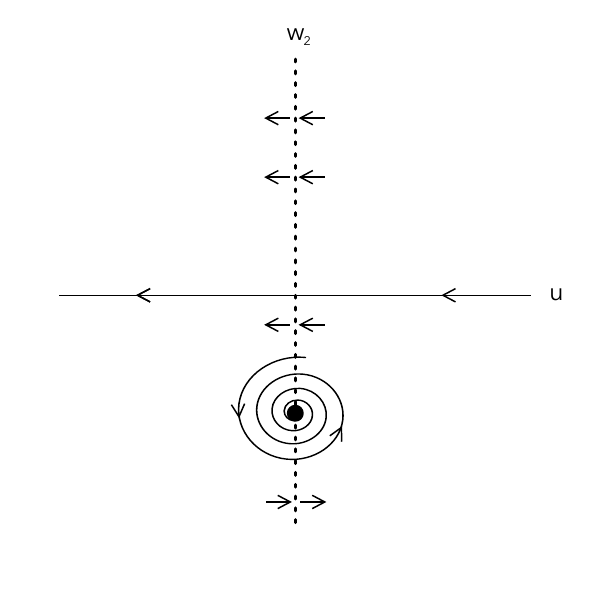}\\
			\caption*{(b)}
		\end{subfigure}
		\begin{subfigure}[h]{3cm}
			\centering
			\includegraphics[width=3cm]{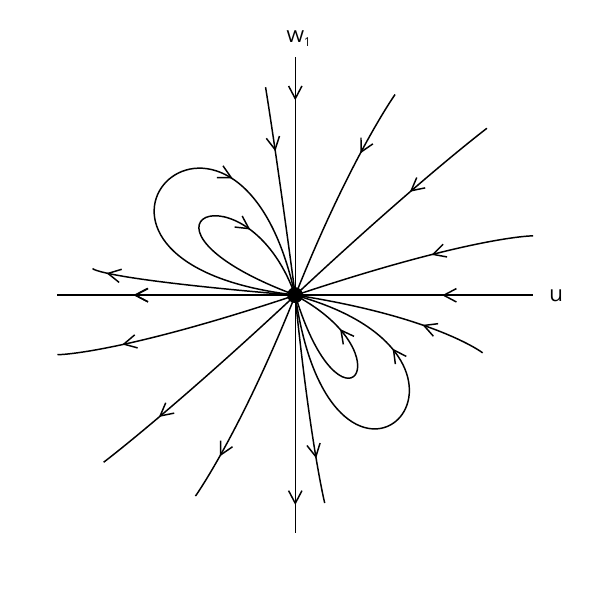}
			\caption*{(c)}
		\end{subfigure}
		\begin{subfigure}[h]{3cm}
			\centering
			\includegraphics[width=3cm]{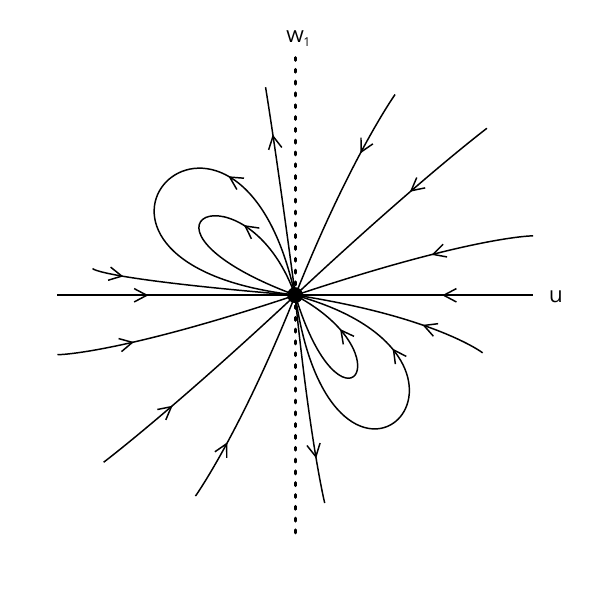}
			\caption*{(d)}
		\end{subfigure}
		\caption{Desingularization of the origin of system \eqref{systemU1} with $c_1\neq0$. Dicritical case (1), $c_2>0$.}
		\label{fig:blowup_c1_dicritical_1}
	\end{figure}
	
From now on we omit explanations in cases in which similar arguments are valid, and same results are obtained. In order to simplify the notation we define $\beta=\sqrt{c_3^2+4c_2(a_0-2c_0)}$.
	
\vspace{0.2cm}

\item If $c_3^2<-4c_2(a_0-2c_0)$ and $c_2c_3>0$, then $P$ is an unstable focus. The reasoning is analogous to the one of the previous case and we obtain the same phase portraits: (L11) if $c_2>0$, and  (L8) if $c_2<0$.

\vspace{0.2cm}

\item If $c_3^2=-4c_2(a_0-2c_0)$ and $c_2c_3<0$ or if $c_3^2>-4c_2(a_0-2c_0)$, $c_2(c_3-\beta)<0$ and $c_2(c_3+\beta)<0$, then $P$ is a stable node. If $c_3^2=-4c_2(a_0-2c_0)$ and $c_2c_3>0$ or if $c_3^2>-4c_2(a_0-2c_0)$, $c_2(c_3-\beta)>0$ and $c_2(c_3+\beta)>0$, then $P$ is an unstable node. In both cases the phase portrait obtained is again (L11) if $c_2>0$, and (L8) if $c_2<0$.

\vspace{0.2cm}

\item If $c_3=0$ and $c_2(a_0-2c_0)<0$ then $P$ is a linear center. In this case the singular point $P$ could be a center or a focus, but the final phase portrait obtained when $P$ is a center is the same as the one we obtained previously for the case with a focus, so the result is (L11) if $c_2>0$, and (L8) if $c_2<0$.

\vspace{0.2cm}

\item If $c_3^2>-4c_2(a_0-2c_0)$ and $(c_3-\beta)(c_3+\beta)<0$, or if $c_3=0$ and $c_2(a_0-2c_0)>0$, then $P$ is a saddle. The blowing-down considering $c_2>0$ is represented in Figure \ref{fig:blowup_c1_dicritical_5}. The final result is (L17)  of Figure \ref{fig:localppO1}. If we take $c_2<0$ we obtain the phase portrait (L18).
	
	\begin{figure}[h]
		\centering
		\begin{subfigure}[h]{3cm}
			\centering
			\includegraphics[width=3cm]{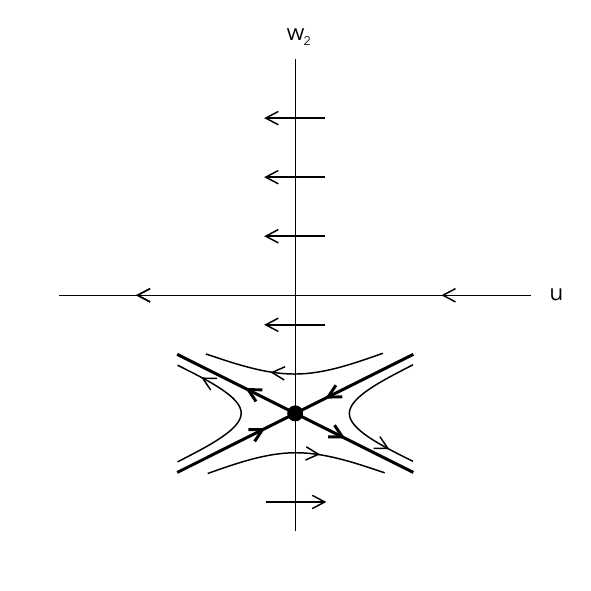}
			\caption*{System \eqref{sis_blowup6}}
		\end{subfigure}
		\begin{subfigure}[h]{3cm}
			\centering
			\includegraphics[width=3cm]{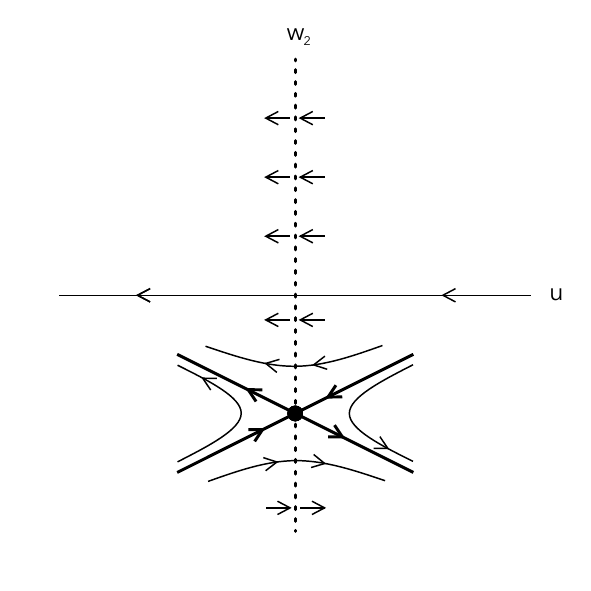}\\
			\caption*{System \eqref{sis_blowup4}}
		\end{subfigure}
		\begin{subfigure}[h]{3cm}
			\centering
			\includegraphics[width=3cm]{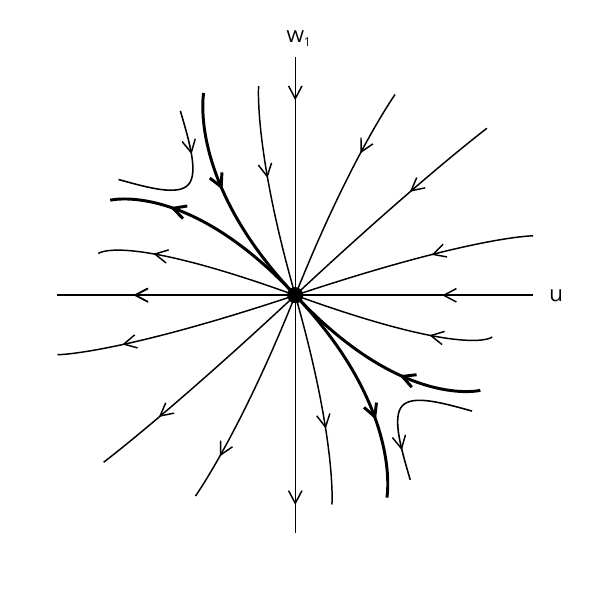}
			\caption*{System \eqref{sis_blowup3}}
		\end{subfigure}
		\begin{subfigure}[h]{3cm}
			\centering
			\includegraphics[width=3cm]{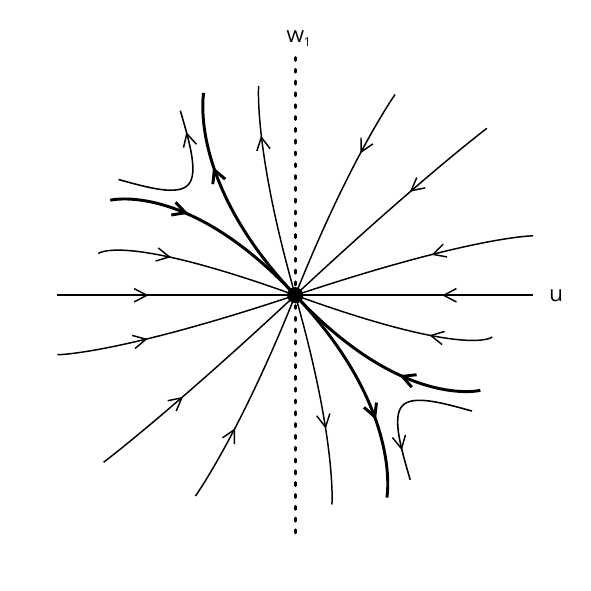}
			\caption*{System \eqref{sis_blowup2}}
		\end{subfigure}
		\caption{Desingularization of the origin of system \eqref{systemU1} with $c_1\neq0$. Dicritical case (5), $c_2>0$.}
		\label{fig:blowup_c1_dicritical_5}
	\end{figure}

\end{enumerate}

\subsection{Case $c_1$ zero}\label{subsec:O1_c1zero}

We consider system \eqref{systemU1} and do the same variable change that we did in the case with $c_1\neq0$, the result is obviously system \eqref{sis_blowup2} but taking $c_1=0$, i.e.
\begin{equation}\label{sis_blowup2*}
\begin{split}
\dot{u}&=(c_0-a_0)u^3w_1^2+c_3(\mu+1)u^3w_1+c_2(\mu+1)u^3,\\
\dot{w_1}&=-c_0u^2w_1^3-c_3u^2w_1^2-c_2u^2w_1.
\end{split}
\end{equation}
In this case we can cancel a common factor $u^2$ getting the system
\begin{equation}\label{sis_blowup3*}
\begin{split}
\dot{u}&=(c_0-a_0)uw_1^2+c_3(\mu+1)uw_1+c_2(\mu+1)u,\\
\dot{w_1}&=-c_0w_1^3-c_3w_1^2-c_2w_1,
\end{split}
\end{equation}
for which we must study the singular points on the exceptional divisor, i.e. on  the straight line $u=0$.

The origin $S_0=\left( 0,0\right)$ is always a singular point. The other singular points on this line are those for which $w_1$ is a solution of $c_0w_1^2+c_3w_1+c_2=0$. If $c_0\neq0$ and $c_3^2>4c_0c_2$ then $S_1=\left( 0,-(R_c+c_3)/(2c_0)\right)$  and $S_2=\left( 0,(R_c-c_3)/(2c_0)\right)$ are singular points. If $c_0\neq0$ and $c_3^2=4c_0c_2$, then $S_3=\left( 0,-c_3/2c_0\right)$ is a singular point, and finally, if $c_0$ and $c_3$ are non-zero, $S_4=\left( 0,-c_2/c_3\right) $ is a singular point.

\begin{table}[h]
	\begin{center}
		\begin{tabular}{|cll|}
			\hline
			\textbf{Case} & \textbf{Conditions} & \textbf{Singular points} \\
			\hline
			\hline
			A& $c_0=0$, $c_3=0$. & $S_0$. \\
			\hline
			B& $c_0=0$, $c_3\neq0$. & $S_0$, $S_4$. \\
			\hline
			C& $c_0\neq0$, $c_3^2<4c_0c_2$. & $S_0$. \\
			\hline
			D& $c_0\neq0$, $c_3^2=4c_0c_2$. & $S_0$, $S_3$.  \\
			\hline
			E& $c_0\neq0$, $c_3^2>4c_0c_2$.  & $S_0$, $S_1$, $S_2$. \\
			\hline
		\end{tabular}
		\caption{Cases with the singular points on the exceptional divisor of system \eqref{sis_blowup3*}.}
		\label{tab:cases_blowup_c1zero}
	\end{center}
\end{table}

In summary we shall study the five cases given in Table \ref{tab:cases_blowup_c1zero}. For doing this we study separately the local phase portrait of each singular point assuming in each case the necessary hypothesis for its existence.

The singular points $S_0$, $S_1$ and $S_2$ are hyperbolic. $S_0$  is a saddle if  $\mu>-1$, a stable node if  $c_2>0$ and $\mu<-1$, and an unstable node  if $c_2<0$ and $\mu<-1$.
$S_1$ is a saddle if $c_0(a_0+c_0\mu)<0$, a stable node if $c_0>0$ and $(a_0+c_0\mu)>0$, and an unstable node if $c_0<0$ and $(a_0+c_0\mu)<0$. $S_2$ is a saddle if $c_0(a_0+c_0\mu) (R_c-c_3)<0$, a stable node if $c_0(R_c-c_3)>0$ and $(a_0+c_0\mu)>0$, and an unstable node if $c_0(R_c-c_3)<0$ and $(a_0+c_0\mu)<0$.

The singular point $S_3$ is a semi-hyperbolic saddle-node and finally, $S_4$ is a hyperbolic saddle if $c_2>0$, and a hyperbolic stable node if $c_2<0$.	

Using these informations we study the next cases from the five given in Table \ref{tab:cases_blowup_c1zero}. 

\vspace{0.2cm}
First of all we study case (A) in which the only singular point is the origin, so we have the next three possibilities.

If $c_0=c_3=0$ and $\mu>-1$, then $S_0$ is a saddle. In order to determine the phase portrait around the $w_1$-axis for system \ref{sis_blowup3*}, we must fix the sign of $c_2$, which determines the sense of the flow along the axes. Considering $c_2>0$ we get the phase portrait given in Figure \eqref{fig:localppO1}(L19), and with $c_2<0$ we obtain the phase portrait (L20).
	
If $c_0=c_3=0$, $\mu<-1$ and $c_2>0$, then $S_0$ is a stable node, and we obtain the phase portrait (L21) of Figure \eqref{fig:localppO1}.
	
If $c_0=c_3=0$, $\mu<-1$ and $c_2<0$, then $S_0$ is an unstable node, and we get the phase portrait (L22) of Figure \eqref{fig:localppO1}.

\vspace{0.2cm}
	
In case (B), fixed the phase portrait of $S_4$, only two phase portraits will be possible for the origin, as the sign of $c_2$ is determined, and so we get the four following cases.

If $c_0=0$, $c_3\neq0$, $\mu>-1$ and $c_2>0$, then $S_0$ and $S_4$ are both saddle points and from the blowing-down we obtain the phase portrait (L23) of Figure \eqref{fig:localppO1}.
	
If  $c_0=0$, $c_3\neq0$, $\mu>-1$ and $c_2<0$, then $S_0$ is a saddle and $S_4$ a stable node. We obtain the phase portrait (L24) of Figure \eqref{fig:localppO1}.
	
If  $c_0=0$, $c_3\neq0$, $\mu<-1$ and $c_2>0$, then $S_0$ is a stable node and $S_4$ a saddle. We obtain the phase portrait (L25) of Figure \eqref{fig:localppO1}.
	
If  $c_0=0$, $c_3\neq0$, $\mu<-1$ and $c_2<0$, then $S_0$ is an unstable node and $S_4$ a stable node. We obtain the phase portrait (L26) of Figure \eqref{fig:localppO1}.

\vspace{0.2cm}
Again in case (C) the only singular point is the origin so we distinguish three cases, and obtain the same local phase portrait that in case (A), but under different conditions. If $c_0\neq0$, $c_3^2<4c_0c_2$ and $\mu>-1$, then $S_0$ is a saddle. Attending to the sign of $c_2$, which determines the sense of the flow on the axes, we consider the following cases: if $c_2>0$ we obtain the phase portrait (L19) of Figure \eqref{fig:localppO1}, and if $c_2<0$ we obtain the phase portrait (L20) of Figure \eqref{fig:localppO1}.
	
If $c_0\neq0$, $c_3^2<4c_0c_2$, $\mu<-1$ and $c_2>0$, then $S_0$ is a stable node. We obtain the phase portrait (L21) of Figure \eqref{fig:localppO1}.
	
If $c_0\neq0$, $c_3^2<4c_0c_2$, $\mu<-1$ and $c_2<0$, then $S_0$ is an unstable node. We obtain the phase portrait (L22) of Figure \eqref{fig:localppO1}.

\vspace{0.2cm}	
In case (D) apart from the origin, there exists the singular point $S_3$, which is always a saddle node, so again we get only three cases.

If $c_0\neq0$, $c_3^2=4c_0c_2$ and $\mu>-1$, then $S_0$ is a saddle and $S_3$ a saddle-node. We must distinguish four subcases according to the signs of $c_0$ and $a_0+c_0\mu$, which determine the position of the saddle-node $S_3$ and its sectors.

If $c_0>0$ and $a_0+c_0\mu>0$ we obtain the phase portrait (L27) of Figure \eqref{fig:localppO1}, if $c_0>0$ and $a_0+c_0\mu<0$ we get the phase portrait (L28) of Figure \eqref{fig:localppO1}, if $c_0<0$ and $a_0+c_0\mu>0$ we have the phase portrait (L29), and if $c_0<0$ and $a_0+c_0\mu<0$ the phase portrait (L30).

If $c_0\neq0$, $c_3^2=4c_0c_2$, $\mu<-1$ and $c_2>0$, then $S_0$ is a stable node and $S_3$ a saddle-node. We distinguish two subcases setting the sign of $a_0+c_0\mu$ which determines the position of the sectors of the saddle-node $S_3$. If $a_0+c_0\mu>0$ we obtain the phase portrait (L31) of Figure \eqref{fig:localppO1}, and if $a_0+c_0\mu<0$ we obtain the phase portrait (L32) of Figure \eqref{fig:localppO1}.
	
If $c_0\neq0$, $c_3^2=4c_0c_2$, $\mu<-1$ and $c_2<0$, then $S_0$ is an unstable node and $S_3$ a saddle-node. The only possibility is that $a_0+c_0\mu>0$, and we obtain the phase portrait (L33) of Figure \eqref{fig:localppO1}.

\vspace{0.2cm}

In case (E) there exist three singular points, with three possible phase portraits for each of them, however, many of the combinations are not possible, and only 13 cases will be feasible.

First, due to the conditions which define the local phase portrait in each singular point, it is obvious that if $S_1$ is a stable node, then $S_2$ cannot be an unstable node, and if $S_1$ is an unstable node, $S_2$ cannot be a stable node, due to the sign of $a_0+c_0\mu$.

If $S_0$ and $S_2$ were stable nodes and $S_1$ a saddle, the conditions $c_2>0$, $R_c-c_3<0$, and $c_0<0$ will hold. Squaring both terms in the condition $R_c<c_3$ we obtain $c_3^2-4c_0c_2<c_3^2$, and then $c_0c_2>0$, which is a contradiction. The same reasoning is valid in the next two cases.

If $S_0$ and $S_2$ are unstable nodes and $S_1$ a saddle, then the conditions $c_2<0$, $R_c-c_3<0$ and $c_0>0$ hold, and if $S_0$ is an unstable node, $S_1$ a stable node and $S_2$ a saddle, then the same three conditions hold.

If $S_0$, $S_1$ and $S_2$ are stable nodes, the conditions $c_2>0$, $R_c-c_3>0$ and $c_0>0$ hold. Now we take condition $R_c<c_3$ and squaring both terms we obtain $c_3^2-4c_0c_2<c_3^2$, and then $c_0c_2>0$, which is a contradiction.

If $S_0$ is a stable node and $S_1$ an unstable node, the conditions $\mu<-1$, $c_0<0$ and $a_0+c_0\mu<0$ hold. Then according to the signs of $c_0$ and $\mu$ which are fixed,  $a_0<-c_0\mu<0$ which contradicts the hypothesis $(H_2)$. The same reasoning is valid if $S_0$ and $S_1$ are unstable nodes, because the same conditions hold. Now we will study the feasible cases.

\begin{enumerate}
\item[(E1)] If $c_0\neq0$, $c_3^2>4c_0c_2$, $\mu>-1$, $c_0(a_0+c_0\mu)(2c_0c_2- c_3^2-c_3R_c)>0$ and $c_0(a_0+c_0\mu)(R_c-c_3)(2c_0c_2-c_3^2+c_3R_c)>0$, then $S_0$, $S_1$ and $S_2$ are saddles. We must distinguish two subcases depending on the position of the singular points $S_1$ and $S_2$ on the $w_1$-axis. First if $S_1$ is on the negative $w_1$-axis and $S_2$ on the positive $w_1$-axis, that corresponds with conditions $c_0>0$, $R_c-c_3>0$ and $c_2<0$, we obtain the phase portrait (L34) of Figure \ref{fig:localppO1}. Note that if $R_c-c_3>0$, the singular points $S_1$ and $S_2$ are one on the positive part of the axis and the other on the negative part, but in any case the absolute value of the second coordinate of $S_1$ is greater or equal than the absolute value of the second coordinate of $S_2$, and this determines the relation between the slopes of orbits in the phase portraits. Conversely if we have $S_1$ on the positive $w_1$-axis and $S_2$ on the negative one, i.e. under the conditions $c_0<0$, $R_c-c_3>0$ and $c_2>0$, we obtain the phase portrait (L35) of Figure \eqref{fig:localppO1}.

\vspace{0.1cm}	

\item[(E2)] If $c_0\neq0$, $c_3^2>4c_0c_2$, $\mu>-1$,  $c_0(a_0+c_0\mu)(2c_0c_2- c_3^2-c_3R_c)>0$, $c_0(R_c-c_3)>0$ and  $(a_0+c_0\mu)(2c_0c_2-c_3^2+c_3R_c)<0$, then $S_0$ and $S_1$ are saddles and $S_2$ is a stable node. If $c_0>0$ then $R_c-c_3>0$ and so $c_3^2-4c_0c_2>c_3^2$ and $c_0c_2<0$. If $a_0+c_0\mu>0$ then $2c_0c_2-c_3^2-c_3R_c>0$ which is not possible because $2c_0c_2<0$ and we subtract two positive terms. Conversely if $a_0+c_0\mu<0$ then $2c_0c_2-c_3^2<c_3R_c$ and $2c_0c_2-c_3^2>-c_3R_c$, so $\abs{2c_0c_2-c_3^2}<c_3R_c$. Squaring we get $4c_0^2c_2^2<0$ which is not possible. If $c_0<0$ then $R_c-c_3<0$ and we deduce $c_2<0$. If $a_0+c_0\mu<0$ then $2c_0c_2-c_3^2-c_3R_c>0$, but $c_3^2-2c_0c_2>c_3^2-4c_0c_2>0$ so $2c_0c_2-c_3^2<0$ and subtracting $c_3R_c>0$ the result cannot be positive. In conclusion we deduce that $c_0,c_2<0$ and $a_0+c_0\mu>0$. Hence we have $-(R_c+c_3)/(2c_0)>(R_c-c_3)/(2c_0)>0$. This determines the only possible position of the singular points which are both in the positive $w_1$-axis. Undoing the blow up we obtain the phase portrait (L36) of Figure \ref{fig:localppO1}.
	
	\vspace{0.1cm}	
	
\item[(E3)] If $c_0\neq0$, $c_3^2>4c_0c_2$, $\mu>-1$, $c_0(a_0+c_0\mu)(2c_0c_2- ^2- c_3R_c)>0$, $c_0(R_c-c_3)<0$ and $(a_0+c_0\mu)(2c_0c_2-c_3^2+c_3R_c)>0$, then $S_0$ and $S_1$ are saddles and $S_2$ is an unstable node. Therefore we deduce that $0>(R_c-c_3)/(2c_0)>-(R_c+c_3)/(2c_0)$, so both singular points are on the negative $w_1$ axis, $S_1$ under $S_2$. We obtain the phase portrait (L37) of Figure \eqref{fig:localppO1}.
	
	\vspace{0.1cm}	
	
\item[(E4)] If $c_0\neq0$, $c_3^2>4c_0c_2$, $\mu>-1$, $c_0>0$, $(a_0+c_0\mu) (2c_0c_2- c_3^2-c_3R_c)<0$ and $c_0(R_c-c_3)(a_0+c_0\mu)(2c_0c_2-c_3^2+c_3R_c)>0$, then $S_0$ and $S_2$ are saddles and $S_1$ is a stable node. Then we deduce that $-(R_c+c_3)/(2c_0)<(R_c-c_3)/(c_0)<0$, so both singular points are on the negative $w_1$ axis, $S_1$ under $S_2$. We obtain the phase portrait (L38) of Figure \eqref{fig:localppO1}.
	
	\vspace{0.1cm}	
	
\item[(E5)]  If $c_0\neq0$, $c_3^2>4c_0c_2$, $\mu>-1$, $c_0>0$, $(a_0+c_0\mu) (2c_0c_2-c_3^2 -c_3R_c)<0$, $c_0(R_c-c_3)>0$ and  $(a_0+c_0\mu)(2c_0c_2-c_3^2+ )<0$, then $S_0$ is a saddle and $S_1$ and $S_2$ are stable nodes. We obtain the phase portrait (L45) of Figure \eqref{fig:localppO1}.
	
	\vspace{0.1cm}	
	
\item[(E6)] If $c_0\neq0$, $c_3^2>4c_0c_2$, $\mu>-1$, $c_0<0$, $(a_0+c_0\mu) (2c_0c_2- c_3^2-c_3R_c)>0$ and $(a_0+c_0\mu)(R_c-c_3)(2c_0c_2-c_3^2+c_3R_c)<0$, then $S_0$ and $S_2$ are saddles and $S_1$ is an unstable node. Hence we deduce that $0<(R_c-c_3)/(2c_0)<-(R_c+c_3)/(2c_0)$, so both singular points are on the positive $w_1$ axis, $S_2$ under $S_1$. We obtain the phase portrait (L47) of Figure \eqref{fig:localppO1}.
	
	\vspace{0.1cm}	
	
\item[(E7)] If $c_0\neq0$, $c_3^2>4c_0c_2$, $\mu>-1$, $c_0<0$, $(a_0+c_0\mu) (2c_0c_2- c_3^2-c_3R_c)>0$, $R_c-c_3>0$ and $(a_0+c_0\mu)(2c_0c_2-c_3^2 +c_3R_c)>0$, then $S_0$ is a saddle and $S_1$ and $S_2$ are unstable nodes. We obtain the phase portrait (L46) of Figure \eqref{fig:localppO1}.

\vspace{0.1cm}	
	
\item[(E8)] If $c_0\neq0$, $c_3^2>4c_0c_2$,  $c_2>0$, $\mu<-1$, $c_0(a_0+c_0\mu) (2c_0c_2-c_3^2-c_3R_c)>0$ and $c_0(a_0+c_0\mu)(R_c-c_3)(2c_0c_2-c_3^2+c_3R_c)>0$, then $S_0$ is a stable node and $S_1$ and $S_2$ are saddles. We obtain the phase portrait (L39) of Figure \eqref{fig:localppO1}.
	
	\vspace{0.1cm}	
	
\item[(E9)] If $c_0\neq0$, $c_3^2>4c_0c_2$, $c_2>0$, $\mu<-1$, $c_0(a_0+c_0\mu) (2c_0c_2-c_3^2-c_3R_c)>0$, $c_0(R_c-c_3)<0$ and $(a_0+c_0\mu)(2c_0c_2-c_3^2 +c_3R_c)>0$, then $S_0$ is a stable node, $S_1$ is a saddle and $S_2$ is an unstable node. We obtain the phase portrait (L40) of Figure \eqref{fig:localppO1}.
	
	\vspace{0.1cm}	
	
\item[(E10)]  If $c_0\neq0$, $c_3^2>4c_0c_2$, $c_2>0$, $\mu<-1$, $c_0>0$, $(a_0+c_0\mu) (2c_0c_2-c_3^2-c_3R_c)<0$ and $(a_0+c_0\mu)(R_c-c_3) (2c_0c_2-c_3^2+ c_3R_c)>0$, then $S_0$ and $S_1$ are stable nodes and $S_2$ is a saddle. We obtain the phase portrait (L41) of Figure \eqref{fig:localppO1}.
	
	\vspace{0.1cm}	
	
\item[(E11)]  If $c_0\neq0$, $c_3^2>4c_0c_2$, $c_2<0$, $\mu<-1$, $c_0(a_0+c_0\mu) (2c_0c_2 -c_3^2-c_3R_c)>0$ and $c_0(a_0+c_0\mu)(R_c-c_3)(2c_0c_2-c_3^2+c_3R_c)>0$, then $S_0$ is an unstable node and $S_1$ and $S_2$ are saddles. We obtain the phase portrait (L42) of Figure \eqref{fig:localppO1}.
	
	\vspace{0.1cm}	
	
\item[(E12)]  If $c_0\neq0$, $c_3^2>4c_0c_2$, $c_2<0$, $\mu<-1$, $c_0(a_0+c_0\mu) (2c_0c_2- c_3^2-c_3R_c)>0$, $c_0(R_c-c_3)>0$ and $(a_0+c_0\mu)(2c_0c_2-c_3^2 +c_3R_c)<0$, then $S_0$ is an unstable node, $S_1$ is a saddle and $S_2$ is a stable node. We obtain the phase portrait (L43) of Figure \eqref{fig:localppO1}.
	
	\vspace{0.1cm}	
	
\item[(E13)]  If $c_0\neq0$, $c_3^2>4c_0c_2$, $c_2<0$, $\mu<-1$, $c_0>0$, $(a_0+c_0\mu) (2c_0c_2- c_3^2-c_3R_c)<0$, $R_c-c_3>0$, $(a_0+c_0\mu)(2c_0c_2- c_3^2+c_3R_c)<0$, then $S_0$ is an unstable node and $S_1$ and $S_2$ are stable nodes. We obtain the phase portrait (L44) of Figure \eqref{fig:localppO1}.
\end{enumerate}
	
Note that in the phase portraits (L22), (L30), (L33), (L43), (L46) and (L47) of Figure \ref{fig:localppO1} it is possible to consider hyperbolic sectors instead of the elliptic ones, but we have only represented the elliptic cases by the same reason given before, i.e. because applying the index theory to the phase portraits in the sphere $\mathbb{S}^2$ described in Section \ref{sec:global}, we proved that they are the only feasible.
	
Completed the study in the local chart $U_1$, we address the study of the origin of chart $U_2$ which turned out to be much simpler. The system has the expression
\begin{equation}\label{systemU2}
\begin{split}
\dot{u}&=-c_1(\mu+1)u^2v+(a_0-c_0)uv^2-c_3(\mu+1)uv-c_2(\mu+1)u,\\
\dot{v}&=-c_1 uv^2 - c_0 v^3 - c_3 v^2 - c_2 v.
\end{split}
\end{equation}
	
\begin{lemma}\label{lemma_O2}
The origin of chart $U_2$ is always a hyperbolic infinite singular point of system \eqref{system}. It is a saddle if $\mu<-1$, a stable node if $c_2>0$ and $\mu>-1$, and an unstable node if $c_2<0$ and $\mu>-1$.
\end{lemma}

\section{Global phase portraits} \label{sec:global}
		
In order to prove the global result stated in Theorem \ref{th_global}, we will bring together the local information obtained in the previous sections. We start our classification from the cases in Tables \ref{tab:cases_fin_local_1} to  \ref{tab:cases_fin_local_6}. In some of them the conditions determine only one local phase portrait in each one of the infinite singular points but, in many others we shall distinguish several possibilities. In some cases the local information gives rise to  only one  phase portrait, this occurs when the sepatrices can be connected in only one way, but in others several global possibilities appear, and we shall prove which of them are feasible.

In Table \ref{tab:global} we give, for each case in the Tables \ref{tab:cases_fin_local_1} to \ref{tab:cases_fin_local_6}, the local phase portrait of the infinite singularities $O_1$ and $O_2$ (in most cases this depends on the parameters), and also we give the global phase portrait on the Poincar\'e disc obtained. And now we detail the reasonings in some cases, although they will not be showed in all cases to avoid repetitions.

\vspace{0.1cm}

\textbf{Case 1.1.} The infinite singular point $O_1$ has the local phase portrait (L12) given in Figure \ref{fig:localppO1}, and $O_2$ is an unstable node. As we said in Section \ref{sec:infinite}, the elliptic sectors appearing in phase portrait (L12) could be hyperbolic sectors if we attend only to local results, but now having all the global information we can prove that they are elliptic by using the index theory. By Theorem \ref{th_PoincareHopf} the sum of the indices of all the singular points on the Poincar\'e sphere has to be $2$. To compute this sum we must consider that the finite singular points on the Poincar\'e disc appear twice on the sphere (on the northern hemisphere and on the southern hemisphere).
Thus if we denote by $ind_F$ the sum of the indices of the finite singular points, and by $ind_I$ the sum of the indices of the infinite singular points, the equality $2 ind_F + ind_I=2$ must be satisfied.

In this particular case the finite singular points are a saddle-node whose index is $0$, and two saddles whose index is $-1$, so $ind_F=-2$. We deduce that $ind_I$ must be $6$. The infinite singular points are $O_1$ and $O_2$, the origins of the local charts $U_1$ and $U_2$, and the origins of the symmetric local charts $V_1$ and $V_2$, which have the same index. Since $O_2$ is a node, and so it has index $1$, we get that the sum of the indices of $O_2$ and its symmetric must be $4$, i.e. the index of $O_2$ has to be $2$. From the Poincar\'e formula for the index given in subsection \ref{subsec:indices} we get
\begin{equation*}
\frac{e-h}{2}+1=2 \Rightarrow e-h=2.
\end{equation*}
Hence only the case with tho elliptic sectors on the  local phase portait (L12) is possible, because if we had two hyperbolic sectors instead of the elliptic ones, the index of $O_2$ would be zero.

We recall that by an analogous application of the index theory in the corresponding cases, it can be concluded that elliptic sectors appearing in the local phase portraits (L8), (L9), (L11), (L22), (L30), (L33), (L43), (L46) and (L47) of $O_1$, are indeed elliptic rather than hyperbolic.

In this case 1.1 there is only one possible phase portrait on the Poincaré disc, the one given in Figure \eqref{fig:global_sis2.1} (G1).

\vspace{0.2cm}

\textbf{Case 1.6.} In this case $O_1$ has the local phase portrait (L3) and $O_2$ is a stable node. From the local results we can obtain three possible global phase portraits given in Figure 	\ref{fig:global_1_17_1}.
\begin{figure}[H]
	\centering
	\begin{subfigure}[h]{3cm}
		\centering
		\includegraphics[width=3cm]{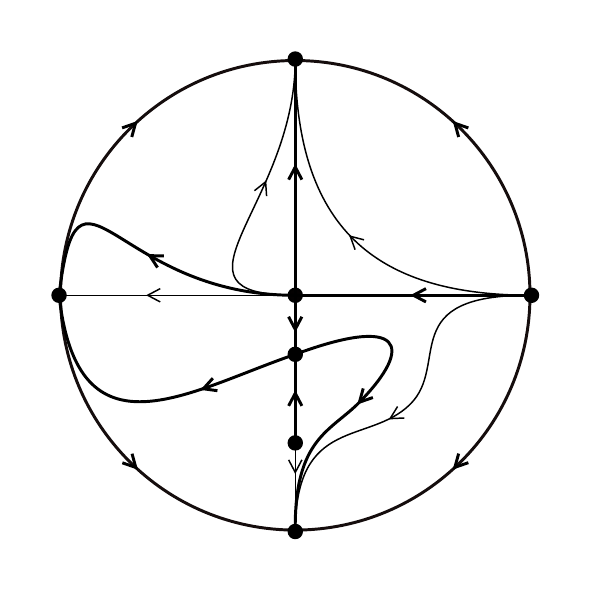}
		\caption*{Subcase 1}
	\end{subfigure}
	\begin{subfigure}[h]{3cm}
		\centering
		\includegraphics[width=3cm]{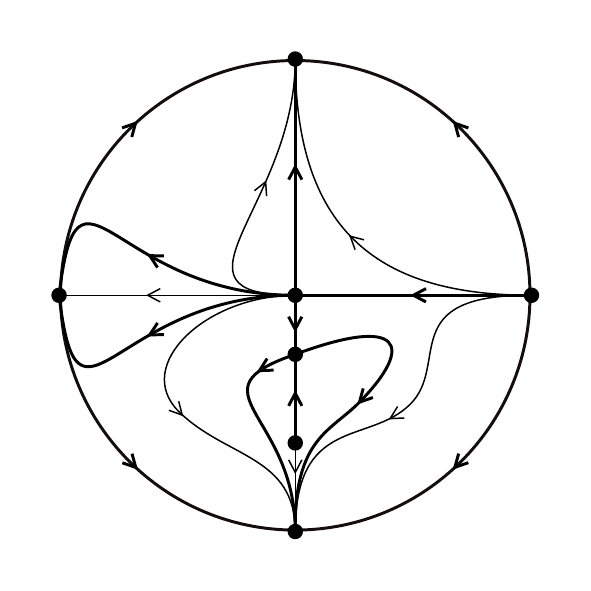}
		\caption*{Subcase 2}
	\end{subfigure}
	\begin{subfigure}[h]{3cm}
		\centering
		\includegraphics[width=3cm]{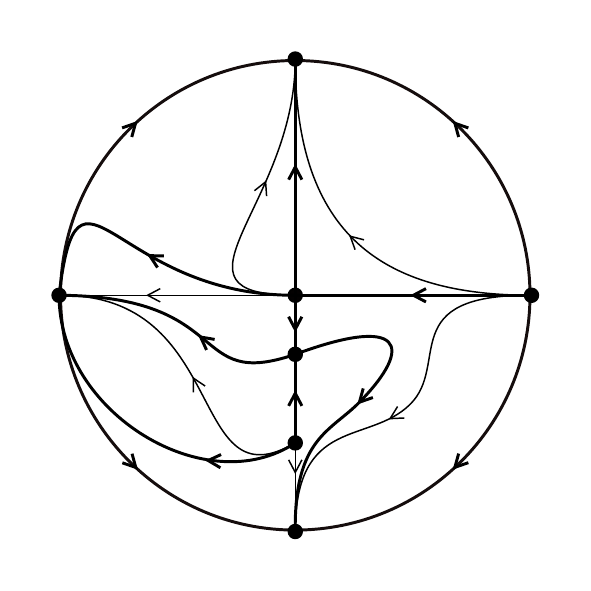}
		\caption*{Subcase 3}
	\end{subfigure}
	\caption{Possible global phase portraits in case 1.6.}
	\label{fig:global_1_17_1}
\end{figure}

By Theorem \ref{th_contactpoints} on the straight lines $z=z_0\neq0$ cannot be more than one contact point, but as it is shown in Figure 	\ref{fig:global_1_17_1_contacpoints}, if subcases 1 and 2 are feasible, there exist straight lines $z=z_0$, with $ -(R_c+c_3)/(2c_2)<z_0 <(R_c-c_3)/(2c_2)$, on which there exist two contact points, so we deduce that the only possible global phase portrait is the subcase 3, i.e. (G10) of Figure \ref{fig:global_sis2.1}.
\begin{figure}[H]
	\centering
	\begin{subfigure}[h]{3cm}
		\centering
		\includegraphics[width=3cm]{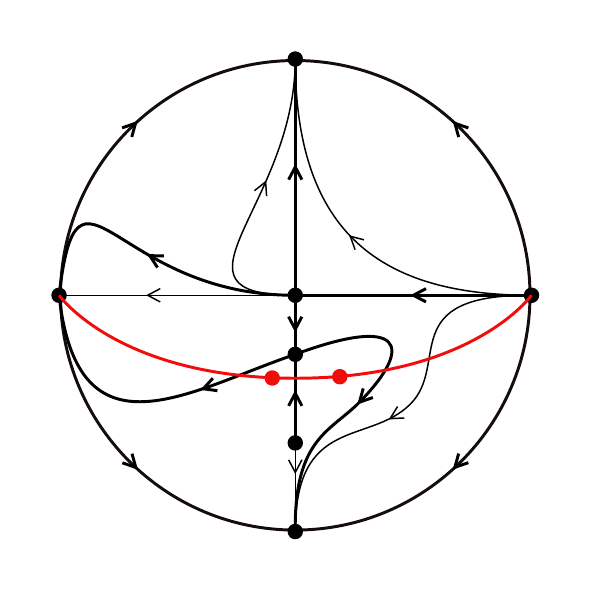}
		\caption*{Subcase 1}
	\end{subfigure}
	\begin{subfigure}[h]{3cm}
		\centering
		\includegraphics[width=3cm]{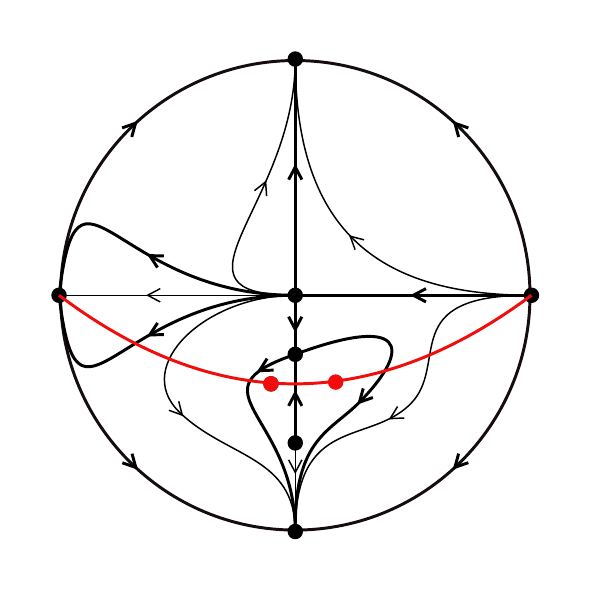}
		\caption*{Subcase 2}
	\end{subfigure}
	\caption{Straight lines with two contact points on the two first subcases of 1.6}
	\label{fig:global_1_17_1_contacpoints}
\end{figure}

\textbf{Case 1.10.} In this case $O_1$ has the local phase portrait (L6) and $O_2$ is an unstable node. From the local results we can obtain three possible global phase portraits, but in two of them shown in Figure	\ref{fig:global_1_29_1_contacpoints} we can find straight lines $z=z_0\neq0$ with $(R_c-c_3)/(2c_2)<z_0<-(R_c+c_3)/(2c_2)$, on which there are two contact points, so according to Theorem \ref{th_contactpoints} they are not possible. Then the only possibility is the phase portrait (G20) of Figure \ref{fig:global_sis2.1}.
\begin{figure}[H]
	\centering
	\begin{subfigure}[h]{3cm}
		\centering
		\includegraphics[width=3cm]{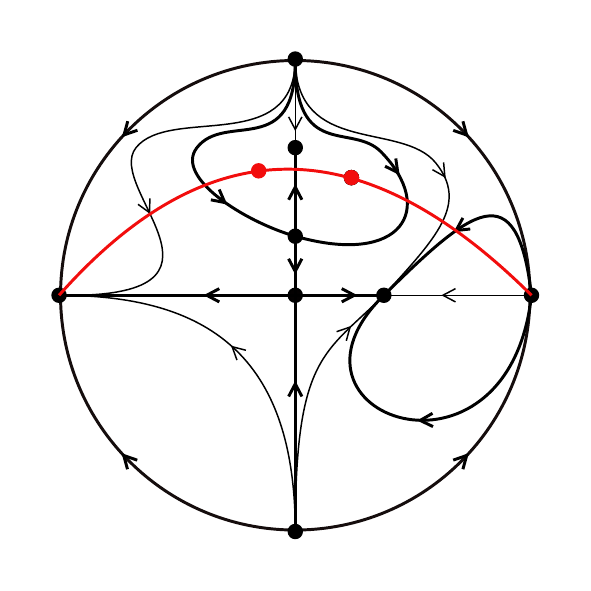}
		\caption*{Subcase 1}
	\end{subfigure}
	\begin{subfigure}[h]{3cm}
		\centering
		\includegraphics[width=3cm]{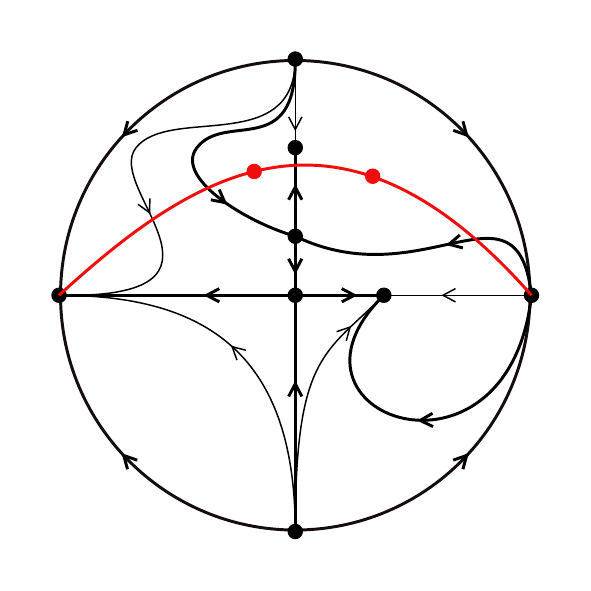}
		\caption*{Subcase 2}
	\end{subfigure}
	\caption{Straight lines with two contact points on two subcases of 1.10.}
	\label{fig:global_1_29_1_contacpoints}
\end{figure}

\textbf{Case 2.2.} In this case $O_1$ has the local phase portrait (L47) and $O_2$ is an unstable node. From the local results and by Theorem \ref{th_simmetry} the phase portrait is symmetric, we obtain three possible global phase portraits, the ones given of Figure \ref{fig:global_2_2_5}.
\vspace{-0.2cm}
\begin{figure}[H]
\centering
	\begin{subfigure}[h]{3cm}
		\centering
		\includegraphics[width=3cm]{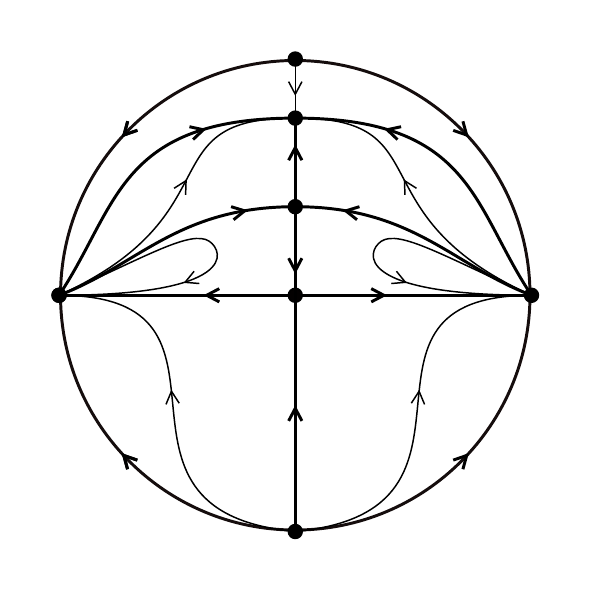}
		\caption*{Subcase 1}
	\end{subfigure}
	\begin{subfigure}[h]{3cm}
		\centering
		\includegraphics[width=3cm]{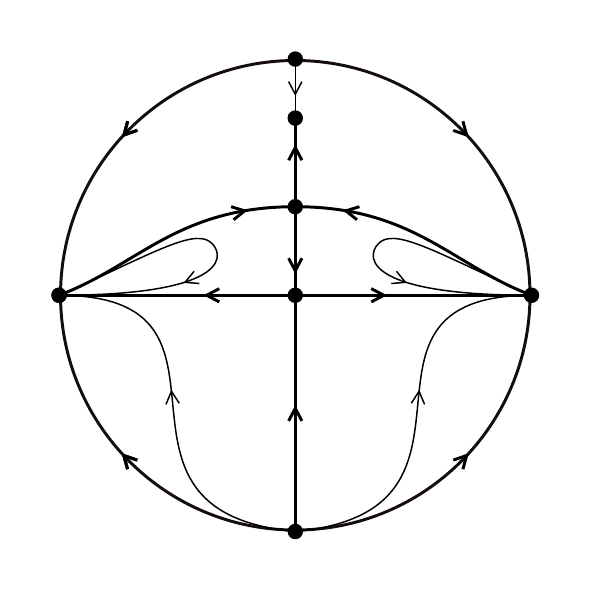}
		\caption*{Subcase 2}
	\end{subfigure}
	\begin{subfigure}[h]{3cm}
		\centering
		\includegraphics[width=3cm]{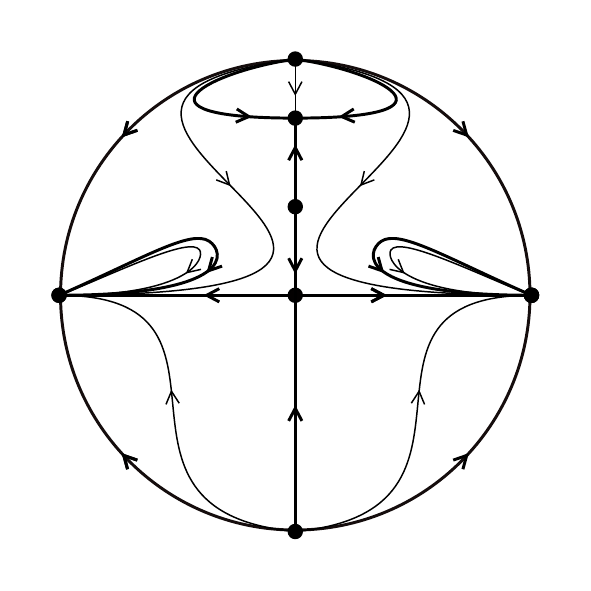}
		\caption*{Subcase 3}
	\end{subfigure}
	\caption{Possible global phase portraits in case (2.2).}
	\label{fig:global_2_2_5}
\end{figure}
By Theorem \ref{th_contactpoints} we know that, under the conditions of this case, two invariant straight lines $z=\pm(R_c-c_3)/(2c_2)$  must exist, and it is only possible in the subcase 1, which provides the phase portrait (G42) of Figure \ref{fig:global_sis2.1}.

\vspace{0.2cm}

\textbf{Case 2.6.} Here we distinguish three subcases and, in two of them, three global phase portraits appear, but in each case we use different arguments to prove wich of the options is realizable. If $\mu=0$, then $O_1$ has the local phase portrait (L5) and $O_2$ is a stable node. We obtain three phase portraits, but we conclude that two of them are not feasible because we can find staight lines $z=z_0\neq0$ with two contact points, as it is shown in Figure \ref{fig:global_2_16_1_contactpoints}. Therefore there is only one global phase portrait, the (G50) of Figure \ref{fig:global_sis2.1}.
\begin{figure}[H]
	\centering
	\begin{subfigure}[h]{3cm}
		\centering
		\includegraphics[width=3cm]{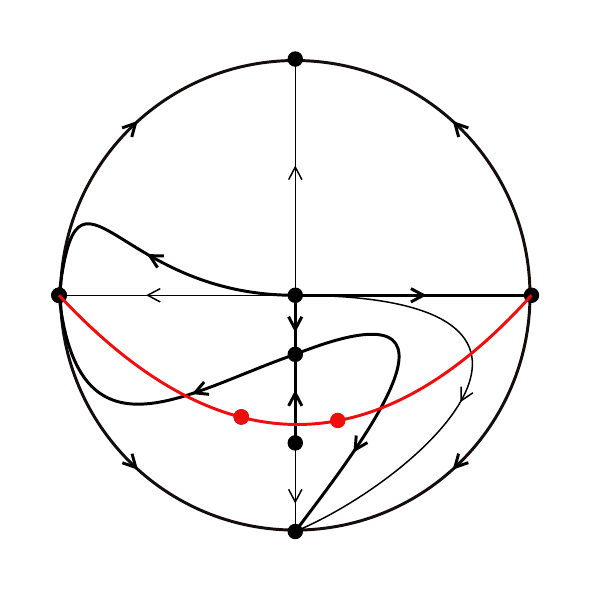}
		\caption*{Subcase 1}
	\end{subfigure}
	\begin{subfigure}[h]{3cm}
		\centering
		\includegraphics[width=3cm]{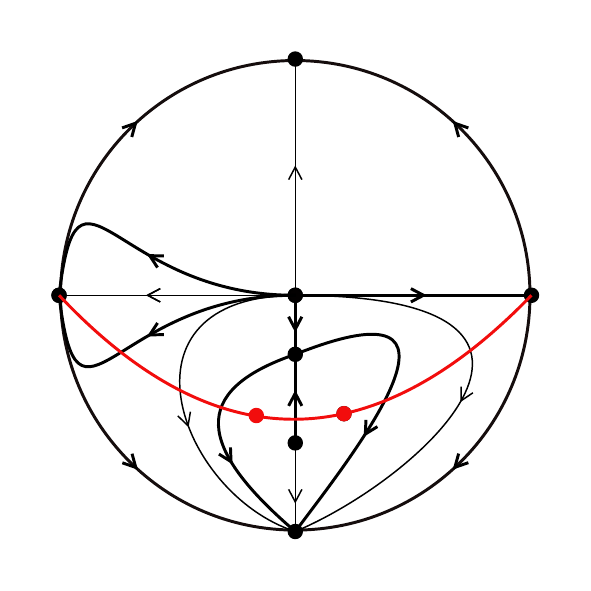}
		\caption*{Subcase 2}
	\end{subfigure}
	\caption{Straight lines with two contact points on two subcases.}
	\label{fig:global_2_16_1_contactpoints}
\end{figure}

If $c_1=0$ and $\mu>-1$, then $O_1$ has the local phase portrait (L38) and $O_2$ is a stable node. We obtain three possible global phase portraits, the ones given in Figure \ref{fig:global_2_16_3}. By Theorem \ref{th_contactpoints} we know that, under the conditions of this case, two invariant straight lines $z=\pm(R_c-c_3)/(2c_2)$ must exist, and it is only possible in the subcase 3, which provides the phase portrait (G51) of Figure \ref{fig:global_sis2.1}.
\begin{figure}[H]
	\centering
	\begin{subfigure}[h]{3cm}
		\centering
		\includegraphics[width=3cm]{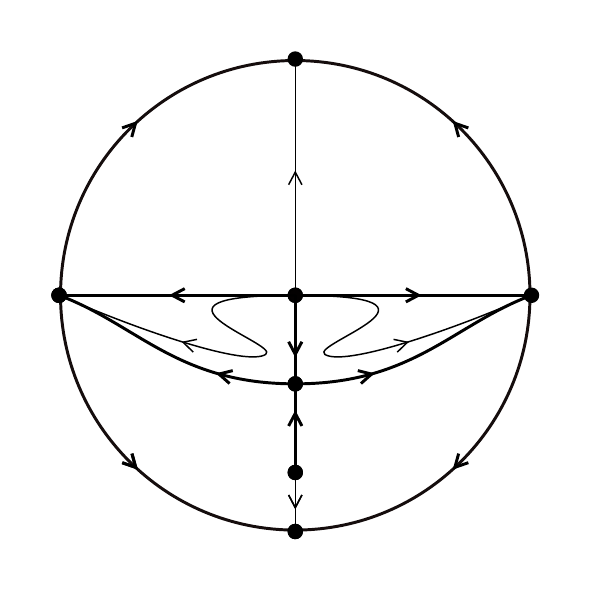}
		\caption*{Subcase 1}
	\end{subfigure}
	\begin{subfigure}[h]{3cm}
		\centering
		\includegraphics[width=3cm]{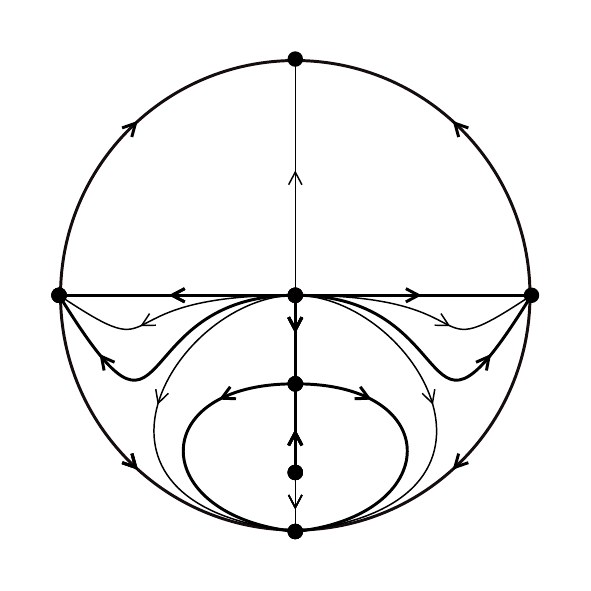}
		\caption*{Subcase 2}
	\end{subfigure}
	\begin{subfigure}[h]{3cm}
		\centering
		\includegraphics[width=3cm]{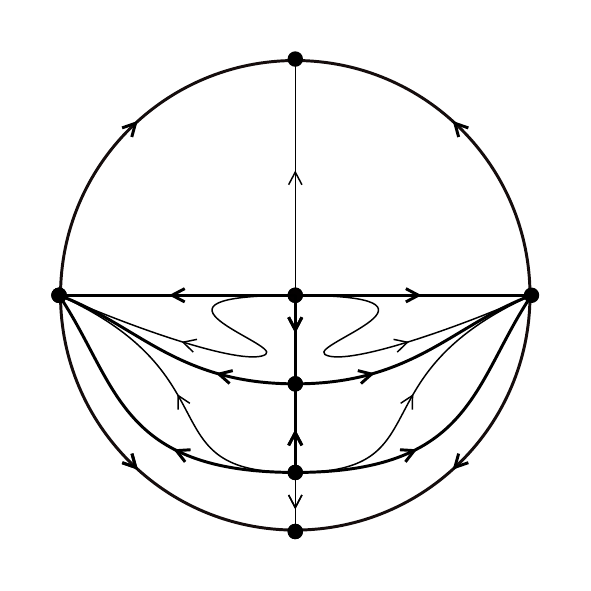}
		\caption*{Subcase 3}
	\end{subfigure}
	\caption{Possible global phase portraits in case 2.6 with $c_1=0$ and $\mu>-1$.}
	\label{fig:global_2_16_3}
\end{figure}

If $c_1=0$ and $\mu<-1$, then $O_1$ has the local phase portrait (L41) and $O_2$ is a saddle. In this case we obtain only one phase portrait (G52) of Figure \ref{fig:global_sis2.1}.

\vspace{0.2cm}

\textbf{Case 3.2.} Here we distinguish three subcases and in two of them there is only one possible global phase portait. More precisely, if $\mu<-1$ then $O_1$ has the local phase portrait (L8), $O_2$ is a saddle and the global phase portrait is (G64). If $\mu\in(-1,0)$ then $O_1$ has the local phase portrait (L13), $O_2$ is an unstable node and we obtain the phase portrait (G65).

If $\mu>0$ then $O_1$ has the local phase portrait (L6) and $O_2$ is an unstable node, but in this case we get three phase portraits, the ones given in Figure	\ref{fig:global_3_2_1}. By Theorem \ref{th_contactpoints}, there must exist a contact point on each straight line $z=z_0$, but if in subcases 1 and 2 we take a straight line $z=z_0$ with $z_0>-c_3/(2c_2)$, there are not contact points on it, so those subcases are not feasible. The only possibility is the subcase 3, which provides the phase portrait (G63) of Figure \ref{fig:global_sis2.1}.
\begin{figure}[H]
\centering
	\begin{subfigure}[h]{3cm}
		\centering
		\includegraphics[width=3cm]{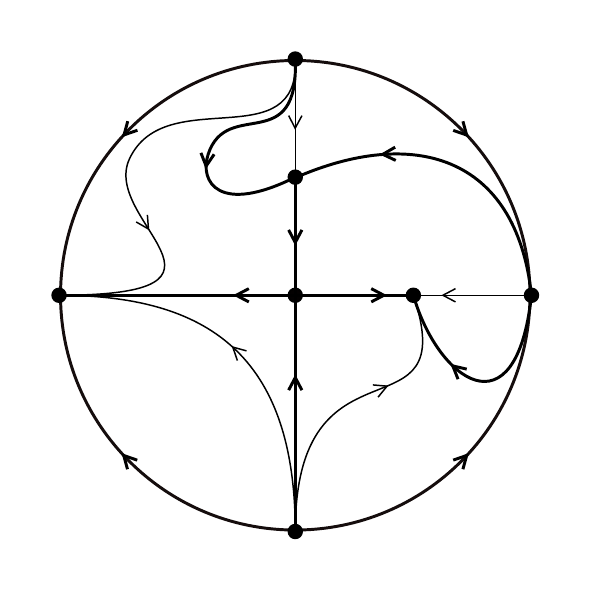}
		\caption*{Subcase 1}
	\end{subfigure}
	\begin{subfigure}[h]{3cm}
		\centering
		\includegraphics[width=3cm]{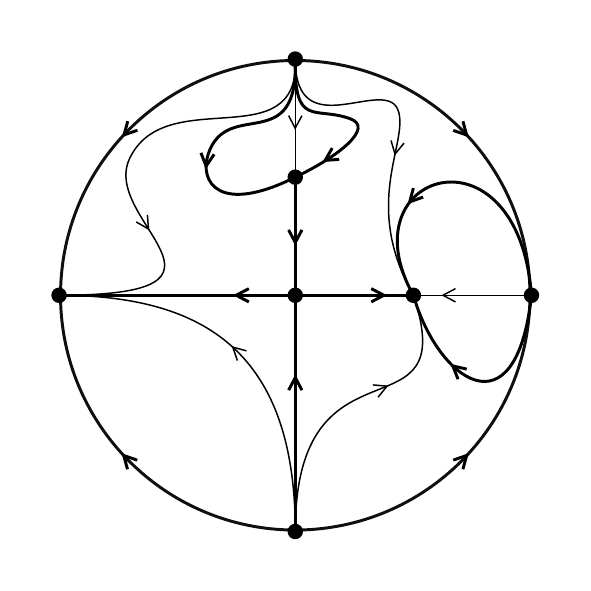}
		\caption*{Subcase 2}
	\end{subfigure}
	\begin{subfigure}[h]{3cm}
		\centering
		\includegraphics[width=3cm]{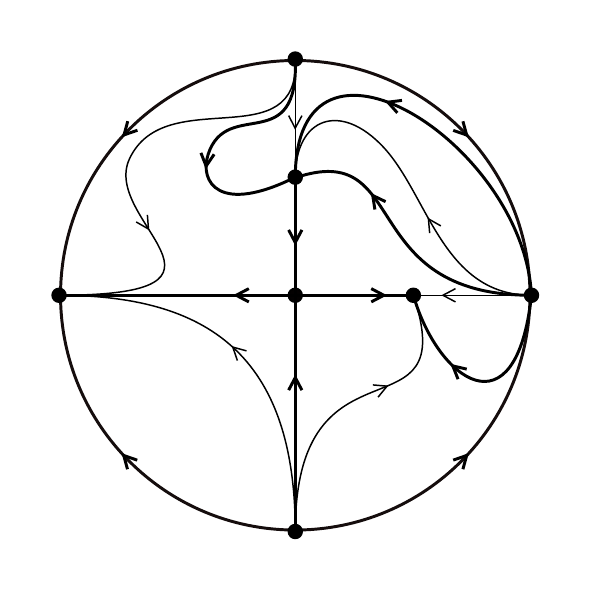}
		\caption*{Subcase 3}
	\end{subfigure}
	\caption{Possible global phase portraits in case 3.2 with $\mu>0$.}
	\label{fig:global_3_2_1}
\end{figure}

\vspace{0.1cm}

\textbf{Case 4.2.} Here we distinguish five different subcases. First if $c_1=0$, $\mu<-1$ and $a_0+c_0\mu>0$, then $O_1$ has the local phase portrait (L31) and $O_2$ is a saddle. In this case we obtain the global phase portrait (G90).

If $c_1=0$, $\mu<-1$ and $a_0+c_0\mu<0$, then $O_1$ has the local phase portrait (L32) and $O_2$ is a saddle. By Corollary \ref{th_simmetry} the phase portrait must be symmetric so we obtain three possibilities, given in Figure \ref{fig:global_4_3_8}. We further know that there must exist an invariant straight line $z=-c_3/2c_2$, so we can deduce that subcase 2 is not feasible  because that invariant straight line does not exist. As we also know that this invariant straight line is a separatrix in the local phase portrait of $O_1$, the one appearing in (L32), the subcase 3 is not feasible, because there would exist another separatrix over the invariant straight line that does not appear on (L32). So finally the only possible phase portrait is (G91).

The same happens in the next cases in which we initially obtain three possibilities but we can discard two of them with the same arguments, so finally we get the next results. If $\mu=0$, then $O_1$ has the local phase portrait (L5) and $O_2$ is a stable node and we obtain the global phase portrait (G87). If $c_1=0$, $\mu>-1$ and $a_0+c_0\mu>0$, then $O_1$ has the local phase portrait (L27), $O_2$ is a stable node, and we obtain the phase portrait (G88). Finally if $c_1=0$, $\mu>-1$ and $a_0+c_0\mu<0$, then $O_1$ has the local phase portrait (L28), $O_2$ is a stable node, and the global phase portait is (G89).
\vspace{-0.1cm}
\begin{figure}[H]
	\centering
	\begin{subfigure}[h]{3cm}
		\centering
		\includegraphics[width=3cm]{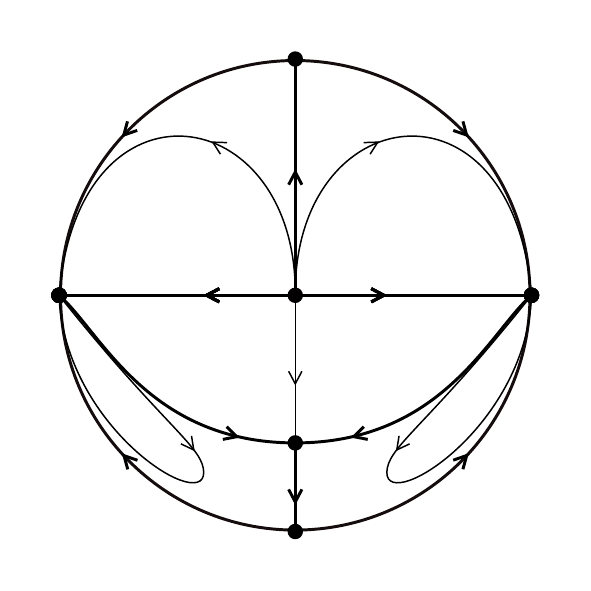}
		\caption*{Subcase 1}
	\end{subfigure}
	\begin{subfigure}[h]{3cm}
		\centering
		\includegraphics[width=3cm]{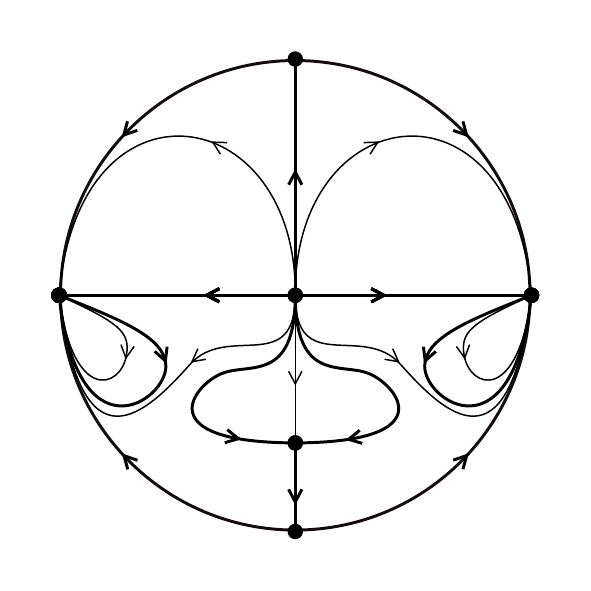}
		\caption*{Subcase 2}
	\end{subfigure}
	\begin{subfigure}[h]{3cm}
		\centering
		\includegraphics[width=3cm]{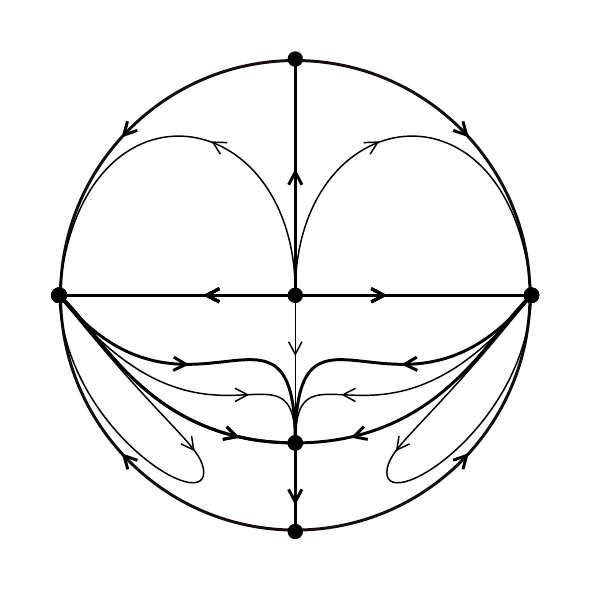}
		\caption*{Subcase 3}
	\end{subfigure}
	\caption{Possible global phase portraits in case 4.2 with $c_1=0$, $\mu<-1$ and $a_0+c_0\mu<0$.}
	\label{fig:global_4_3_8}
\end{figure}

\vspace{0.1cm}

The same methods that we have used in the previous cases for determine which of the global phase portraits are realizable, must be used in some other cases, namely, 1.17, 1.18 with $\mu\geq-2$, 1.19, 2.7, 3.3, 3.4 with $\mu\geq-2$, 3.5 with $\mu>0$, and finally 4.1 with $c_1=0$, $\mu>-1$ and $a_0+c_0\mu<0$.

\newpage

\begin{table}[H]
\begin{center}
		%\resizebox{7.5cm}{!}{%
		\begin{tabular}{|c|c|c|c|c|}
			\hline
			\small\textbf{Case}& \small\textbf{Conditions}& \small\textbf{$\boldsymbol{O_1}$}& \small\textbf{$\boldsymbol{O_2}$} & \small\textbf{Global} \\
			\hline
			\hline
			
			 1.1
			&
			
			&
			L12
			&
			Unstable node
			&
			G1 \\
			\hline
			
			1.2
			&
			
			&
			L3
			&
			Stable node
			&
			G2 \\
		\hline

			\multirow{2}{1cm}{\centering{1.3}}
			&
			$\mu<-1$
			&
			L8
			&
			Saddle
			&
			G3 \\
			\cline{2-5}
			
			\multirow{2}{*}
			&
			$\mu\in(-1,0)$
			&
			L13
			&
			Unstable node
			&
			G4 \\
			\hline

			\multirow{4}{1cm}{\centering{1.4}}
			&
			$\mu<-2$
			&
			L1
			&
			Saddle
			&
			G5 \\
			\cline{2-5}
			
			\multirow{4}{*}
			&
			$\mu\in(-1,0)$
			&
			L4
			&
			Saddle
			&
			G7 \\
			\cline{2-5}
			
			\multirow{4}{*}
			&
			$\mu\in(-2,-1)$
			&
			L15
			&
			Saddle
			&
			G7\\
			\cline{2-5}
			
			\multirow{4}{*}
			&
			$\mu=-2$
			&
			L17
			&
			Saddle
			&
			G8 \\
			\hline

			1.5
			&
			
			&
			L12
			&
			Unstable node
			&
			G9 \\
			\hline

			1.6
			&
			
			&
			L3
			&
			Stable node
			&
			G10 \\
			\hline

		\multirow{2}{1cm}{\centering{1.7}}
		&
		$\mu\in(-1,0)$
		&
		L10
		&
		Stable node
		&
		G11 \\
		\cline{2-5}
		
		\multirow{2}{*}
		&
		$\mu<-1$
		&
		L11
		&
		Saddle
		&
		G12 \\
		\hline

				\multirow{4}{1cm}{\centering{1.8}}
				&
				$\mu<-2$
				&
				L2
				&
				Saddle
				&
				G13 \\
				\cline{2-5}
				
				\multirow{4}{*}
				&
				$\mu\in(-1,0)$
				&
				L7
				&
				Unstable node
				&
				G14 \\
				\cline{2-5}
				
				\multirow{4}{*}
				&
				$\mu\in(-2,-1)$
				&
				L16
				&
				Saddle
				&
				G15 \\
				\cline{2-5}
				
				\multirow{4}{*}
				&
				$\mu=-2$
				&
				L18
				&
				Saddle
				&
				G16 \\
				\hline
			
				1.9
				&
				
				&
				L9
				&
				Stable node
				&
				G19 \\
				\hline
				
				1.10
				&
				
				&
				L6
				&
				Unstable node
				&
				G20 \\
				\hline
				
				1.11
				&
				
				&
				L12
				&
				Unstable node
				&
				G17 \\
				\hline

				\multirow{2}{1cm}{\centering{1.12}}
				&
				$\mu<-1$
				&
				L8
				&
				Saddle
				&
				G21 \\
				\cline{2-5}
				
				\multirow{2}{*}
				&
				$\mu\in(-1,0)$
				&
				L13
				&
				Unstable node
				&
				G22 \\
				\hline
				
				1.13
				&
				
				&
				L3
				&
				Stable node
				&
				G18 \\
				\hline

				\multirow{4}{1cm}{\centering{1.14}}
				&
				$\mu<-2$
				&
				L1
				&
				Saddle
				&
				G23 \\
				\cline{2-5}
				
				\multirow{4}{*}
				&
				$\mu\in(-1,0)$
				&
				L4
				&
				Stable node
				&
				G24 \\
				\cline{2-5}
				
				\multirow{4}{*}
				&
				$\mu\in(-2,-1)$
				&
				L15
				&
				Saddle
				&
				G25 \\
				\cline{2-5}

				\multirow{4}{*}
				&
				$\mu=-2$
				&
				L17
				&
				Saddle
				&
				G26 \\
				\hline
				
				1.15
				&
				
				&
				L12
				&
				Unstable node
				&
				G27 \\
				\hline
				
				\multirow{2}{1cm}{\centering{1.16}}
				&
				$\mu<-1$
				&
				L8
				&
				Saddle
				&
				G35 \\
				\cline{2-5}
				
				\multirow{2}{*}
				&
				$\mu\in(-1,0)$
				&
				L13
				&
				Unstable node
				&
				G36 \\
				\hline
				
				1.17
				&
				
				&
				L3
				&
				Stable node
				&
				G28 \\
				\hline
				
				\multirow{4}{1cm}{\centering{1.18}}
				&
				$\mu<-2$
				&
				L1
				&
				Saddle
				&
				G37 \\
				\cline{2-5}
				
				\multirow{4}{*}
				&
				$\mu\in(-1,0)$
				&
				L4
				&
				Stable node
				&
				G38 \\
				\cline{2-5}
				
				\multirow{4}{*}
				&
				$\mu\in(-2,-1)$
				&
				L15
				&
				Saddle
				&
				G39 \\
				\cline{2-5}

				\multirow{4}{*}
				&
				$\mu=-2$
				&
				L17
				&
				Saddle
				&
				G40 \\
				\hline

				\multirow{2}{1cm}{\centering{1.19}}
				&
				$\mu\in(-1,0)$
				&
				L10
				&
				Saddle
				&
				G29 \\
				\cline{2-5}
				
				\multirow{2}{*}
				&
				$\mu<-1$
				&
				L11
				&
				Saddle
				&
				G30 \\
				\hline

				\multirow{4}{1cm}{\centering{1.20}}
				&
				$\mu<-2$
				&
				L2
				&
				Saddle
				&
				G31 \\
				\cline{2-5}
				
				\multirow{4}{*}
				&
				$\mu\in(-1,0)$
				&
				L7
				&
				Unstable node
				&
				G32 \\
				\cline{2-5}
				
				\multirow{4}{*}
				&
				$\mu\in(-2,-1)$
				&
				L16
				&
				Saddle
				&
				G33 \\
				\cline{2-5}

				\multirow{4}{*}
				&
				$\mu=-2$
				&
				L18
				&
				Saddle
				&
				G34 \\
				\hline

				2.1
				&
				
				&
				L46
				&
				Stable node
				&
				G41 \\
				\hline
				
				2.2
				&
				
				&
				L47
				&
				Unstable node
				&
				G42 \\
				\hline

		\end{tabular}
	%}
\end{center}
\end{table}	

%\vspace{1cm}

	\begin{table}[H]
		\begin{center}
				%\resizebox{7.5cm}{!}{%
			\begin{tabular}{|c|c|c|c|c|}
				\hline
				\small\textbf{Case}& \small\textbf{Conditions}& \small\textbf{$\boldsymbol{O_1}$}& \small\textbf{$\boldsymbol{O_2}$} & \small\textbf{Global} \\
				\hline
				\hline

					\multirow{3}{1cm}{\centering{2.3}}
				&
				$\mu=0$
				&
				L14
				&
				\multirow{2}{3cm}{\centering{Unstable node}}
				&
				\multirow{2}{1cm}{\centering{G43}} \\
				\cline{2-3}
				
				\multirow{3}{*}
				&
				$c_1=0$, $\mu>-1$
				&
				L36
				&
				\multirow{2}{*}
				&
				\multirow{2}{*} {\phantom{h}}  \\
				\cline{2-5}
				
				\multirow{3}{*}
				&
				$c_1=0$, $\mu<-1$
				&
				L43
				&
				Saddle
				&
				G44 \\
				\hline

			\multirow{3}{1cm}{\centering{2.4}}
			&
			$\mu=0$
			&
			L5
			&
			Stable node
			&
			G45\\
			\cline{2-5}
			
			\multirow{3}{*}
			&
			$c_1=0$, $\mu>-1$
			&
			L35
			&
			Stable node
			&
			G46\\
			\cline{2-5}
			
			\multirow{3}{*}
			&
			$c_1=0$, $\mu<-1$
			&
			L39
			&
			Saddle
			&
			G47 \\
			\hline

			\multirow{3}{1cm}{\centering{2.5}}
			&
			$\mu=0$
			&
			L14
			&
			\multirow{2}{3cm}{\centering{Unstable node}}
			&
			\multirow{2}{1cm}{\centering{G48}}
			\\
			\cline{2-3}
			
			\multirow{3}{*}
			&
			$c_1=0$, $\mu>-1$
			&
			L45
			&
			\multirow{2}{*}
			&
			\multirow{2}{*} {\phantom{h}}
			\\
			\cline{2-5}

			\multirow{3}{*}
			&
			$c_1=0$, $\mu<-1$
			&
			L44
			&
			Saddle
			&
			G49 \\
			\hline
						
				\multirow{3}{1cm}{\centering{2.6}}
			&
			$\mu=0$
			&
			L5
			&
			Stable node
			&
			G50\\
			\cline{2-5}
			
			\multirow{3}{*}
			&
			$c_1=0$, $\mu>-1$
			&
			L38
			&
			Stable node
			&
			G51 \\
			\cline{2-5}
			
			\multirow{3}{*}
			&
			$c_1=0$, $\mu<-1$
			&
			L41
			&
			Saddle
			&
			G52 \\
			\hline

		\multirow{2}{1cm}{\centering{2.7}}
		&
		$\mu>-1$
		&
		L37
		&
		Stable node
		&
		G53\\
		\cline{2-5}
		
		\multirow{2}{*}
		&
		$\mu<-1$
		&
		L40
		&
		Saddle
		&
		G54 \\
		\hline
		
		\multirow{2}{1cm}{\centering{2.8}}
		&
		$\mu>-1$
		&
		L34
		&
		Unstable node
		&
		G55\\
		\cline{2-5}
		
		\multirow{2}{*}
		&
		$\mu<-1$
		&
		L42
		&
		Saddle
		&
		G56 \\
		\hline

		\multirow{3}{1cm}{\centering{2.9}}
		&
		$\mu=0$
		&
		L14
		&
		\multirow{2}{3cm}{\centering{Unstable node}}
		&
		\multirow{2}{1cm}{\centering{G57}}
		\\
		\cline{2-3}
		
		\multirow{3}{*}
		&
		$c_1=0$, $\mu>-1$
		&
		L24
		&
		\multirow{2}{*}
		&
		\multirow{2}{*} {\phantom{h}}
		\\
		\cline{2-5}

		\multirow{3}{*}
		&
		$c_1=0$, $\mu<-1$
		&
		L26
		&
		Saddle
		&
		G58 \\
		\hline
	
		\multirow{3}{1cm}{\centering{2.10}}
		&
		$\mu=0$
		&
		L5
		&
		Stable node
		&
		G59\\
		\cline{2-5}
		
		\multirow{3}{*}
		&
		$c_1=0$, $\mu>-1$
		&
		L23
		&
		Stable node
		&
		G60 \\
		\cline{2-5}
		
		\multirow{3}{*}
		&
		$c_1=0$, $\mu<-1$
		&
		L25
		&
		Saddle
		&
		G61 \\
		\hline

		3.1
		&
	
		&
		L12
		&
		Unstable node
		&
		G62 \\
		\hline

		\multirow{3}{1cm}{\centering{3.2}}
		&
		$\mu>0$
		&
		L6
		&
		Unstable node
		&
		G63\\
		\cline{2-5}
		
		\multirow{3}{*}
		&
		$\mu<-1$
		&
		L8
		&
		Stable node
		&
		G64 \\
		\cline{2-5}
		
		\multirow{3}{*}
		&
		$\mu\in(-1,0)$
		&
		L13
		&
		Unstable node
		&
		G65 \\
		\hline
		
		\multirow{3}{1cm}{\centering{3.3}}
		&
		$\mu>0$
		&
		L3
		&
		Stable node
		&
		G66\\
		\cline{2-5}
		
		\multirow{3}{*}
		&
		$\mu<-1$
		&
		L11
		&
		Stable node
		&
		G68 \\
		\cline{2-5}
		
		\multirow{3}{*}
		&
		$\mu\in(-1,0)$
		&
		L10
		&
		Stable node
		&
		G67 \\
		\hline

		\multirow{4}{1cm}{\centering{3.4}}
		&
		$\mu<-2$
		&
		L1
		&
		Saddle
		&
		G69 \\
		\cline{2-5}
		
		\multirow{4}{*}
		&
		$\mu\in(-1,0)$
		&
		L4
		&
		Stable node
		&
		G70 \\
		\cline{2-5}
		
		\multirow{4}{*}
		&
		$\mu\in(-2,-1)$
		&
		L15
		&
		Saddle
		&
		G71 \\
		\cline{2-5}
		
		\multirow{4}{*}
		&
		$\mu=-2$
		&
		L17
		&
		Saddle
		&
		G72 \\
		\hline

		\multirow{3}{1cm}{\centering{3.5}}
		&
		$\mu>0$
		&
		L3
		&
		Stable node
		&
		G7 \\
		\cline{2-5}
		
		\multirow{3}{*}
		&
		$\mu\in(-1,0)$
		&
		L10
		&
		Stable node
		&
		G74 \\
		\cline{2-5}
		
		\multirow{3}{*}
		&
		$\mu<-1$
		&
		L11
		&
		Saddle
		&
		G75 \\
		\hline
		
	    3.6
		&
		
		&
		L12
		&
		Unstable node
		&
		G76 \\
		\hline
		
			\multirow{2}{1cm}{\centering{3.7}}
		&
		$\mu<-1$
		&
		L8
		&
		Saddle
		&
		G77 \\
		\cline{2-5}
		
		\multirow{2}{*}
		&
		$\mu\in(-1,0)$
		&
		L13
		&
		Unstable node
		&
		G78 \\
		\hline
		
		3.8
		&
		
		&
		L3
		&
		Stable node
		&
		G79 \\
		\hline
		
			\multirow{4}{1cm}{\centering{3.9}}
		&
		$\mu<-2$
		&
		L1
		&
		Saddle
		&
		G80 \\
		\cline{2-5}
		
		\multirow{4}{*}
		&
		$\mu\in(-1,0)$
		&
		L4
		&
		Stable node
		&
		G81\\
		\cline{2-5}
		
		\multirow{4}{*}
		&
		$\mu\in(-2,-1)$
		&
		L15
		&
		Saddle
		&
		G82\\
		\cline{2-5}
		
		\multirow{4}{*}
		&
		$\mu=-2$
		&
		L17
		&
		Saddle
		&
		G83 \\
		\hline
	
\end{tabular}
%}
\end{center}
\end{table}	
	
%\end{multicols}
	
\begin{table}
	\begin{center}
		\begin{tabular}{|c|c|c|c|c|}
			\hline
			\small\textbf{Case}& \small\textbf{Conditions}& \small\textbf{$\boldsymbol{O_1}$}& \small\textbf{$\boldsymbol{O_2}$} & \small\textbf{Global} \\
			\hline
			\hline

				\multirow{4}{1cm}{\centering{4.1}}
				&
				$\mu=0$
				&
				L14
				&
				\multirow{2}{3cm}{\centering{Unstable node}}
				&
				\multirow{2}{1cm}{\centering{G84}} \\
				\cline{2-3}
				
				\multirow{4}{*}
				&
				$c_1=0$, $\mu>-1$, $a_0+c_0\mu>0$
				&
				L29
				&
				\multirow{2}{*}
				&
				\multirow{2}{*}  {\phantom{h}} 
				\\
				\cline{2-5}
				
				\multirow{4}{*}
				&
				$c_1=0$, $\mu>-1$, $a_0+c_0\mu<0$
				&
				L30
				&
				Unstable node
				&
				G85\\
				\cline{2-5}
				
				\multirow{4}{*}
				&
				$c_1=0$, $\mu<-1$
				&
				L33
				&
				Saddle
				&
				G86 \\
				\hline

				\multirow{5}{1cm}{\centering{4.2}}
				&
				$\mu=0$
				&
				L5
				&
				Stable node
				&
			    G87\\
				\cline{2-5}
				
				\multirow{5}{*}
				&
				$c_1=0$, $\mu>-1$, $a_0+c_0\mu>0$
				&
				L27
				&
			    Stable node
				&
				G88 \\
				\cline{2-5}
				
				\multirow{5}{*}
				&
				$c_1=0$, $\mu>-1$, $a_0+c_0\mu<0$
				&
				L28
				&
				Stable node
				&
				G89 \\
				\cline{2-5}
				
				\multirow{5}{*}
				&
				$c_1=0$, $\mu<-1$, $a_0+c_0\mu>0$
				&
				L31
				&
				Saddle
				&
				G90 \\
				\cline{2-5}
				
				\multirow{5}{*}
				&
				$c_1=0$, $\mu<-1$, $a_0+c_0\mu<0$
				&
				L32
				&
				Saddle
				&
				G91 \\
				\hline
				
				\multirow{3}{1cm}{\centering{4.3}}
				&
				$\mu=0$
				&
				L14
				&
				\multirow{2}{3cm}{\centering{Unstable node}}
				&
				\multirow{2}{1cm}{\centering{G92}} \\
				\cline{2-3}
				
				\multirow{3}{*}
				&
				$c_1=0$, $\mu>-1$
				&
				L20
				&
				\multirow{2}{*}
				&
				\multirow{2}{*}  {\phantom{h}}  \\
				\cline{2-5}
				
				\multirow{3}{*}
				&
				$c_1=0$, $\mu<-1$
				&
				L22
				&
				Saddle
				&
				G93 \\
				\hline
				
				\multirow{3}{1cm}{\centering{4.4}}
				&
				$\mu=0$
				&
				L5
				&
				Stable node
				&
				G94 \\
				\cline{2-5}
				
				\multirow{3}{*}
				&
				$c_1=0$, $\mu>-1$
				&
				L19
				&
				Stable node
				&
				G95  \\
				\cline{2-5}
				
				\multirow{3}{*}
				&
				$c_1=0$, $\mu<-1$
				&
				L21
				&
				Saddle
				&
				G96 \\
				\hline
				
				\multirow{3}{1cm}{\centering{5.1}}
				&
				$\mu>0$
				&
				L3
				&
				Stable node
				&
				G97 \\
				\cline{2-5}
				
				\multirow{3}{*}
				&
				$\mu\in(-1,0)$
				&
				L10
				&
				Stable node
				&
				G98  \\
				\cline{2-5}
				
				\multirow{3}{*}
				&
				$\mu<-1$
				&
				L11
				&
				Saddle
				&
				G99 \\
				\hline

				5.2
				&
				
				&
				L12
				&
				Unstable node
				&
				G76 \\
				\hline

					\multirow{3}{1cm}{\centering{5.3}}
				&
				$\mu>0$
				&
				L6
				&
				Unstable node
				&
				G100 \\
				\cline{2-5}
				
				\multirow{3}{*}
				&
				$\mu<-1$
				&
				L8
				&
				Saddle
				&
				G77  \\
				\cline{2-5}
				
				\multirow{3}{*}
				&
				$\mu\in(-1,0)$
				&
				L13
				&
				Unstable node
				&
				G78 \\
				\hline

						\multirow{3}{1cm}{\centering{5.4}}
				&
				$\mu>0$
				&
				L3
				&
				Stable node
				&
				G79 \\
				\cline{2-5}
				
				\multirow{3}{*}
				&
				 $\mu\in(-1,0)$
				&
				L10
				&
				Stable node
				&
				G101  \\
				\cline{2-5}
				
				\multirow{3}{*}
				&
				$\mu<-1$
				&
				L11
				&
				Saddle
				&
				G102 \\
				\hline

				\multirow{4}{1cm}{\centering{5.5}}
				&
				$\mu<-2$
				&
				L1
				&
				Saddle
				&
				G80 \\
				\cline{2-5}
				
				\multirow{4}{*}
				&
				$\mu\in(-1,0)$
				&
				L4
				&
				Stable node
				&
				G81  \\
				\cline{2-5}
				
				\multirow{4}{*}
				&
				$\mu\in(-2,-1)$
				&
				L15
				&
				Saddle
				&
				G82  \\
				\cline{2-5}
				
				\multirow{4}{*}
				&
				$\mu=-2$
				&
				L17
				&
				Saddle
				&
				G83 \\
				\hline

				\multirow{3}{1cm}{\centering{6.1}}
				&
				$\mu=0$
				&
				L14
				&
				\multirow{2}{3cm}{\centering{Unstable node}}
				&
				\multirow{2}{1cm}{\centering{G92}}
				\\
				\cline{2-3}
				
				\multirow{3}{*}
				&
				$c_1=0$, $\mu>-1$
				&
				L20
				&
				\multirow{2}{*}
				&
				\multirow{2}{*} {\phantom{h}}
				\\
				\cline{2-5}

				\multirow{3}{*}
				&
				$c_1=0$, $\mu<-1$
				&
				L22
				&
				Saddle
				&
				G93 \\
				\hline

				\multirow{3}{1cm}{\centering{6.2}}
				&
				$\mu=0$
				&
				L5
				&
				Stable node
				&
				G94
				\\
				\cline{2-5}
				
				\multirow{3}{*}
				&
				$c_1=0$, $\mu>-1$
				&
				L19
				&
				Stable node
				&
				G95
				\\
				\cline{2-5}

				\multirow{3}{*}
				&
				$c_1=0$, $\mu<-1$
				&
				L21
				&
				Saddle
				&
				G96 \\
				\hline

\end{tabular}
\caption{Classification of global phase portrais of system \eqref{system}.}
\label{tab:global}
\end{center}
\end{table}	

\phantom{-}

\begin{figure}[H]
	\centering
		\begin{subfigure}[h]{2.4cm}
		\centering
		\includegraphics[width=2.4cm]{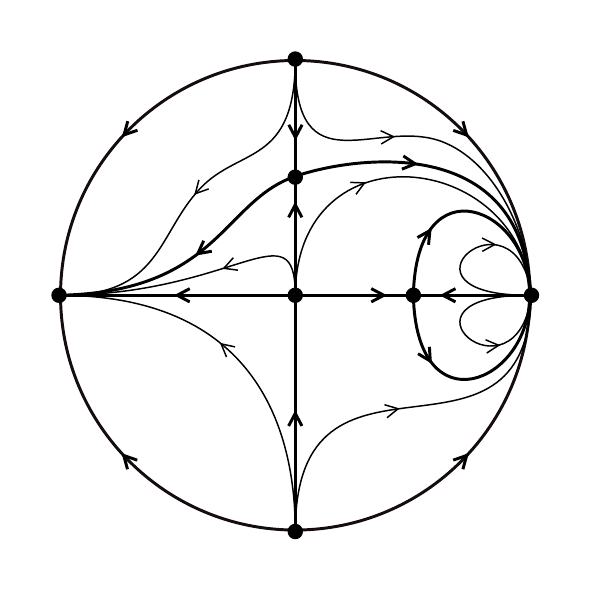}
		\caption*{\scriptsize (G1) [R=8, S=21]}
	\end{subfigure}
	\begin{subfigure}[h]{2.4cm}
		\centering
		\includegraphics[width=2.4cm]{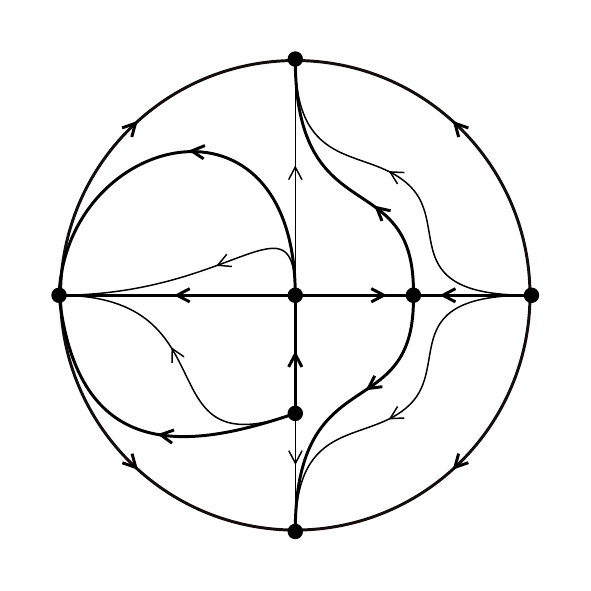}\\
		\caption*{\scriptsize(G2) [R=6, S=19]}
	\end{subfigure}
	\begin{subfigure}[h]{2.4cm}
		\centering
		\includegraphics[width=2.4cm]{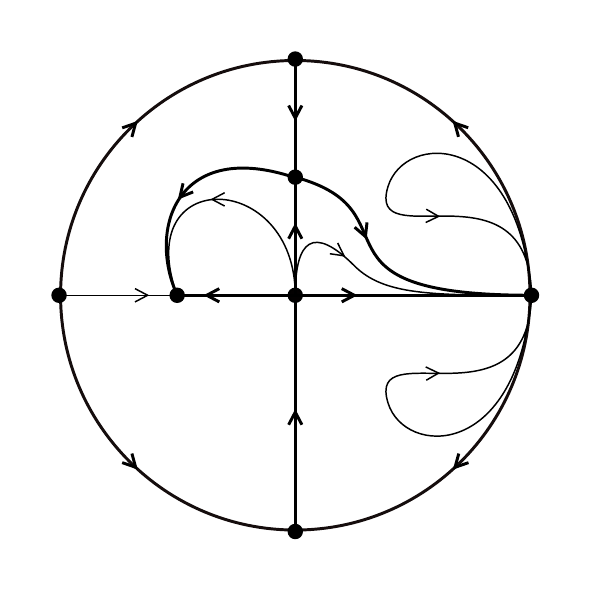}
		\caption*{\scriptsize (G3) [R=5, S=18]}
	\end{subfigure}
	\begin{subfigure}[h]{2.4cm}
		\centering
		\includegraphics[width=2.4cm]{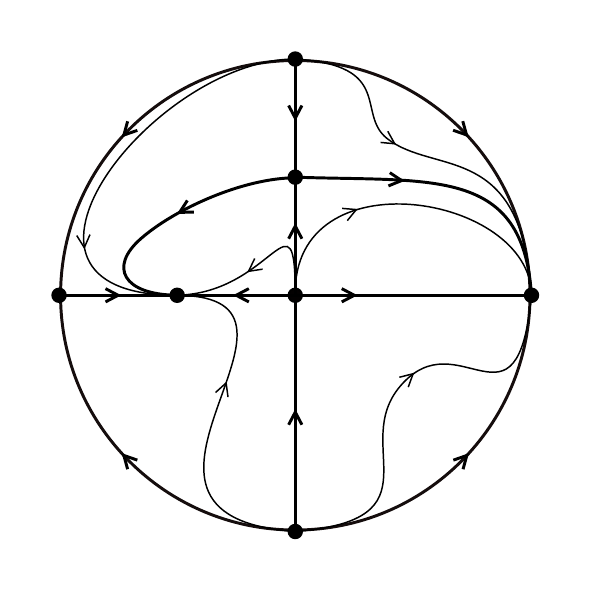}
		\caption*{\scriptsize (G4) [R=6, S=19]}
	\end{subfigure}
	\begin{subfigure}[h]{2.4cm}
		\centering
		\includegraphics[width=2.4cm]{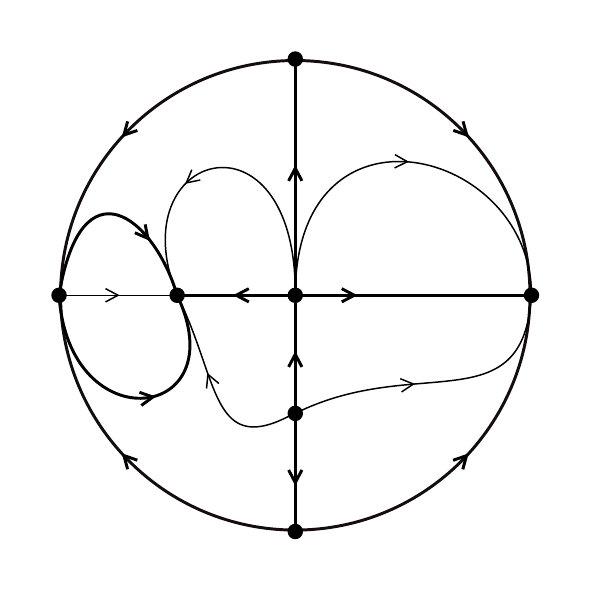}
		\caption*{\scriptsize (G5) [R=5, S=18]}
	\end{subfigure}
	\begin{subfigure}[h]{2.4cm}
		\centering
		\includegraphics[width=2.4cm]{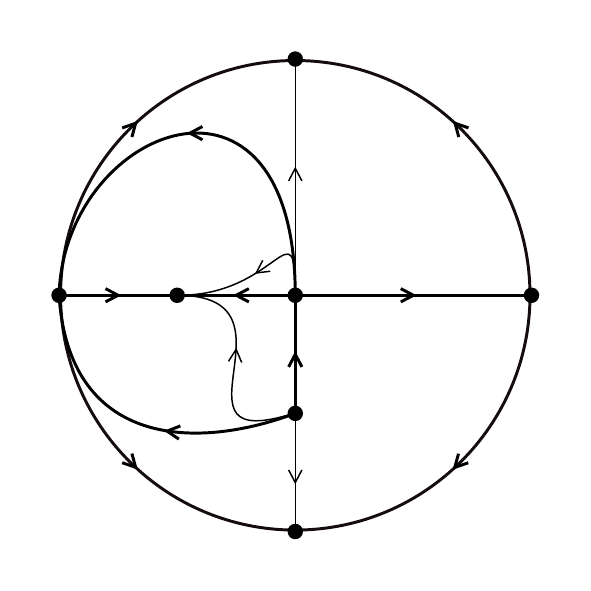}
		\caption*{\scriptsize (G6) [R=4, S=17]}
	\end{subfigure}
	\begin{subfigure}[h]{2.4cm}
		\centering
		\includegraphics[width=2.4cm]{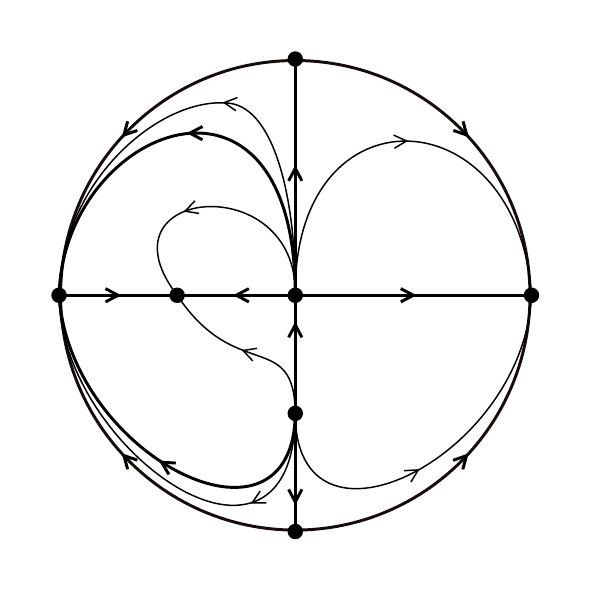}
		\caption*{\scriptsize (G7) [R=6, S=19]}
	\end{subfigure}
	\begin{subfigure}[h]{2.4cm}
		\centering
		\includegraphics[width=2.4cm]{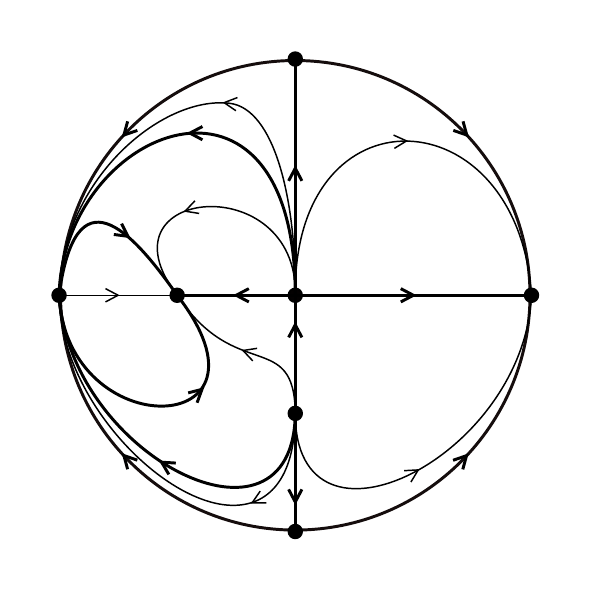}
		\caption*{\scriptsize (G8) [R=7, S=20]}
	\end{subfigure}
	\begin{subfigure}[h]{2.4cm}
		\centering
		\includegraphics[width=2.4cm]{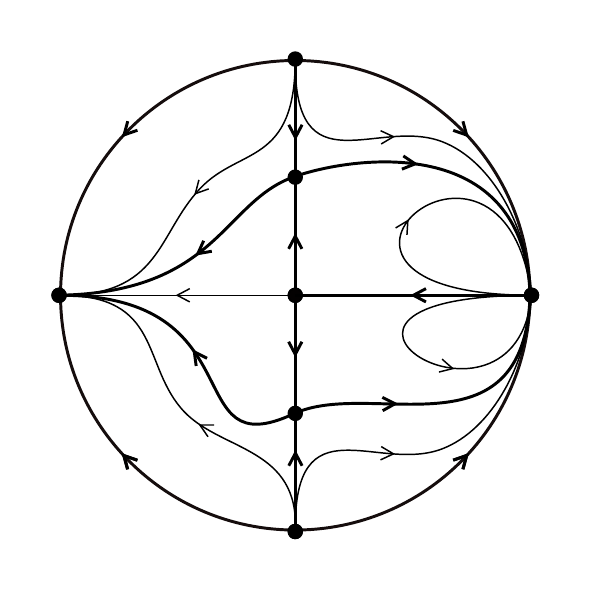}
		\caption*{\scriptsize (G9) [R=7, S=20]}
	\end{subfigure}
	\begin{subfigure}[h]{2.4cm}
		\centering
		\includegraphics[width=2.4cm]{chap2/global/G10}
		\caption*{\scriptsize (G10) [R=6, S=19]}
	\end{subfigure}
	\begin{subfigure}[h]{2.4cm}
		\centering
		\includegraphics[width=2.4cm]{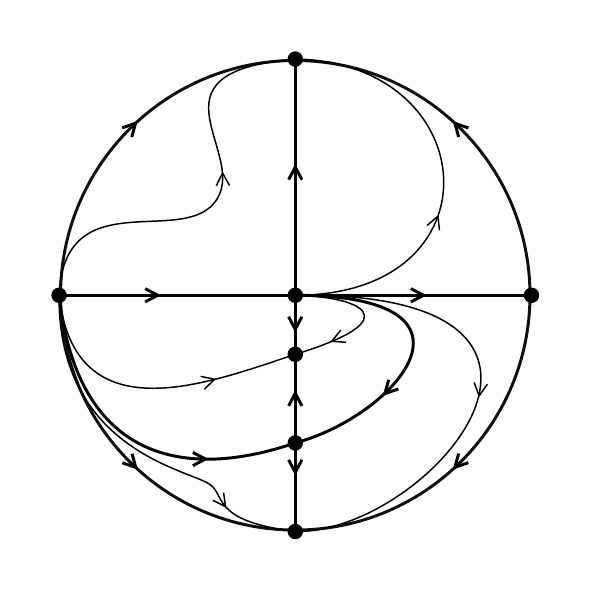}
		\caption*{\scriptsize (G11) [R=6, S=19]}
	\end{subfigure}
	\begin{subfigure}[h]{2.4cm}
		\centering
		\includegraphics[width=2.4cm]{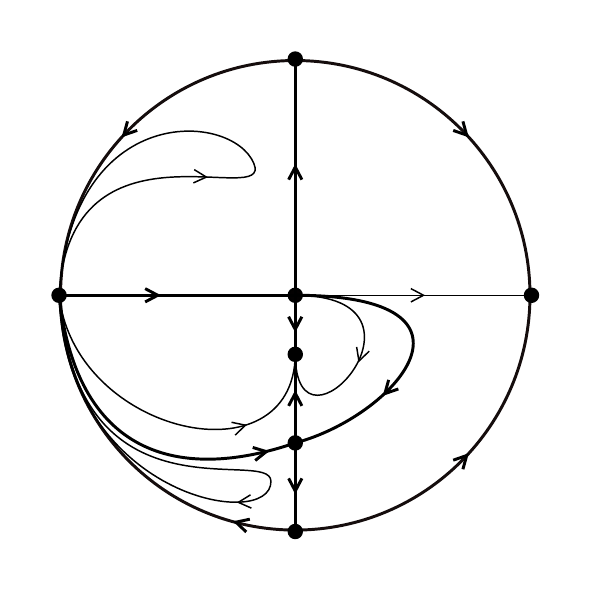}
		\caption*{\scriptsize (G12) [R=5, S=18]}
	\end{subfigure}
	\begin{subfigure}[h]{2.4cm}
		\centering
		\includegraphics[width=2.4cm]{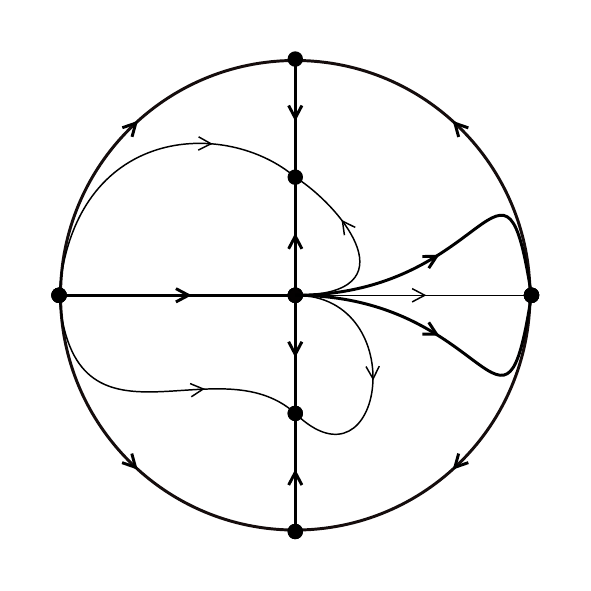}
		\caption*{\scriptsize (G13) [R=5, S=18]}
	\end{subfigure}
	\begin{subfigure}[h]{2.4cm}
		\centering
		\includegraphics[width=2.4cm]{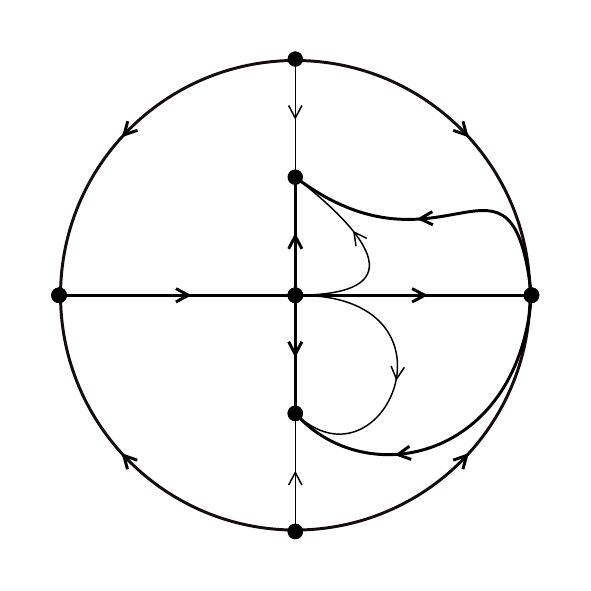}
		\caption*{\scriptsize (G14) [R=4, S=17]}
	\end{subfigure}
	\begin{subfigure}[h]{2.4cm}
		\centering
		\includegraphics[width=2.4cm]{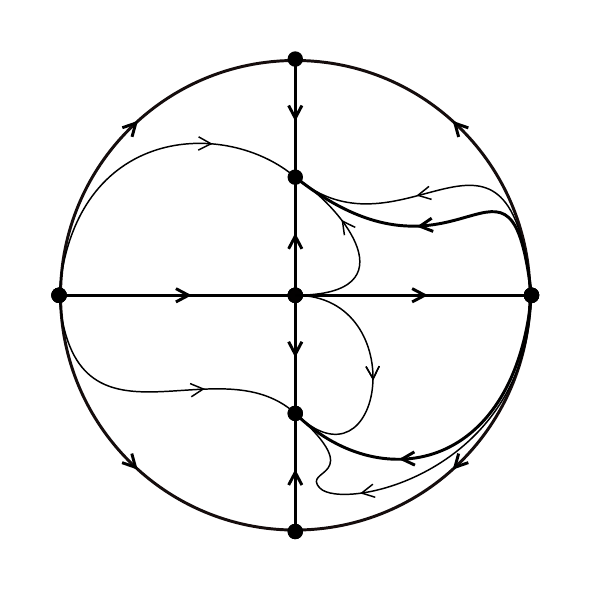}
		\caption*{\scriptsize (G15) [R=6, S=19]}
	\end{subfigure}
	\begin{subfigure}[h]{2.4cm}
		\centering
		\includegraphics[width=2.4cm]{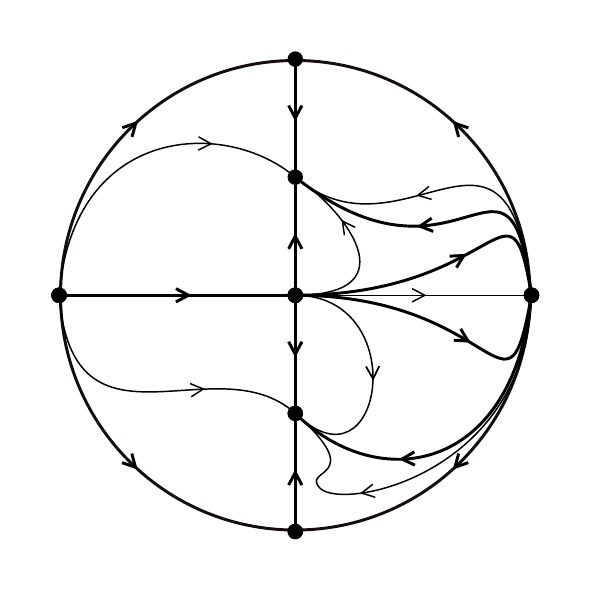}
		\caption*{\scriptsize (G16) [R=7, S=20]}
	\end{subfigure}
	\begin{subfigure}[h]{2.4cm}
		\centering
		\includegraphics[width=2.4cm]{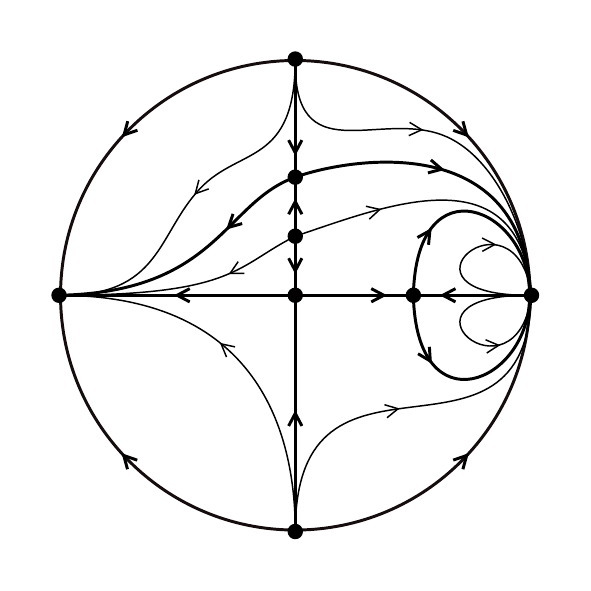}
		\caption*{\scriptsize (G17) [R=8, S=23]}
	\end{subfigure}
	\begin{subfigure}[h]{2.4cm}
		\centering
		\includegraphics[width=2.4cm]{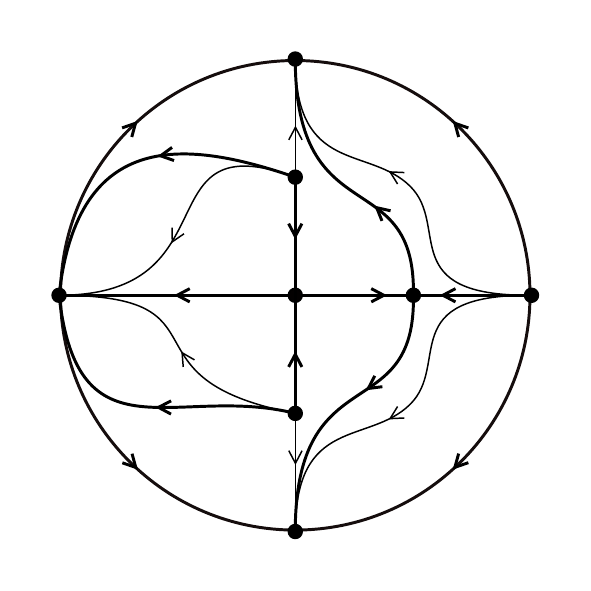}
		\caption*{\scriptsize (G18) [R=6, S=21]}
	\end{subfigure}
	\begin{subfigure}[h]{2.4cm}
		\centering
		\includegraphics[width=2.4cm]{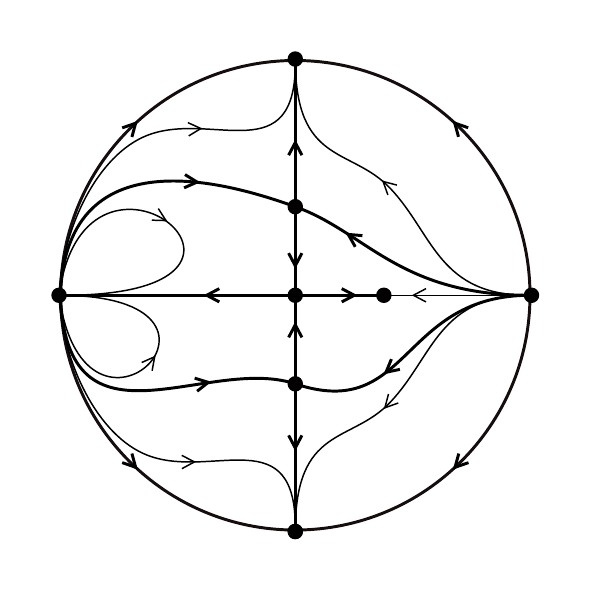}
		\caption*{\scriptsize (G19) [R=7, S=22]}
	\end{subfigure}
	\begin{subfigure}[h]{2.4cm}
		\centering
		\includegraphics[width=2.4cm]{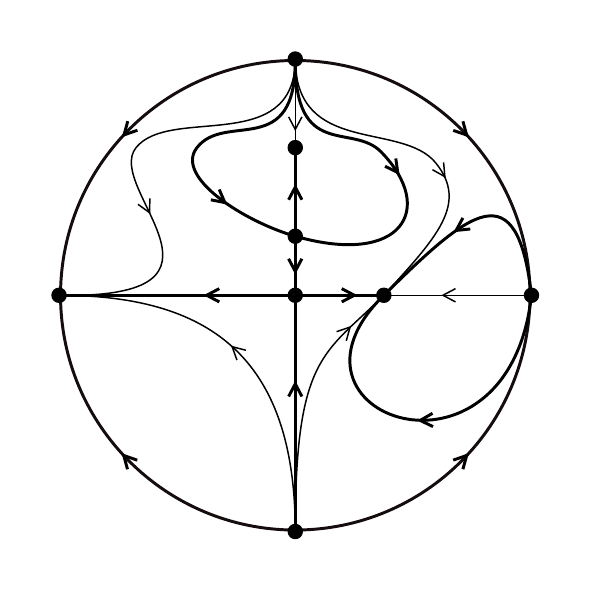}
		\caption*{\scriptsize (G20) [R=6, S=21]}
	\end{subfigure}
	\begin{subfigure}[h]{2.4cm}
	\centering
	\includegraphics[width=2.4cm]{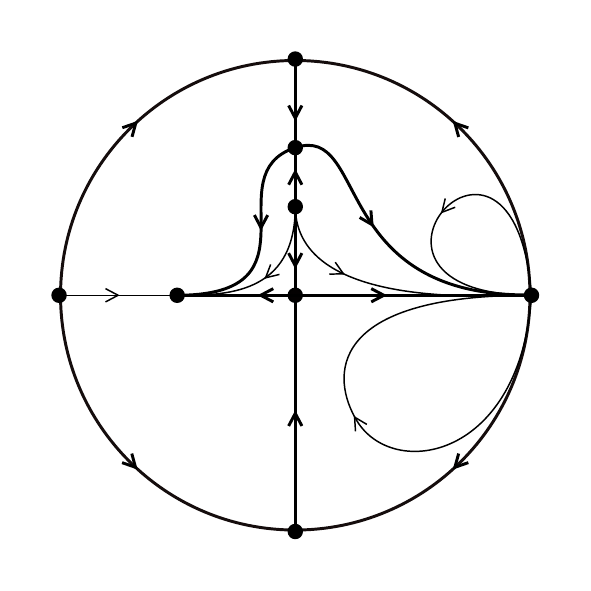}
	\caption*{\scriptsize (G21) [R=5, S=20]}
\end{subfigure}
\begin{subfigure}[h]{2.4cm}
	\centering
	\includegraphics[width=2.4cm]{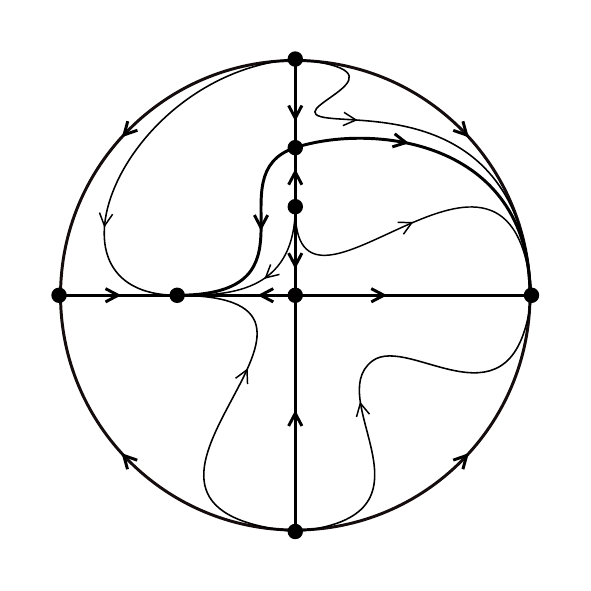}
	\caption*{\scriptsize (G22) [R=6, S=21]}
\end{subfigure}
\begin{subfigure}[h]{2.4cm}
	\centering
	\includegraphics[width=2.4cm]{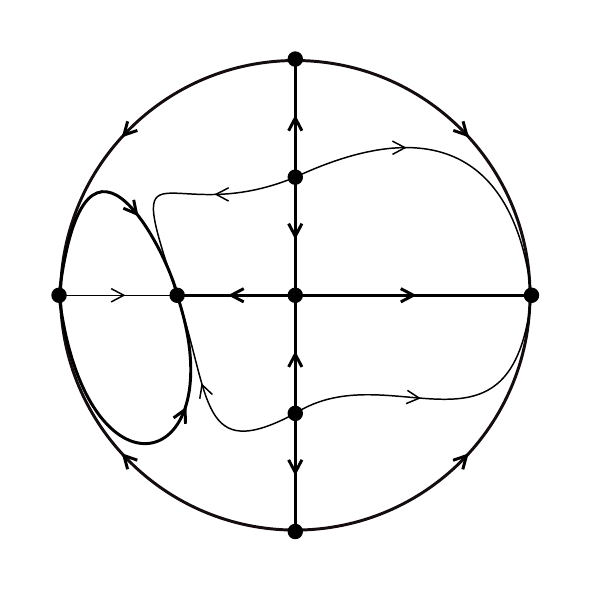}
	\caption*{\scriptsize (G23) [R=5, S=20]}
\end{subfigure}
\begin{subfigure}[h]{2.4cm}
	\centering
	\includegraphics[width=2.4cm]{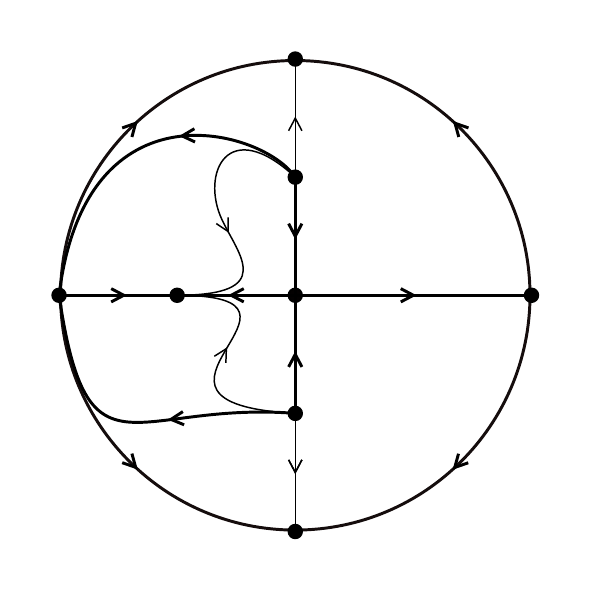}
	\caption*{\scriptsize (G24) [R=4, S=19]}
\end{subfigure}
	\begin{subfigure}[h]{2.4cm}
	\centering
	\includegraphics[width=2.4cm]{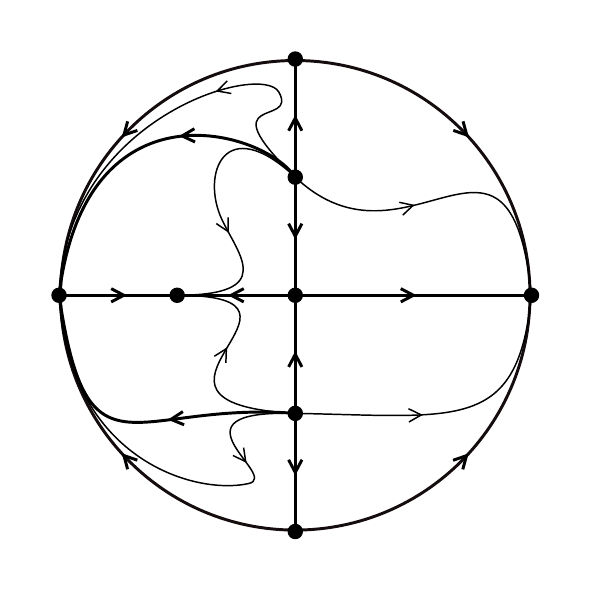}
	\caption*{\scriptsize (G25) [R=6, S=21]}
\end{subfigure}
	\begin{subfigure}[h]{2.4cm}
	\centering
	\includegraphics[width=2.4cm]{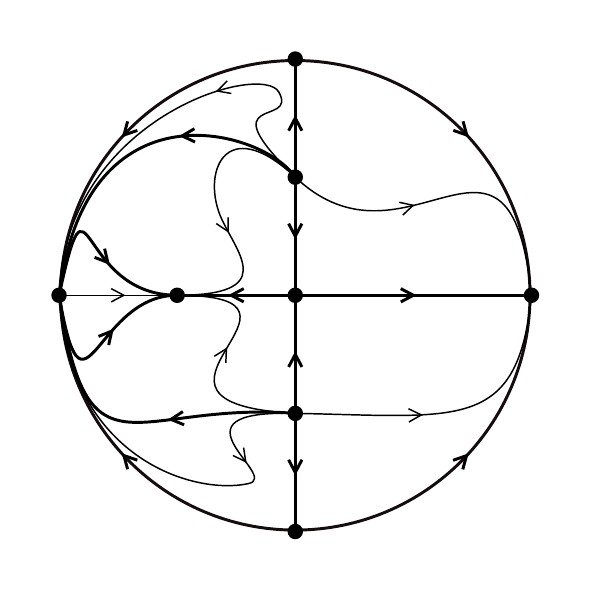}
	\caption*{\scriptsize (G26) [R=7, S=22]}
\end{subfigure}
\begin{subfigure}[h]{2.4cm}
	\centering
	\includegraphics[width=2.4cm]{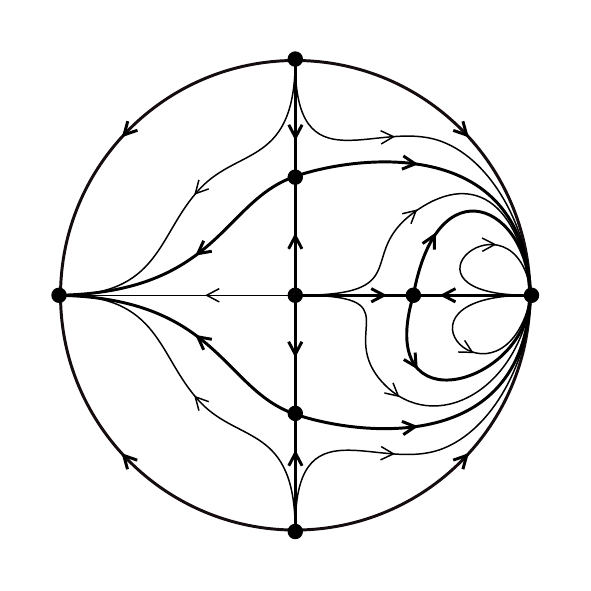}
	\caption*{\scriptsize (G27) [R=9, S=24]}
\end{subfigure}
\begin{subfigure}[h]{2.4cm}
	\centering
	\includegraphics[width=2.4cm]{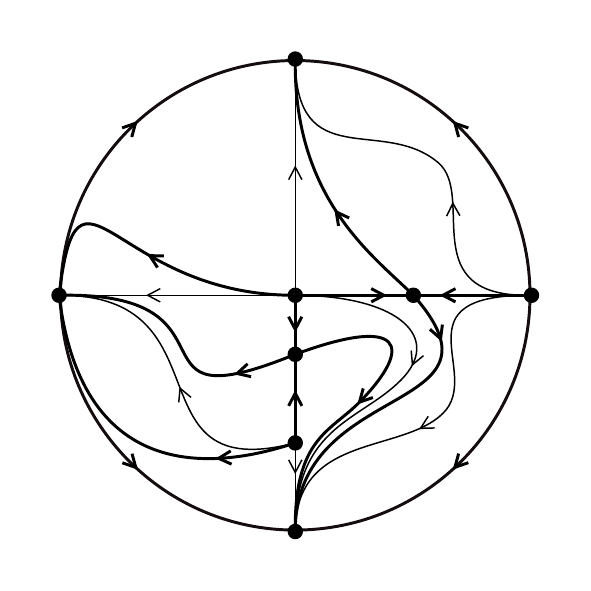}
	\caption*{\scriptsize (G28) [R=7, S=22]}
\end{subfigure}
\begin{subfigure}[h]{2.4cm}
	\centering
	\includegraphics[width=2.4cm]{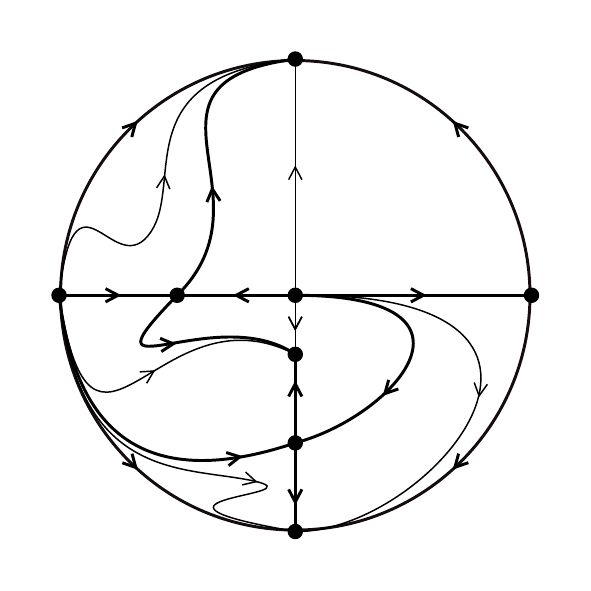}
	\caption*{\scriptsize (G29) [R=6, S=21]}
\end{subfigure}
\begin{subfigure}[h]{2.4cm}
	\centering
	\includegraphics[width=2.4cm]{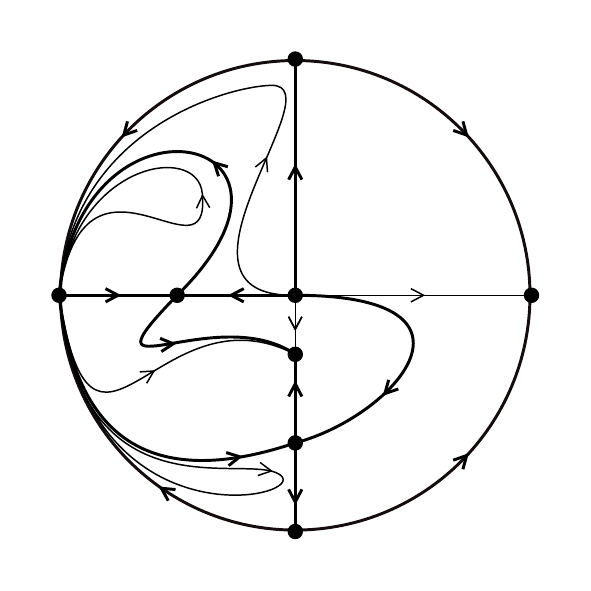}
	\caption*{\scriptsize (G30) [R=6, S=21]}
\end{subfigure}
\begin{subfigure}[h]{2.4cm}
	\centering
	\includegraphics[width=2.4cm]{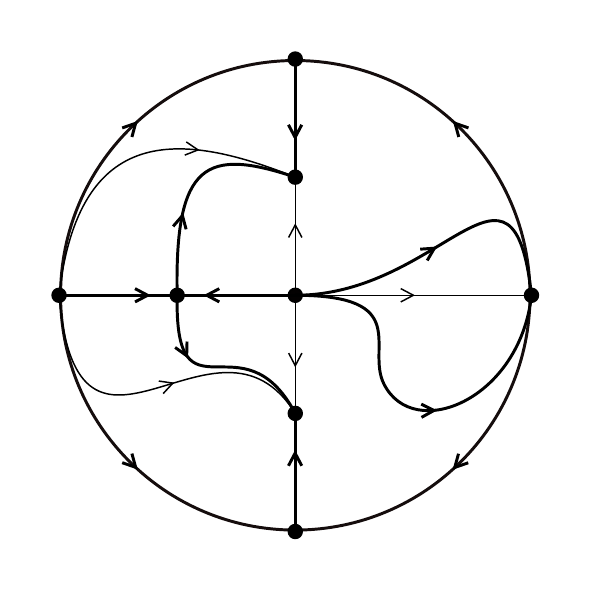}
	\caption*{\scriptsize (G31) [R=5, S=20]}
\end{subfigure}
\begin{subfigure}[h]{2.4cm}
	\centering
	\includegraphics[width=2.4cm]{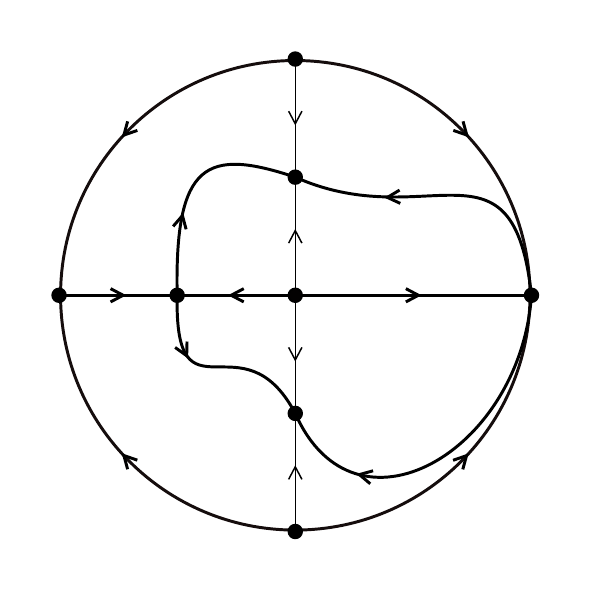}
	\caption*{\scriptsize (G32) [R=4, S=19]}
\end{subfigure}
\begin{subfigure}[h]{2.4cm}
	\centering
	\includegraphics[width=2.4cm]{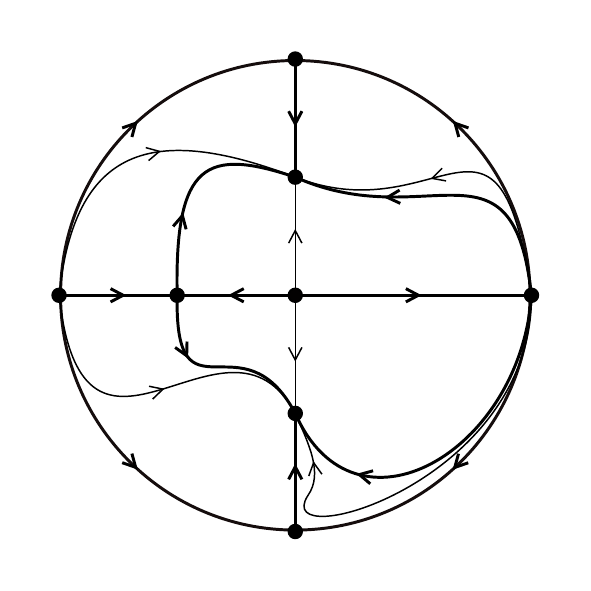}
	\caption*{\scriptsize (G33) [R=6, S=21]}
\end{subfigure}
\begin{subfigure}[h]{2.4cm}
	\centering
	\includegraphics[width=2.4cm]{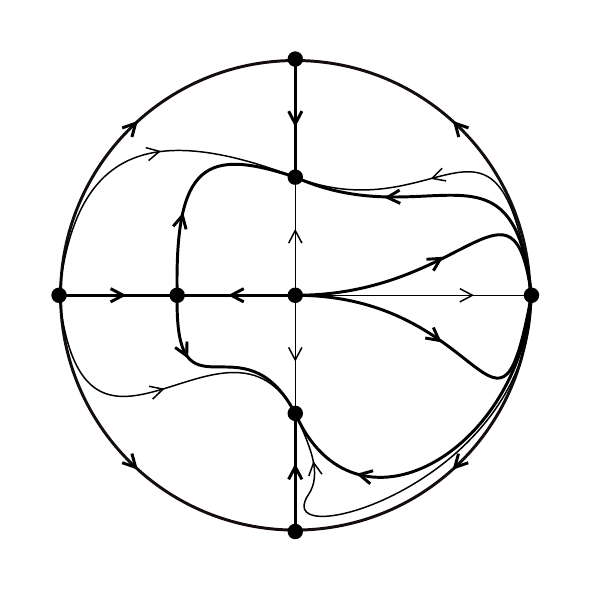}
	\caption*{\scriptsize (G34) [R=7, S=22]}
\end{subfigure}
\begin{subfigure}[h]{2.4cm}
	\centering
	\includegraphics[width=2.4cm]{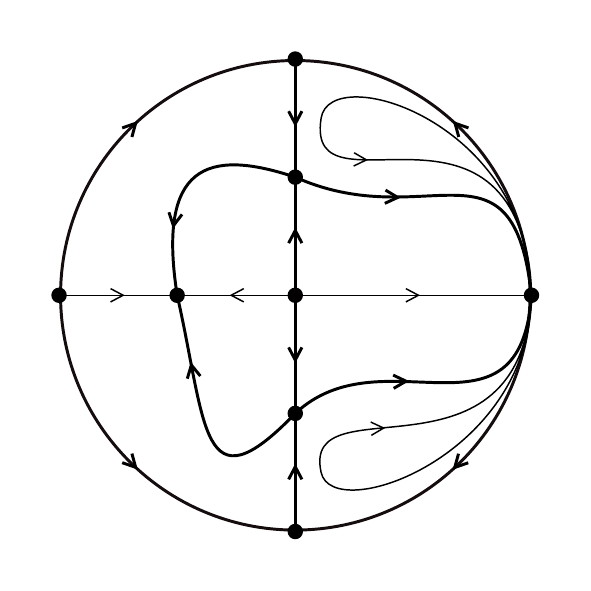}
	\caption*{\scriptsize (G35) [R=5, S=20]}
\end{subfigure}
\end{figure}

\begin{figure}[H]
	\centering
	\ContinuedFloat

	\begin{subfigure}[h]{2.4cm}
		\centering
		\includegraphics[width=2.4cm]{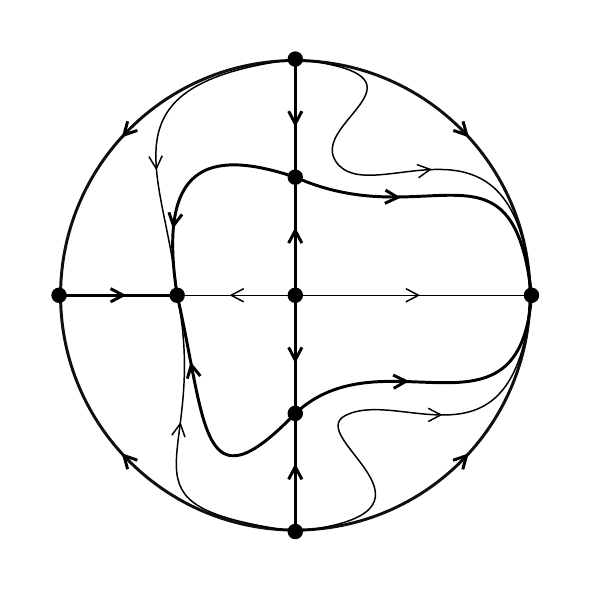}
		\caption*{\scriptsize (G36) [R=6, S=21]}
	\end{subfigure}
	\begin{subfigure}[h]{2.4cm}
		\centering
		\includegraphics[width=2.4cm]{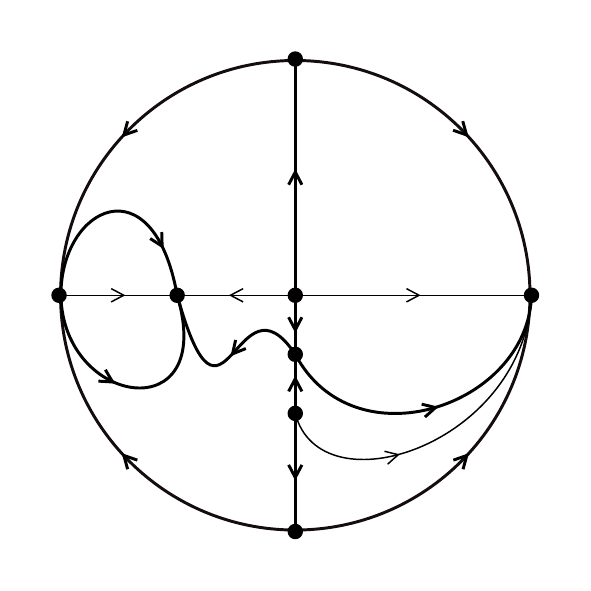}
		\caption*{\scriptsize (G37) [R=5, S=20]}
	\end{subfigure}
	\begin{subfigure}[h]{2.4cm}
		\centering
		\includegraphics[width=2.4cm]{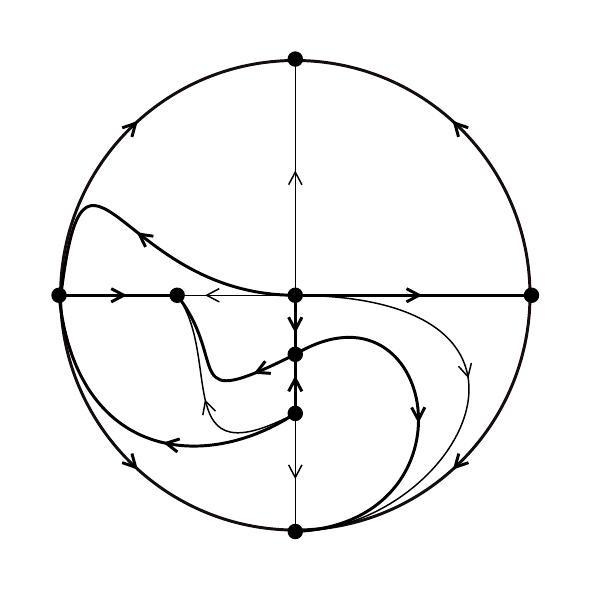}
		\caption*{\scriptsize (G38) [R=5, S=20]}
	\end{subfigure}
	\begin{subfigure}[h]{2.4cm}
		\centering
		\includegraphics[width=2.4cm]{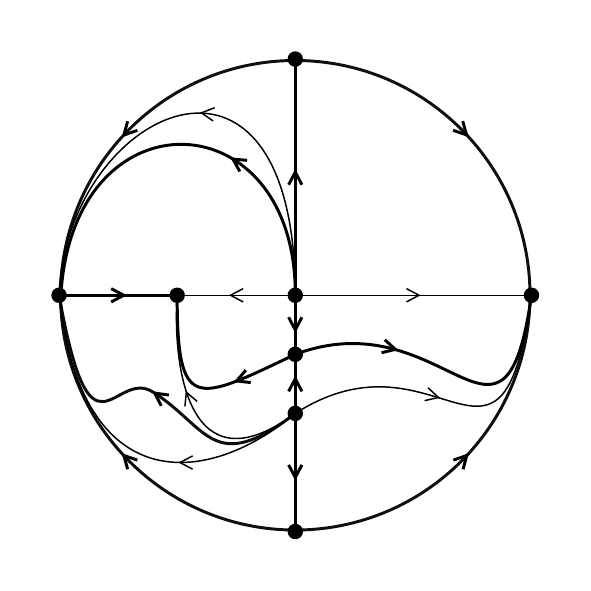}
		\caption*{\scriptsize (G39) [R=6, S=21]}
	\end{subfigure}
	\begin{subfigure}[h]{2.4cm}
		\centering
		\includegraphics[width=2.4cm]{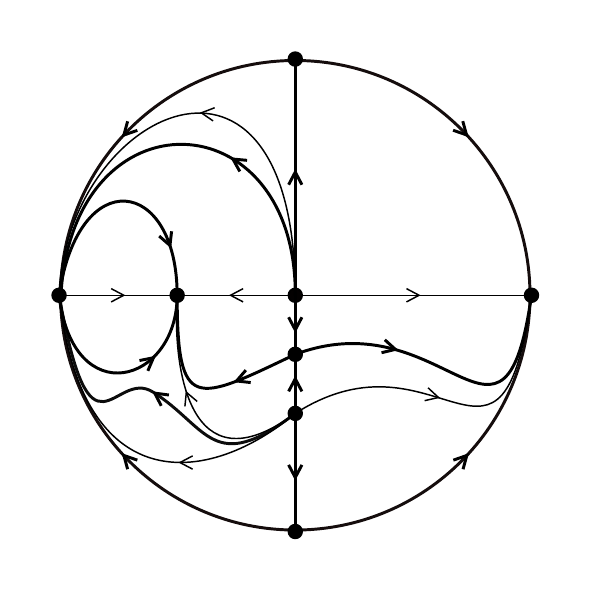}
		\caption*{\scriptsize (G40) [R=7, S=22]}
	\end{subfigure}
	\begin{subfigure}[h]{2.4cm}
	\centering
	\includegraphics[width=2.4cm]{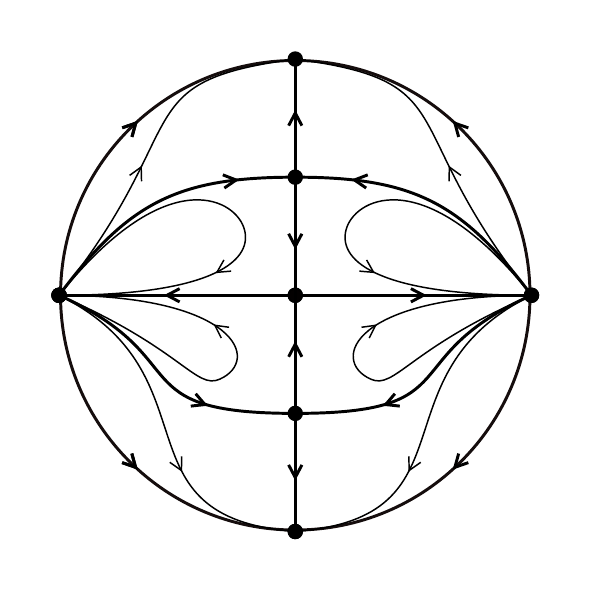}
	\caption*{\scriptsize (G41) [R=8, S=21]}
\end{subfigure}
\begin{subfigure}[h]{2.4cm}
	\centering
	\includegraphics[width=2.4cm]{chap2/global/G42}
	\caption*{\scriptsize (G42) [R=7, S=20]}
\end{subfigure}
\begin{subfigure}[h]{2.4cm}
	\centering
	\includegraphics[width=2.4cm]{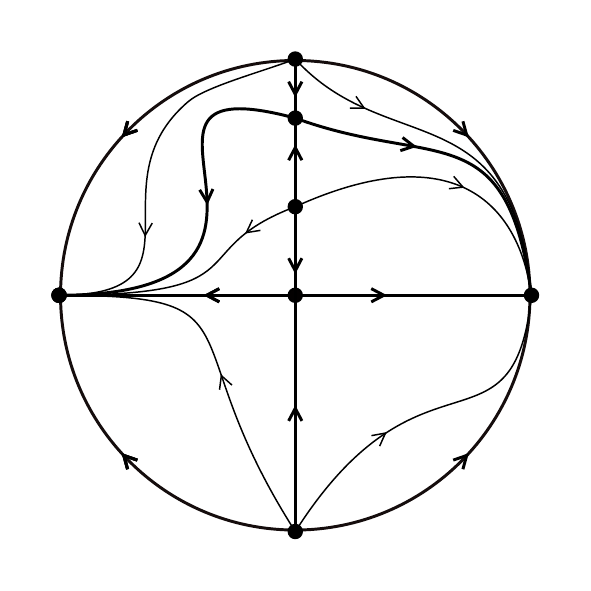}
	\caption*{\scriptsize (G43) [R=6, S=19]}
\end{subfigure}
\begin{subfigure}[h]{2.4cm}
	\centering
	\includegraphics[width=2.4cm]{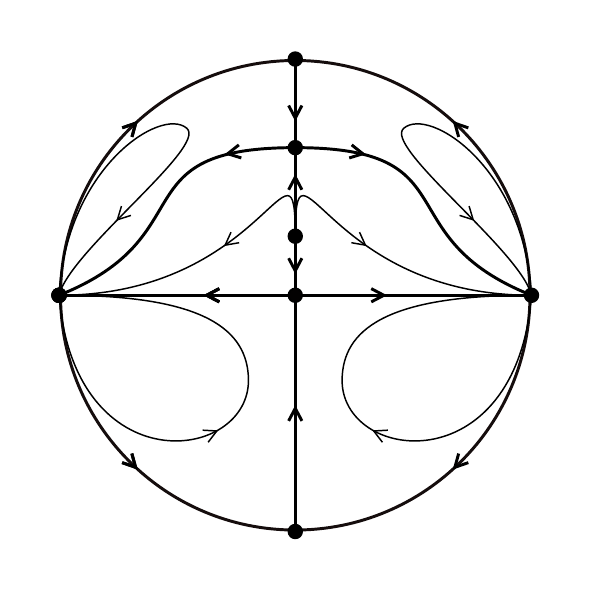}
	\caption*{\scriptsize (G44) [R=6, S=19]}
\end{subfigure}
\begin{subfigure}[h]{2.4cm}
	\centering
	\includegraphics[width=2.4cm]{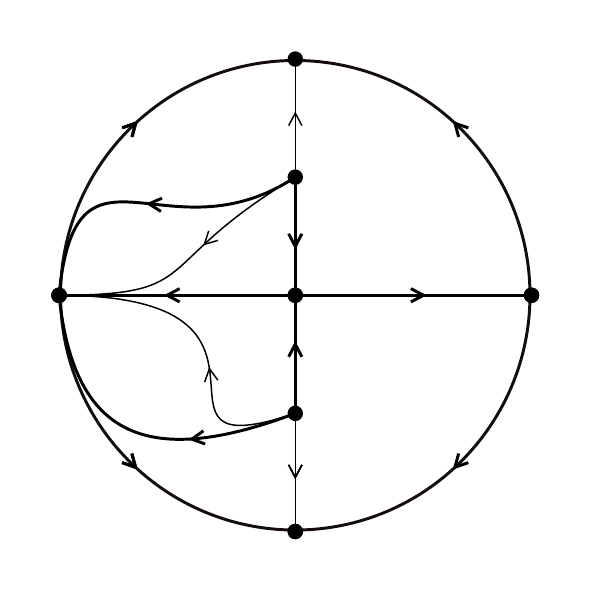}
	\caption*{\scriptsize (G45) [R=4, S=17]}
\end{subfigure}
\begin{subfigure}[h]{2.4cm}
	\centering
	\includegraphics[width=2.4cm]{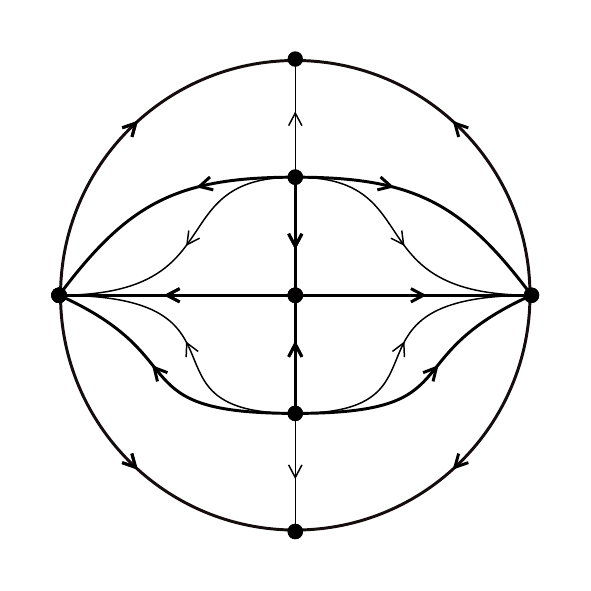}
	\caption*{\scriptsize (G46) [R=6, S=19]}
\end{subfigure}
\begin{subfigure}[h]{2.4cm}
	\centering
	\includegraphics[width=2.4cm]{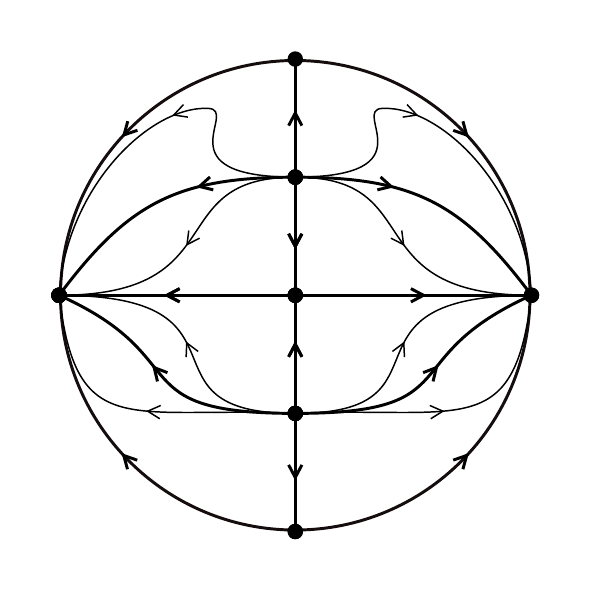}
	\caption*{\scriptsize (G47) [R=8, S=21]}
\end{subfigure}
	\begin{subfigure}[h]{2.4cm}
	\centering
	\includegraphics[width=2.4cm]{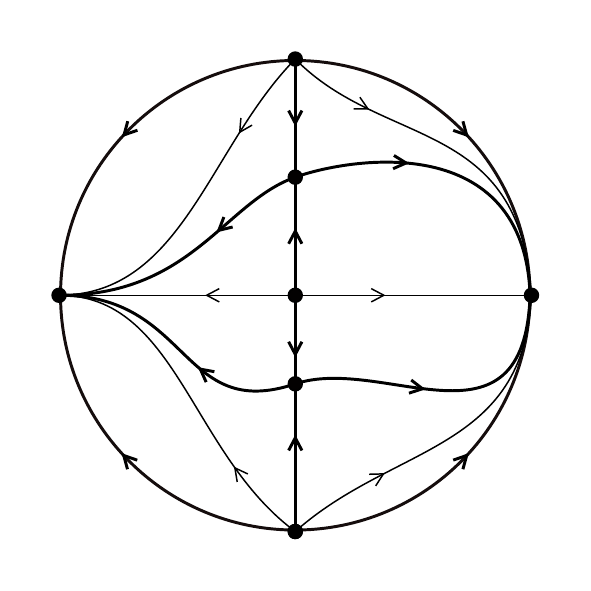}
	\caption*{\scriptsize (G48) [R=6, S=19]}
\end{subfigure}
	\begin{subfigure}[h]{2.4cm}
	\centering
	\includegraphics[width=2.4cm]{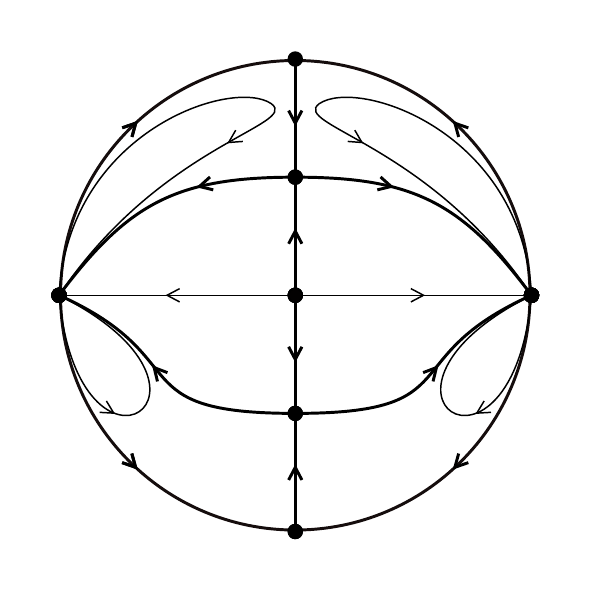}
	\caption*{\scriptsize (G49) [R=6, S=19]}
\end{subfigure}
\begin{subfigure}[h]{2.4cm}
	\centering
	\includegraphics[width=2.4cm]{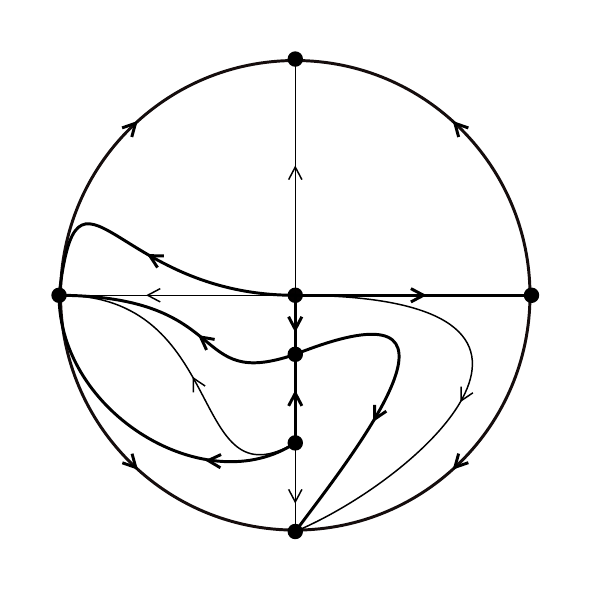}
	\caption*{\scriptsize (G50) [R=5, S=18]}
\end{subfigure}
	\begin{subfigure}[h]{2.4cm}
	\centering
	\includegraphics[width=2.4cm]{chap2/global/G51}
	\caption*{\scriptsize (G51) [R=6, S=19]}
\end{subfigure}
\begin{subfigure}[h]{2.4cm}
	\centering
	\includegraphics[width=2.4cm]{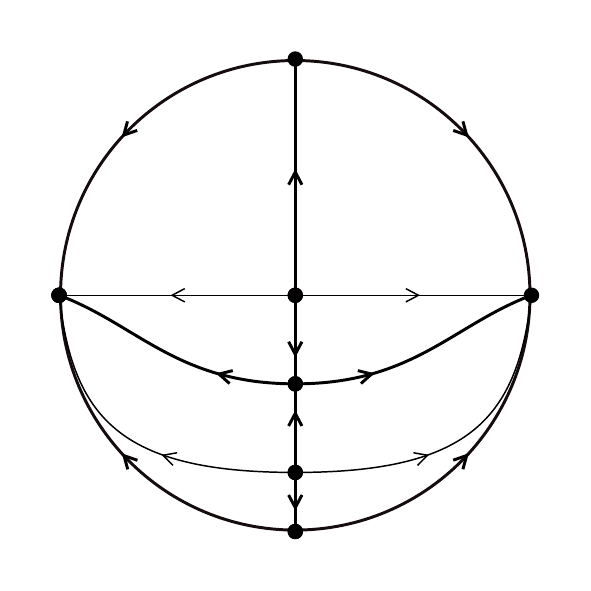}
	\caption*{\scriptsize (G52) [R=4, S=17]}
\end{subfigure}
\begin{subfigure}[h]{2.4cm}
	\centering
	\includegraphics[width=2.4cm]{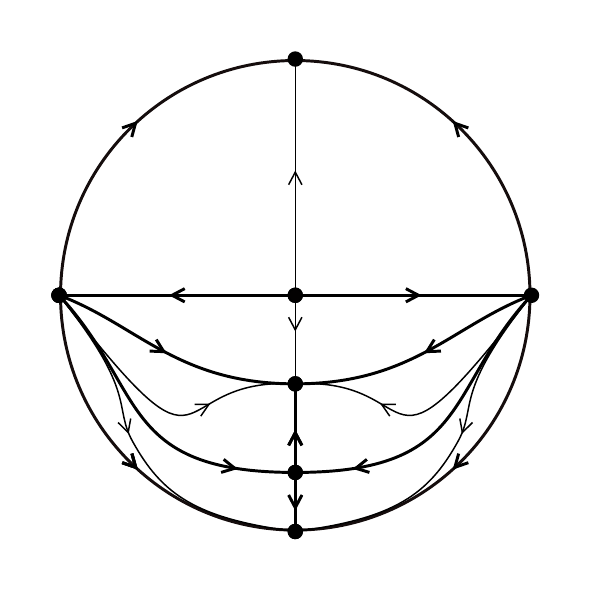}
	\caption*{\scriptsize (G53) [R=6, S=19]}
\end{subfigure}
\begin{subfigure}[h]{2.4cm}
	\centering
	\includegraphics[width=2.4cm]{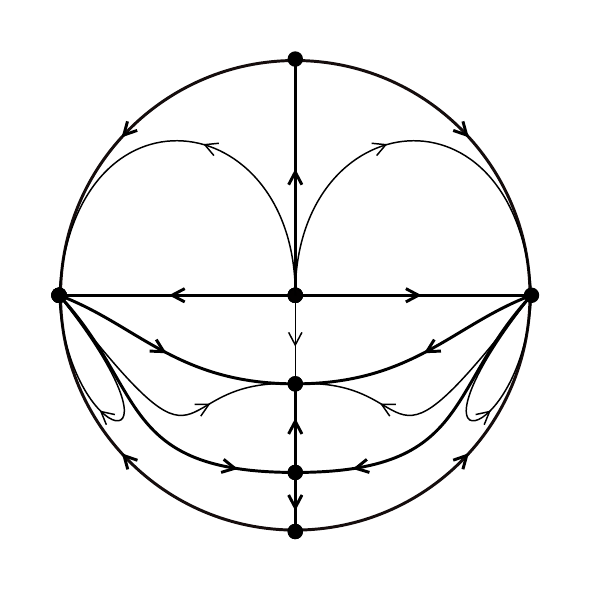}
	\caption*{\scriptsize (G54)[R=7, S=20]}
\end{subfigure}
\begin{subfigure}[h]{2.4cm}
	\centering
	\includegraphics[width=2.4cm]{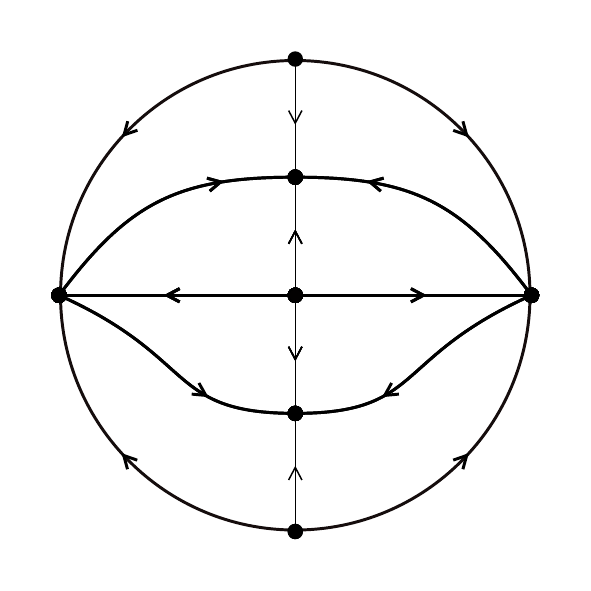}
	\caption*{\scriptsize (G55) [R=4, S=17]}
\end{subfigure}
\begin{subfigure}[h]{2.4cm}
	\centering
	\includegraphics[width=2.4cm]{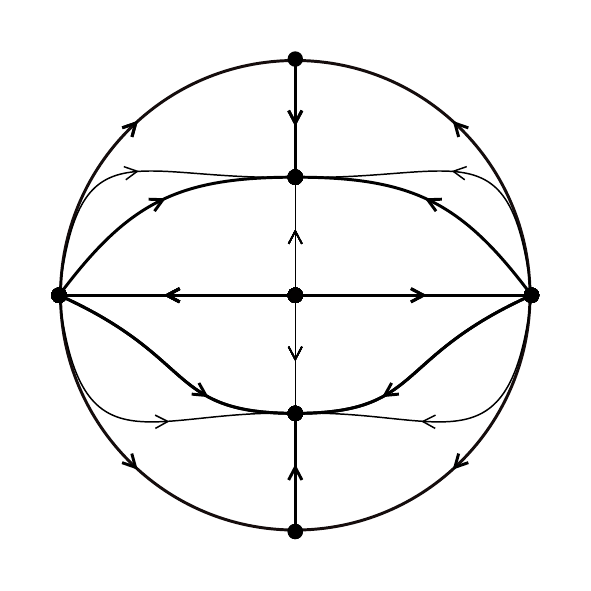}
	\caption*{\scriptsize (G56) [R=6, S=19]}
\end{subfigure}
\begin{subfigure}[h]{2.4cm}
	\centering
	\includegraphics[width=2.4cm]{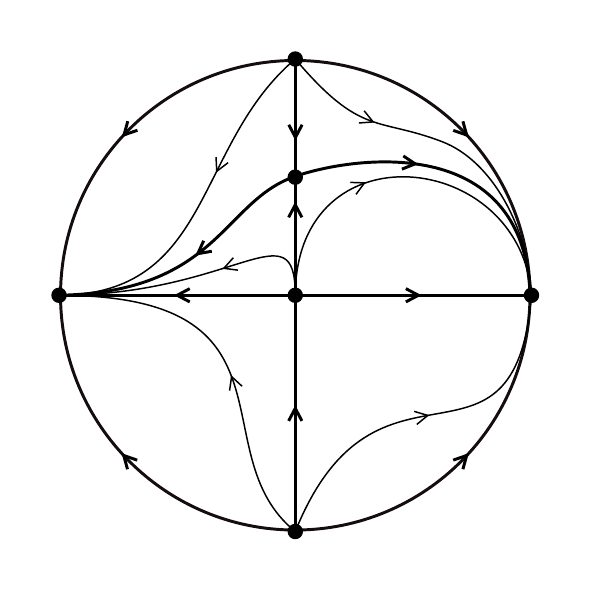}
	\caption*{\scriptsize (G57) [R=6, S=17]}
\end{subfigure}
\begin{subfigure}[h]{2.4cm}
	\centering
	\includegraphics[width=2.4cm]{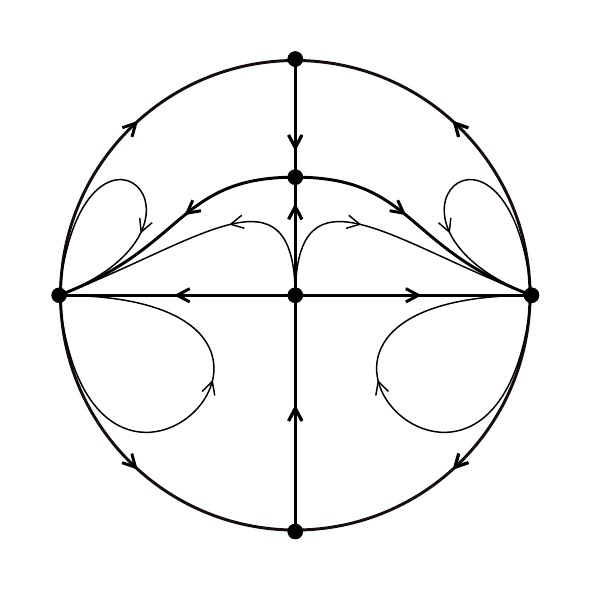}
	\caption*{\scriptsize (G58) [R=6, S=17]}
\end{subfigure}
\begin{subfigure}[h]{2.4cm}
	\centering
	\includegraphics[width=2.4cm]{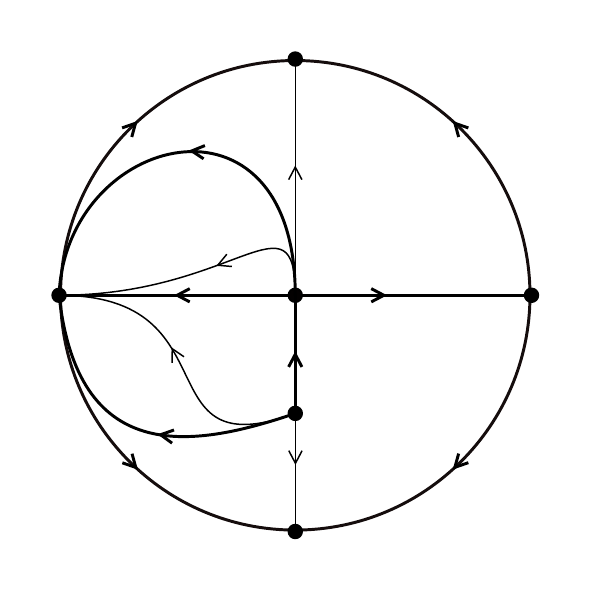}
	\caption*{\scriptsize (G59) [R=4, S=15]}
\end{subfigure}
\begin{subfigure}[h]{2.4cm}
	\centering
	\includegraphics[width=2.4cm]{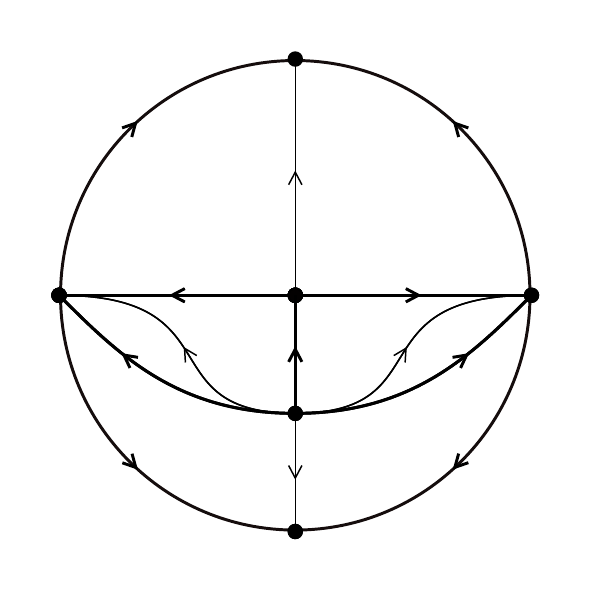}
	\caption*{\scriptsize (G60) [R=4, S=15]}
\end{subfigure}
\begin{subfigure}[h]{2.4cm}
	\centering
	\includegraphics[width=2.4cm]{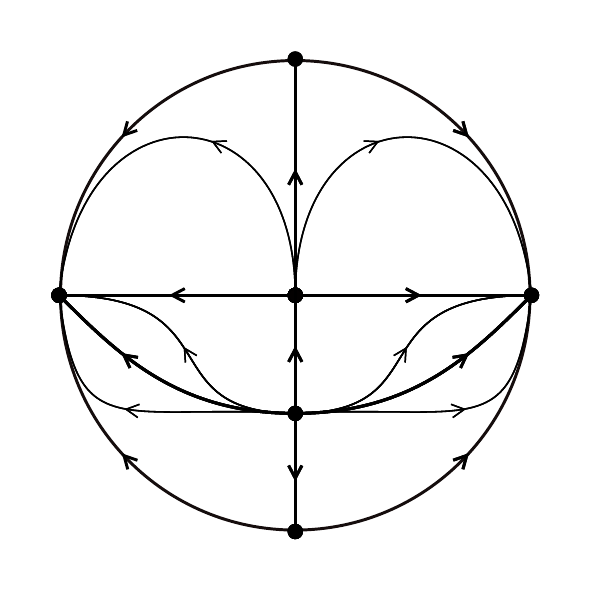}
	\caption*{\scriptsize (G61) [R=6, S=17]}
\end{subfigure}
\begin{subfigure}[h]{2.4cm}
	\centering
	\includegraphics[width=2.4cm]{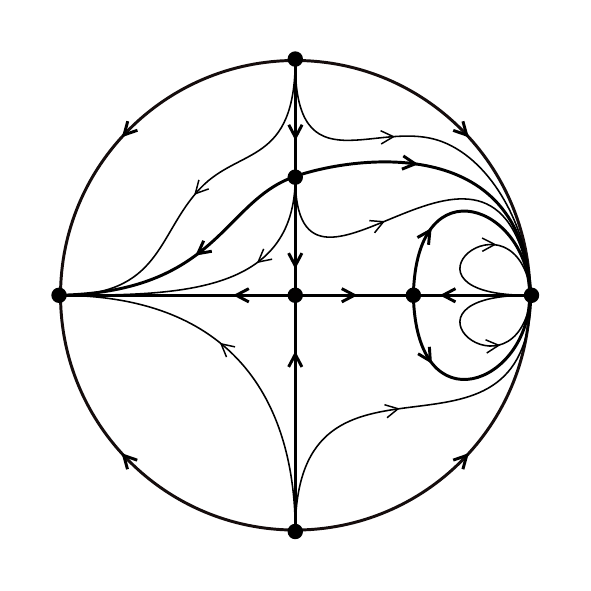}
	\caption*{\scriptsize (G62) [R=8, S=21]}
\end{subfigure}
\begin{subfigure}[h]{2.4cm}
	\centering
	\includegraphics[width=2.4cm]{chap2/global/G63}
	\caption*{\scriptsize (G63) [R=6, S=19]}
\end{subfigure}
\begin{subfigure}[h]{2.4cm}
	\centering
	\includegraphics[width=2.4cm]{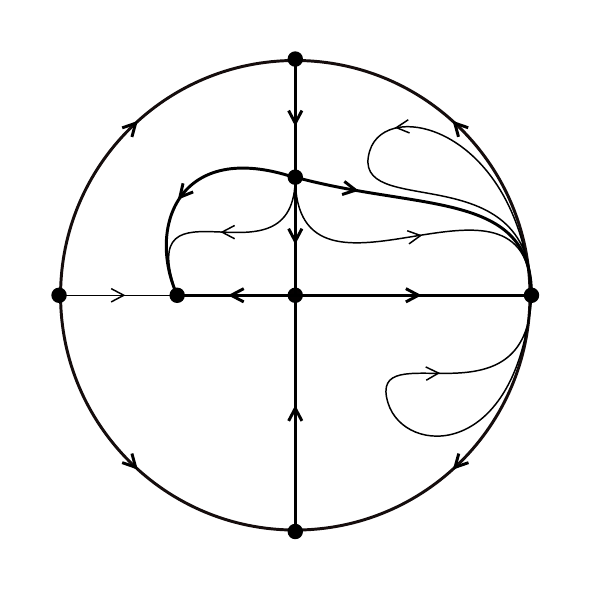}
	\caption*{\scriptsize (G64) [R=5, S=18]}
\end{subfigure}
\begin{subfigure}[h]{2.4cm}
	\centering
	\includegraphics[width=2.4cm]{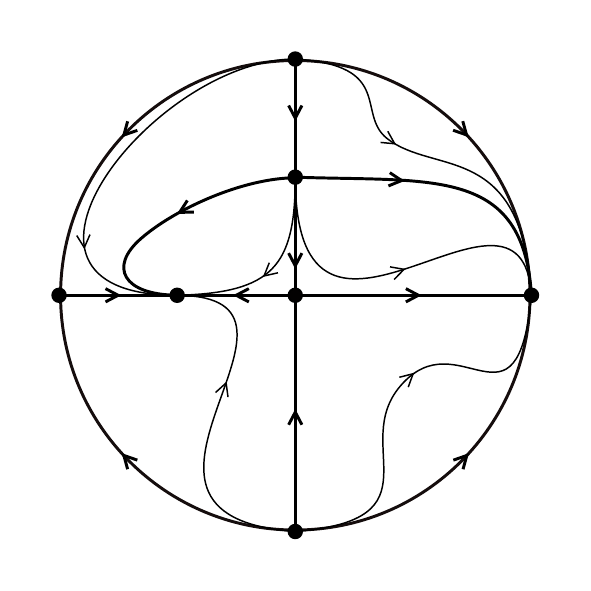}
	\caption*{\scriptsize (G65) [R=6, S=19]}
\end{subfigure}
\begin{subfigure}[h]{2.4cm}
	\centering
	\includegraphics[width=2.4cm]{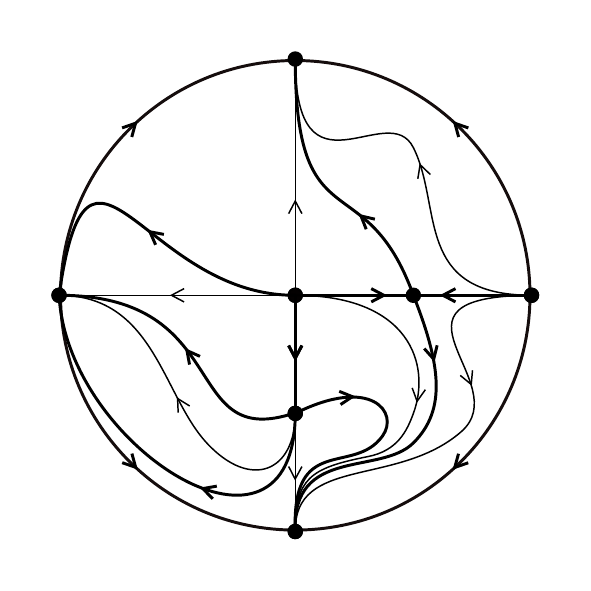}
	\caption*{\scriptsize (G66) [R=7, S=20]}
\end{subfigure}
\begin{subfigure}[h]{2.4cm}
	\centering
	\includegraphics[width=2.4cm]{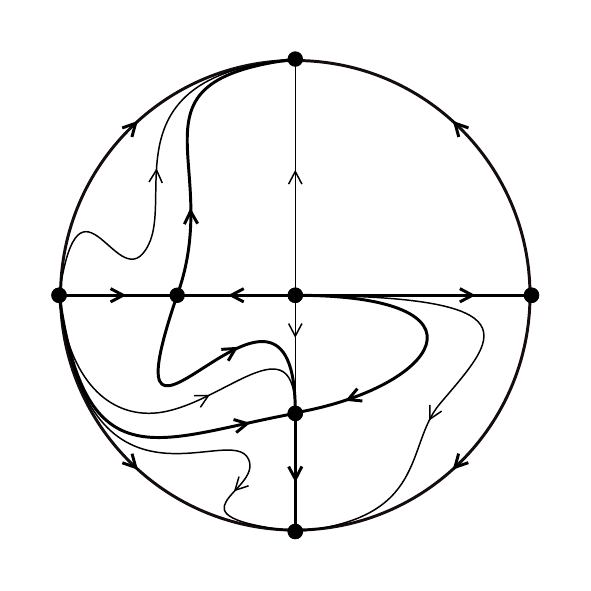}
	\caption*{\scriptsize (G67) [R=6, S=19]}
\end{subfigure}
\begin{subfigure}[h]{2.4cm}
	\centering
	\includegraphics[width=2.4cm]{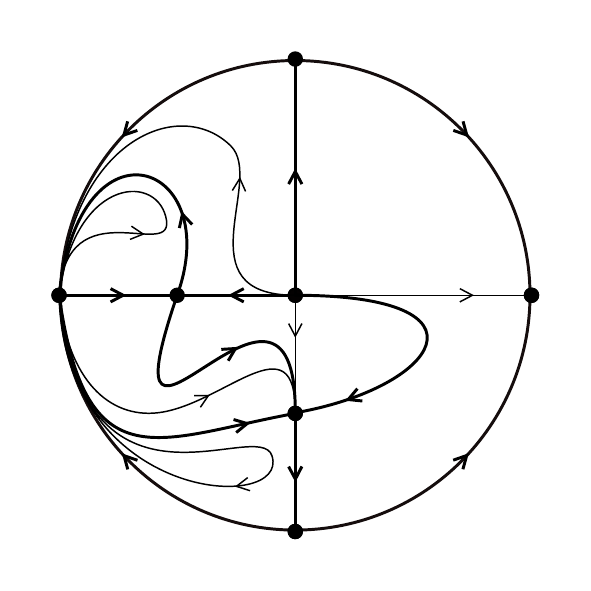}
	\caption*{\scriptsize (G68) [R=6, S=19]}
\end{subfigure}
\begin{subfigure}[h]{2.4cm}
	\centering
	\includegraphics[width=2.4cm]{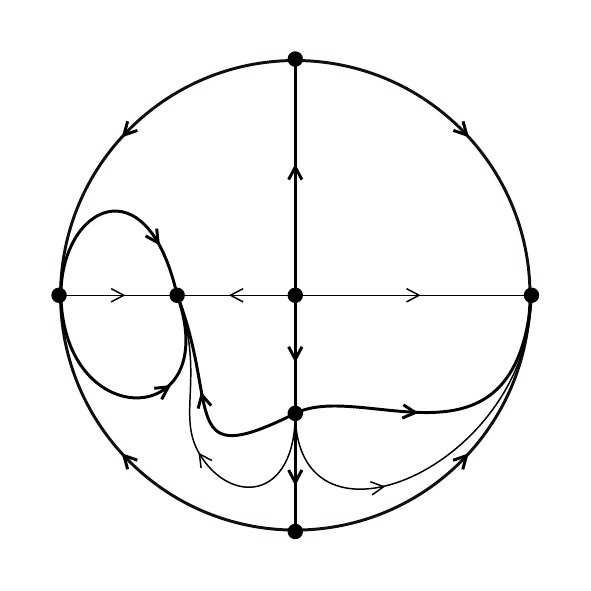}
	\caption*{\scriptsize (G69) [R=5, S=18]}
\end{subfigure}
\begin{subfigure}[h]{2.4cm}
	\centering
	\includegraphics[width=2.4cm]{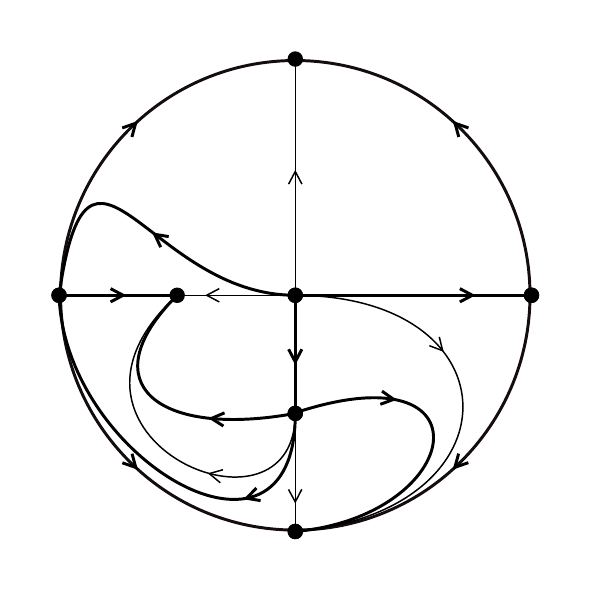}
	\caption*{\scriptsize (G70) [R=5, S=18]}
\end{subfigure}
\end{figure}

\begin{figure}[H]
	\centering
	\ContinuedFloat
	\begin{subfigure}[h]{2.4cm}
		\centering
		\includegraphics[width=2.4cm]{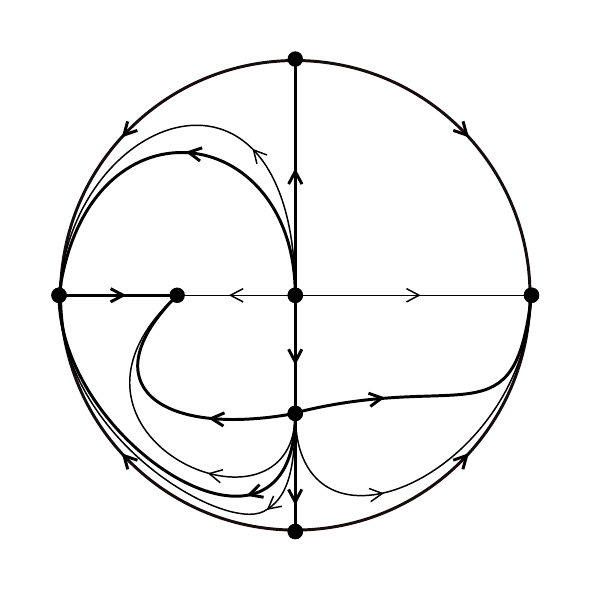}
		\caption*{\scriptsize (G71) [R=6, S=19]}
	\end{subfigure}
	\begin{subfigure}[h]{2.4cm}
		\centering
		\includegraphics[width=2.4cm]{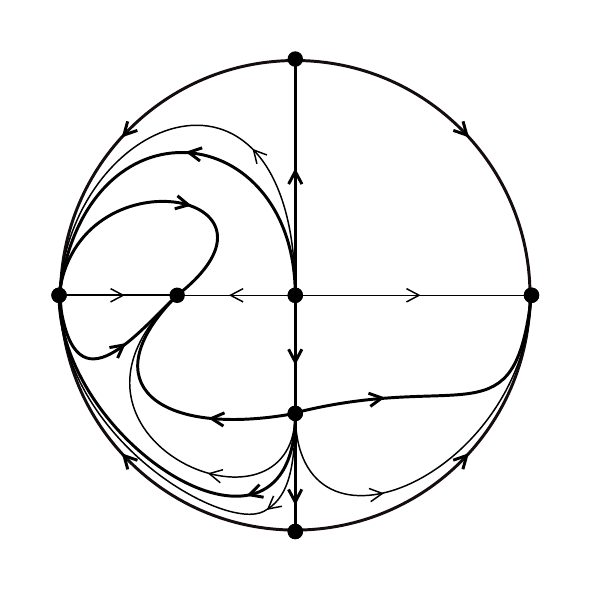}
		\caption*{\scriptsize (G72) [R=7, S=20]}
	\end{subfigure}
		\begin{subfigure}[h]{2.4cm}
		\centering
		\includegraphics[width=2.4cm]{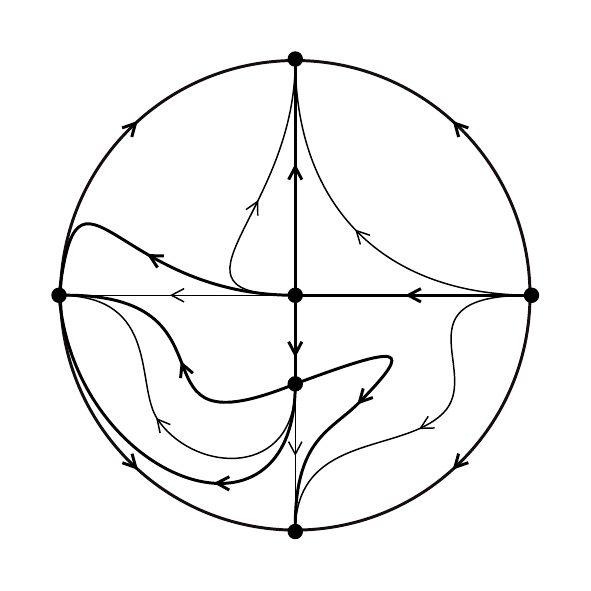}
		\caption*{\scriptsize (G73) [R=6, S=17]}
	\end{subfigure}
	\begin{subfigure}[h]{2.4cm}
		\centering
		\includegraphics[width=2.4cm]{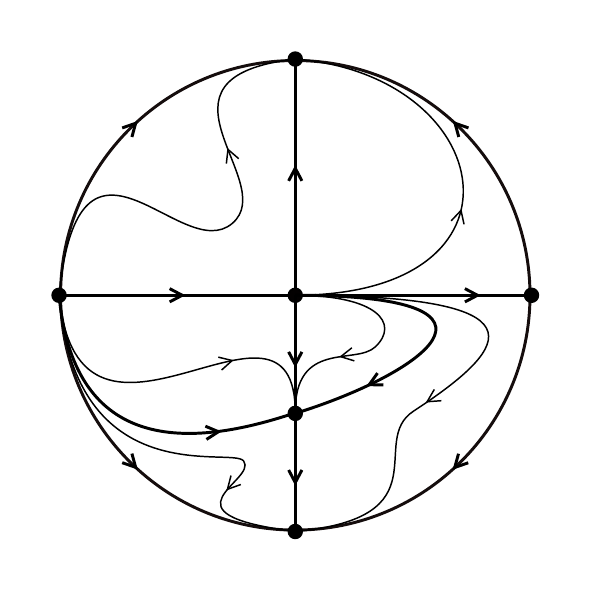}
		\caption*{\scriptsize (G74) [R=6, S=17]}
	\end{subfigure}
	\begin{subfigure}[h]{2.4cm}
		\centering
		\includegraphics[width=2.4cm]{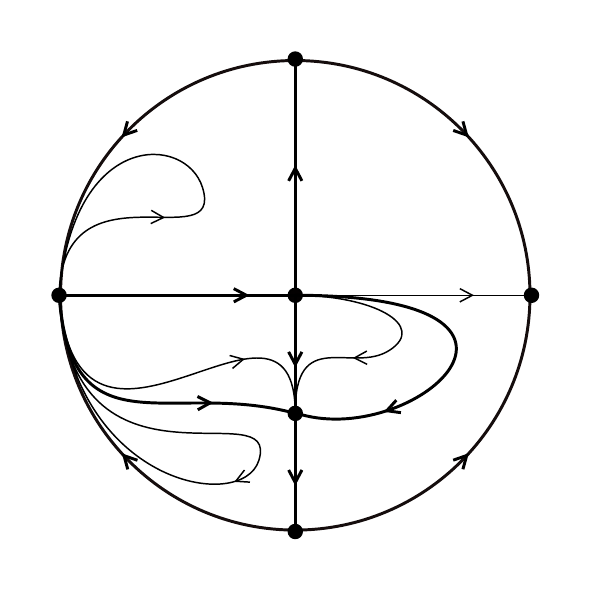}
		\caption*{\scriptsize (G75) [R=5, S=16]}
	\end{subfigure}
	\begin{subfigure}[h]{2.4cm}
		\centering
		\includegraphics[width=2.4cm]{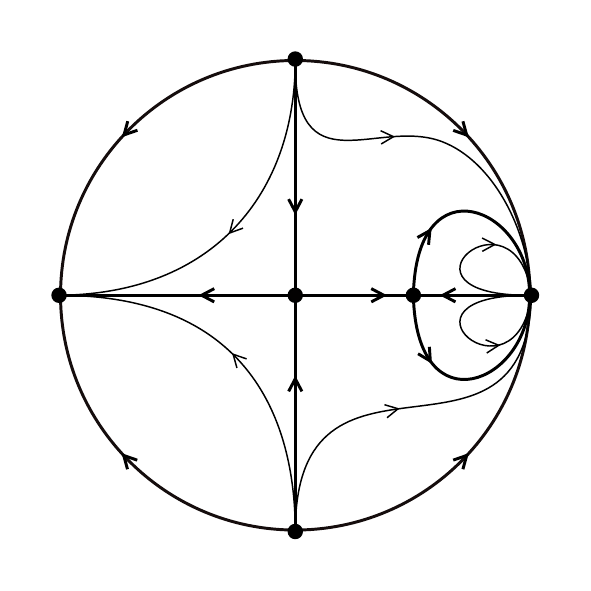}
		\caption*{\scriptsize (G76) [R=6, S=17]}
	\end{subfigure}
	\begin{subfigure}[h]{2.4cm}
		\centering
		\includegraphics[width=2.4cm]{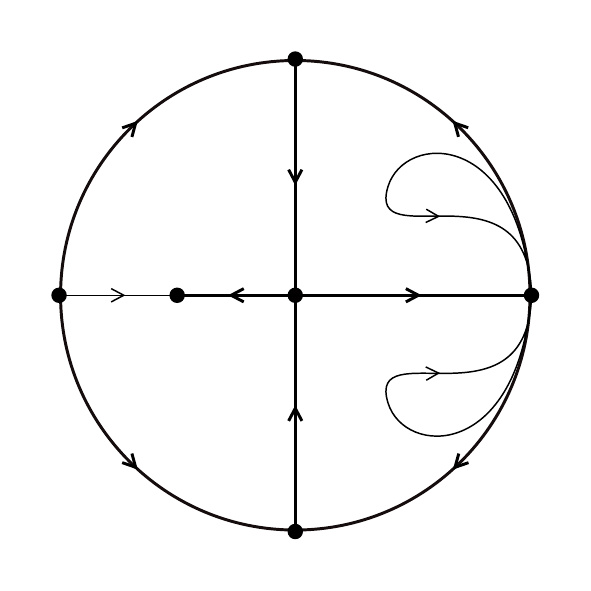}
		\caption*{\scriptsize (G77) [R=3, S=14]}
	\end{subfigure}
	\begin{subfigure}[h]{2.4cm}
		\centering
		\includegraphics[width=2.4cm]{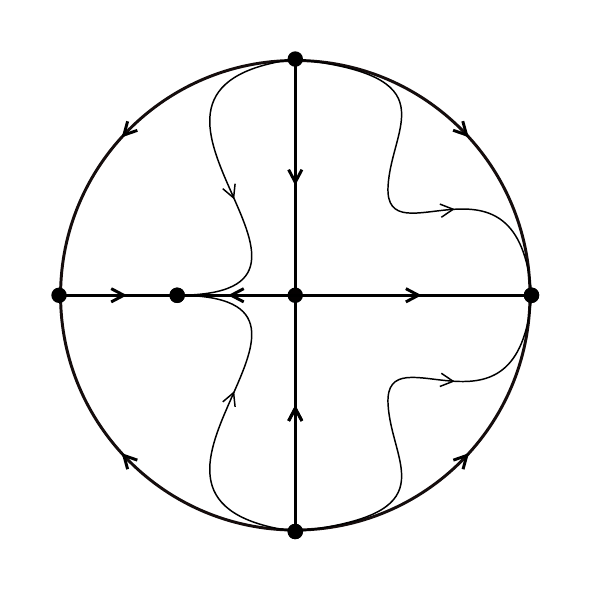}
		\caption*{\scriptsize (G78) [R=4, S=15]}
	\end{subfigure}
	\begin{subfigure}[h]{2.4cm}
		\centering
		\includegraphics[width=2.4cm]{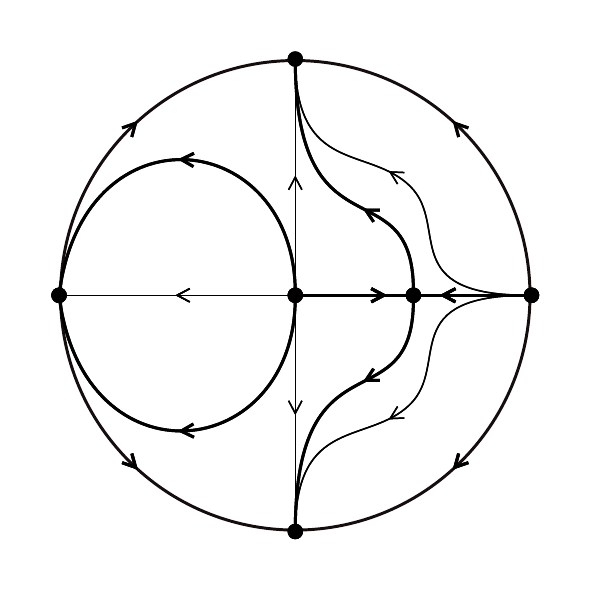}
		\caption*{\scriptsize (G79) [R=5, S=16]}
	\end{subfigure}
	\begin{subfigure}[h]{2.4cm}
		\centering
		\includegraphics[width=2.4cm]{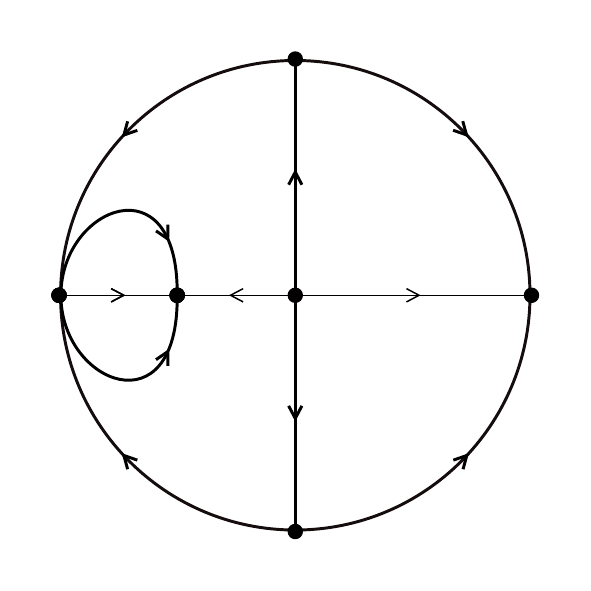}
		\caption*{\scriptsize (G80) [R=3, S=14]}
	\end{subfigure}
	\begin{subfigure}[h]{2.4cm}
		\centering
		\includegraphics[width=2.4cm]{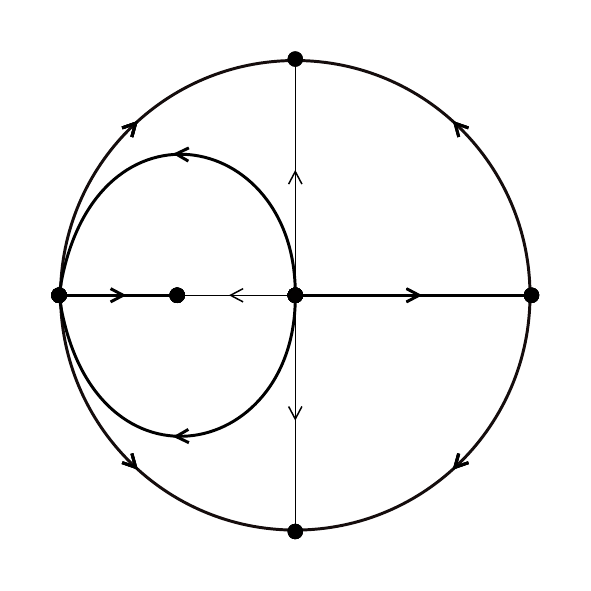}
		\caption*{\scriptsize (G81) [R=3, S=14]}
	\end{subfigure}
	\begin{subfigure}[h]{2.4cm}
		\centering
		\includegraphics[width=2.4cm]{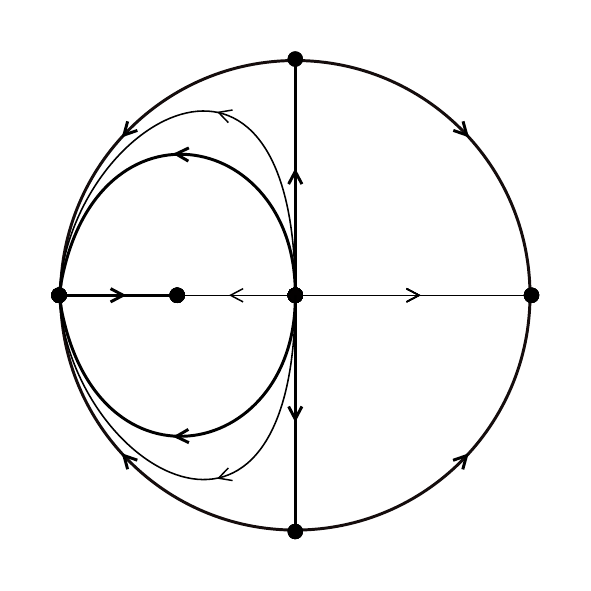}
		\caption*{\scriptsize (G82) [R=4, S=15]}
	\end{subfigure}
	\begin{subfigure}[h]{2.4cm}
		\centering
		\includegraphics[width=2.4cm]{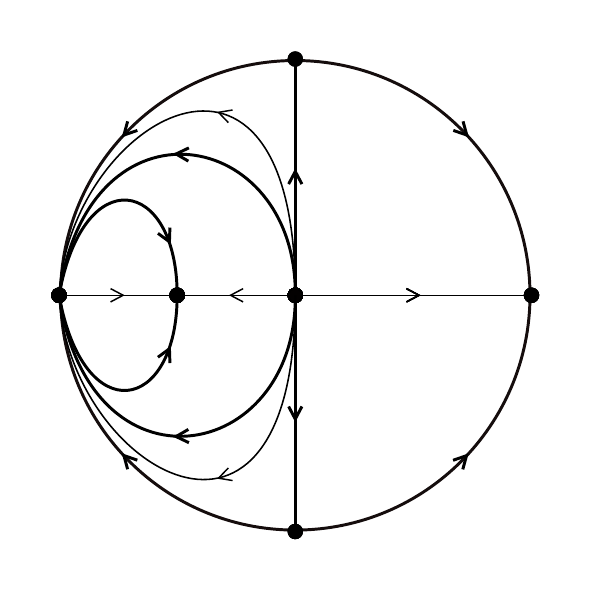}
		\caption*{\scriptsize (G83) [R=5, S=16]}
	\end{subfigure}
	\begin{subfigure}[h]{2.4cm}
		\centering
		\includegraphics[width=2.4cm]{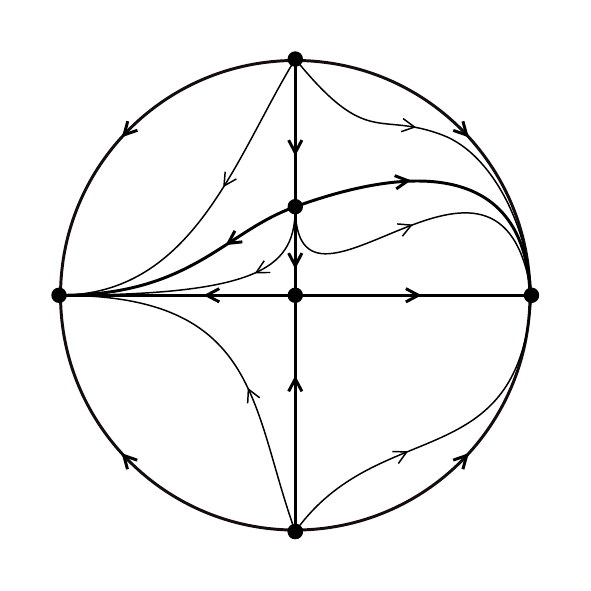}
		\caption*{\scriptsize (G84) [R=6, S=17]}
	\end{subfigure}
	\begin{subfigure}[h]{2.4cm}
		\centering
		\includegraphics[width=2.4cm]{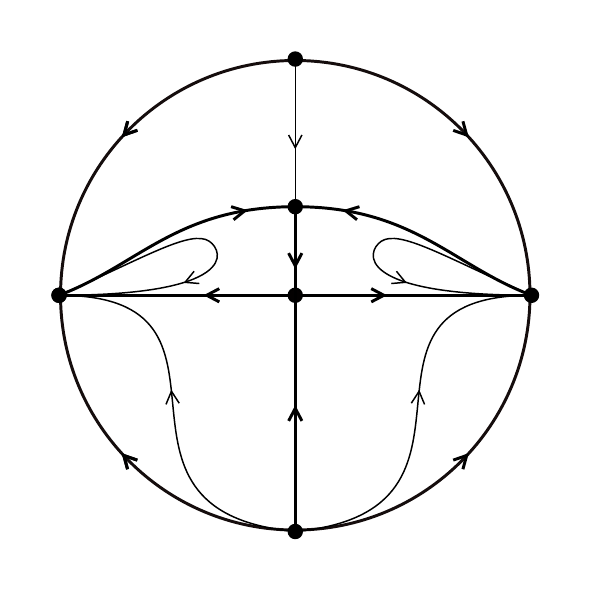}
		\caption*{\scriptsize (G85) [R=5, S=16]}
	\end{subfigure}
	\begin{subfigure}[h]{2.4cm}
		\centering
		\includegraphics[width=2.4cm]{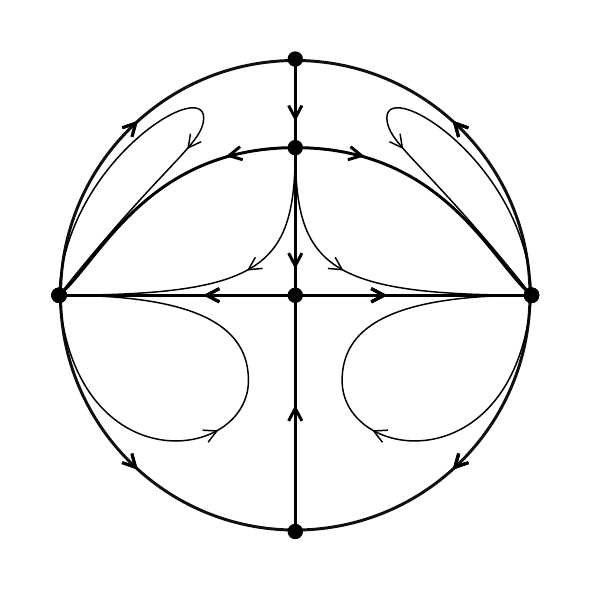}
		\caption*{\scriptsize (G86) [R=6, S=17]}
	\end{subfigure}
	\begin{subfigure}[h]{2.4cm}
		\centering
		\includegraphics[width=2.4cm]{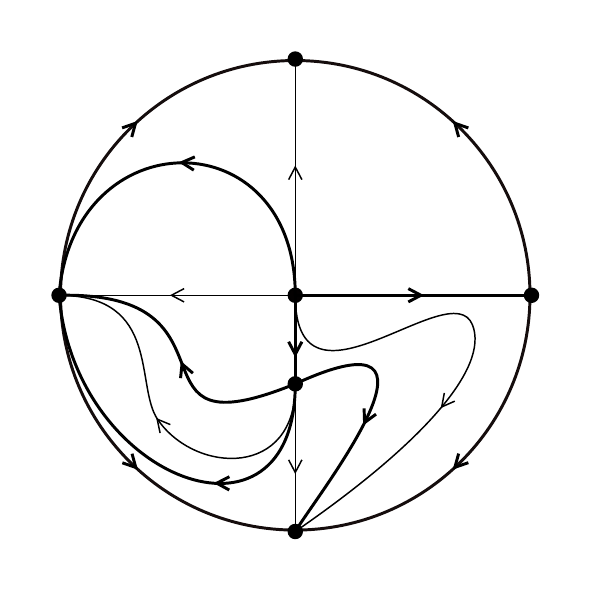}
		\caption*{\scriptsize (G87) [R=5, S=16]}
	\end{subfigure}
	\begin{subfigure}[h]{2.4cm}
		\centering
		\includegraphics[width=2.4cm]{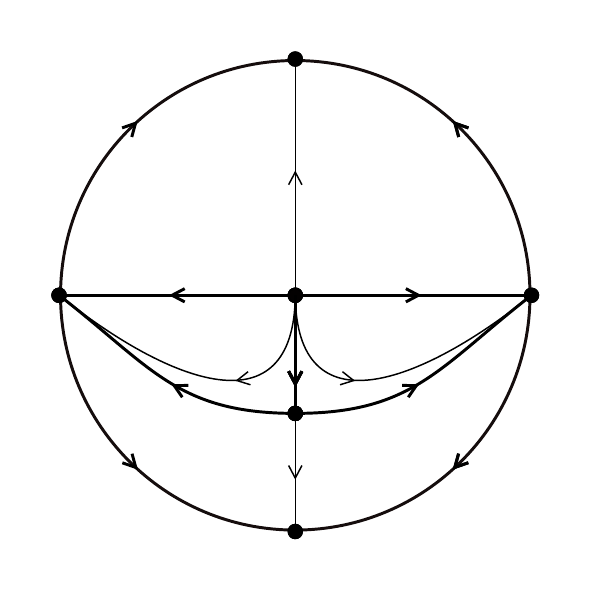}
		\caption*{\scriptsize (G88) [R=4, S=15]}
	\end{subfigure}
	\begin{subfigure}[h]{2.4cm}
		\centering
		\includegraphics[width=2.4cm]{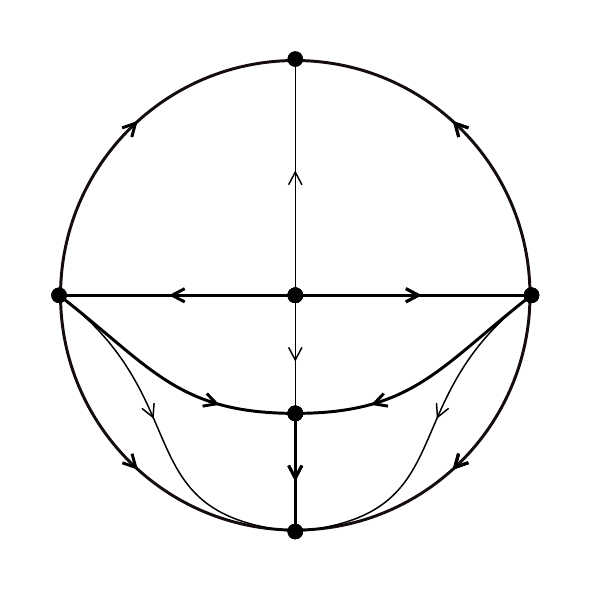}
		\caption*{\scriptsize (G89) [R=4, S=15]}
	\end{subfigure}
	\begin{subfigure}[h]{2.4cm}
		\centering
		\includegraphics[width=2.4cm]{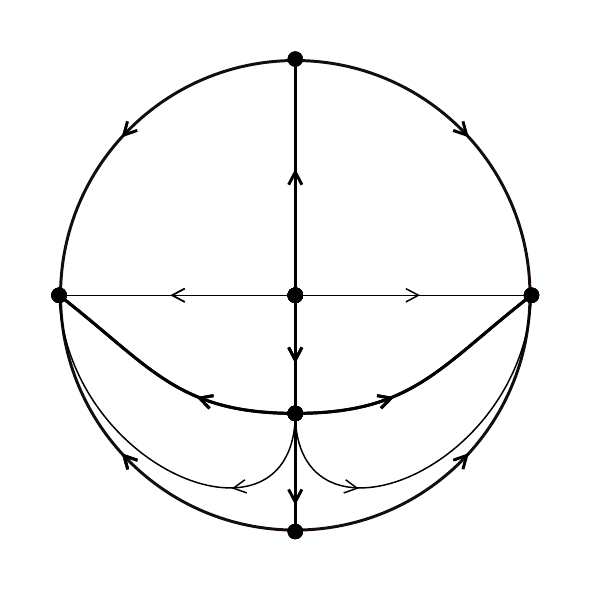}
		\caption*{\scriptsize (G90) [R=4, S=15]}
	\end{subfigure}	
	\begin{subfigure}[h]{2.4cm}
		\centering
		\includegraphics[width=2.4cm]{chap2/global/G91}
		\caption*{\scriptsize (G91) [R=5, S=16]}
	\end{subfigure}
	\begin{subfigure}[h]{2.4cm}
		\centering
		\includegraphics[width=2.4cm]{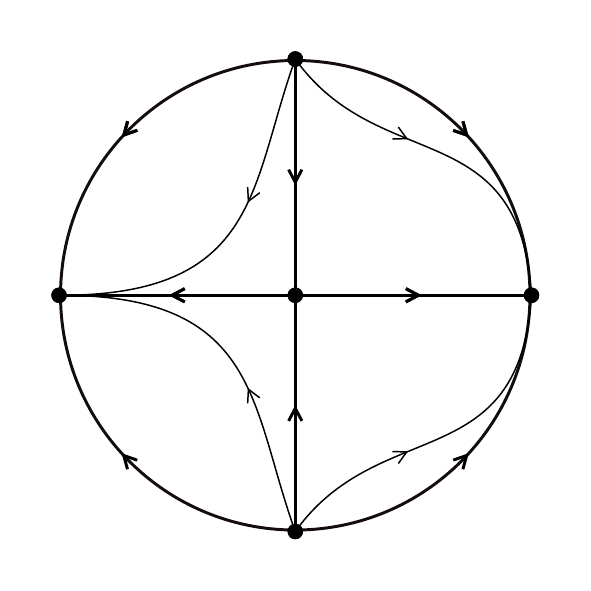}
		\caption*{\scriptsize (G92) [R=4, S=13]}
	\end{subfigure}
	\begin{subfigure}[h]{2.4cm}
		\centering
		\includegraphics[width=2.4cm]{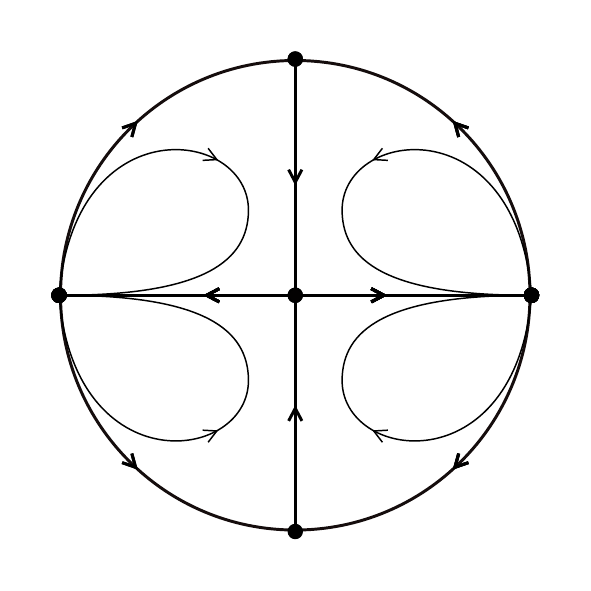}
		\caption*{\scriptsize (G93) [R=4, S=13]}
	\end{subfigure}
	\begin{subfigure}[h]{2.4cm}
		\centering
		\includegraphics[width=2.4cm]{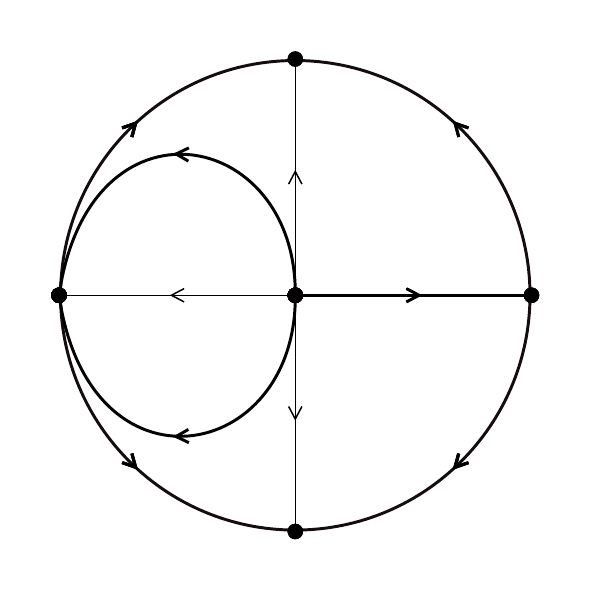}
		\caption*{\scriptsize (G94) [R=3, S=12]}
	\end{subfigure}
	\begin{subfigure}[h]{2.4cm}
		\centering
		\includegraphics[width=2.4cm]{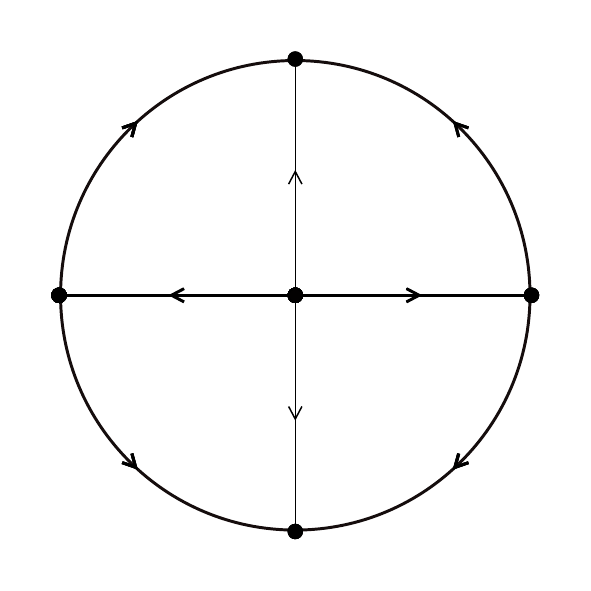}
		\caption*{\scriptsize (G95) [R=2, S=11]}
	\end{subfigure}
	\begin{subfigure}[h]{2.4cm}
		\centering
		\includegraphics[width=2.4cm]{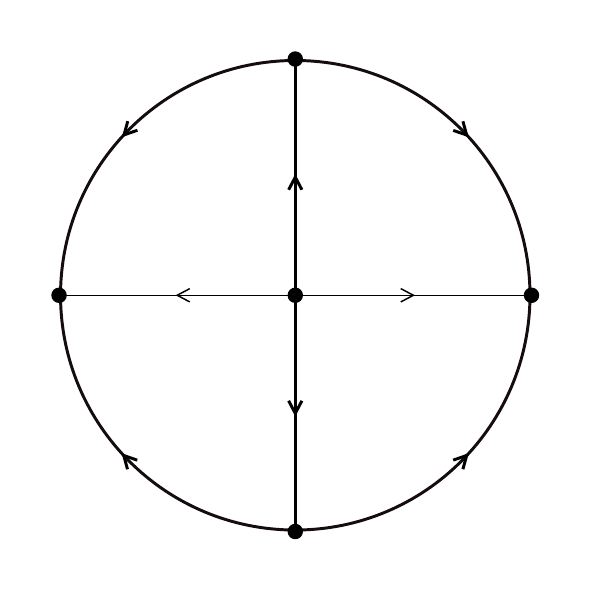}
		\caption*{\scriptsize (G96) [R=2, S=11]}
	\end{subfigure}
		\begin{subfigure}[h]{2.4cm}
		\centering
		\includegraphics[width=2.4cm]{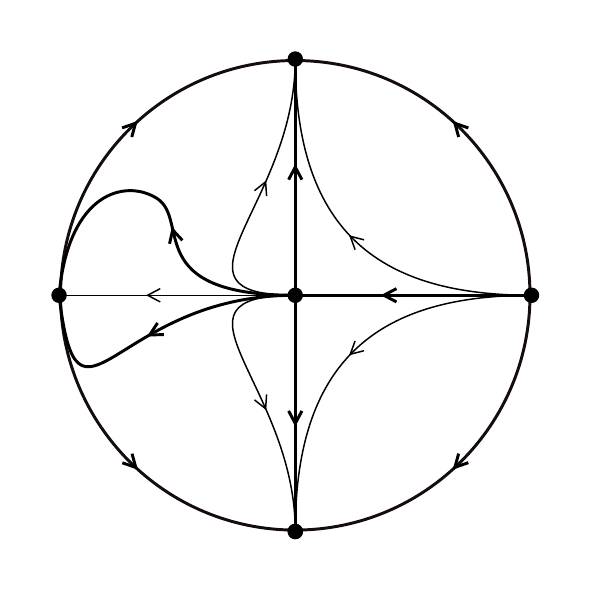}
		\caption*{\scriptsize (G97) [R=5, S=14]}
	\end{subfigure}
	\begin{subfigure}[h]{2.4cm}
		\centering
		\includegraphics[width=2.4cm]{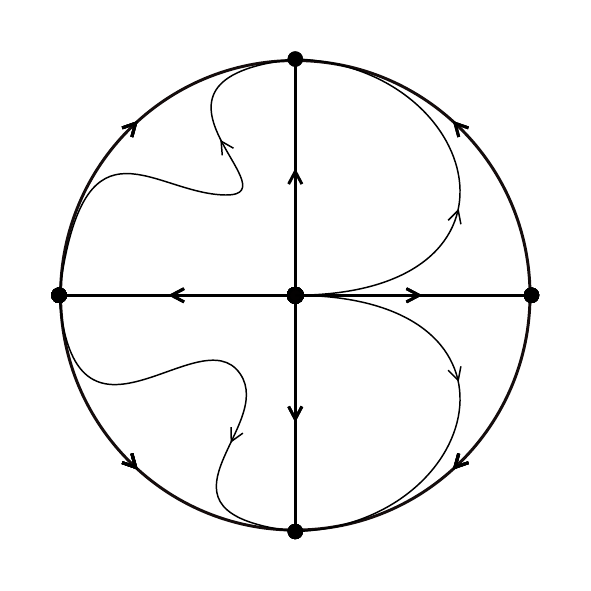}
		\caption*{\scriptsize (G98) [R=4, S=13]}
	\end{subfigure}	
	\begin{subfigure}[h]{2.4cm}
		\centering
		\includegraphics[width=2.4cm]{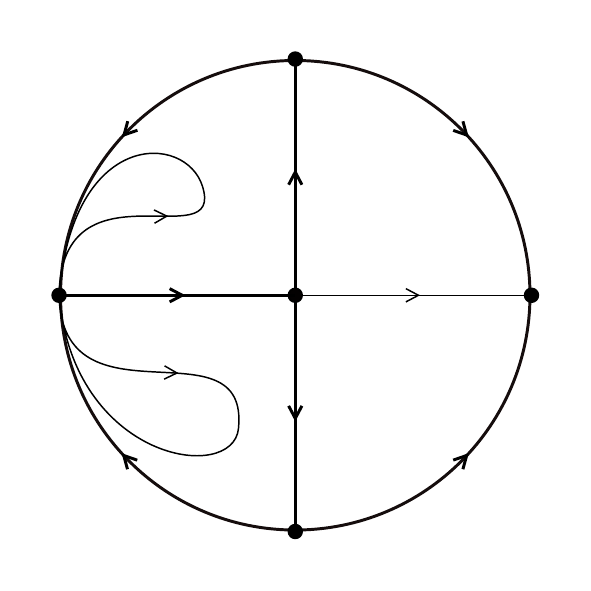}
		\caption*{\scriptsize (G99) [R=3, S=12]}
	\end{subfigure}
	\begin{subfigure}[h]{2.4cm}
		\centering
		\includegraphics[width=2.4cm]{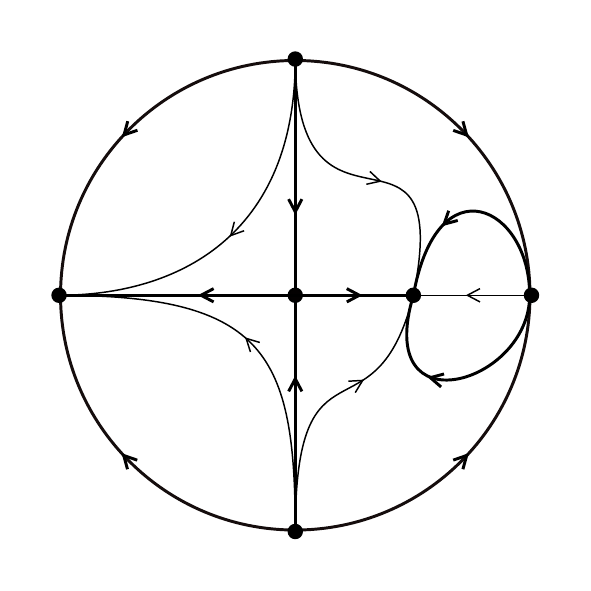}
		\caption*{\tiny (G100) [R=5, S=16]}
	\end{subfigure}
	\begin{subfigure}[h]{2.4cm}
		\centering
		\includegraphics[width=2.4cm]{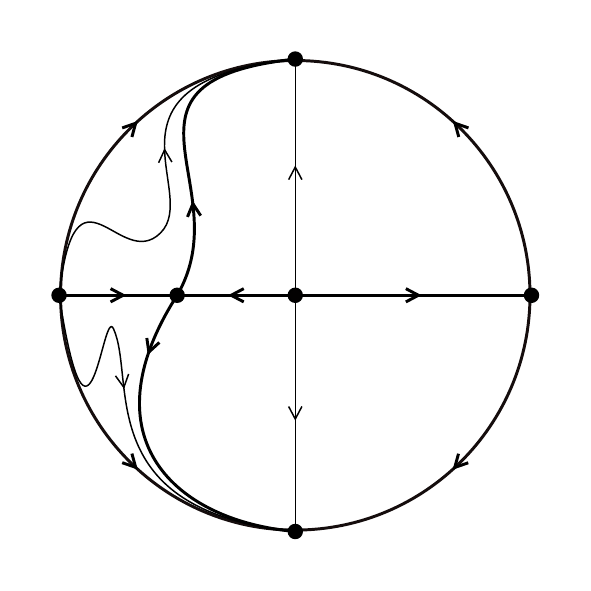}
		\caption*{\tiny (G101) [R=4, S=15]}
	\end{subfigure}
	\begin{subfigure}[h]{2.4cm}
		\centering
		\includegraphics[width=2.4cm]{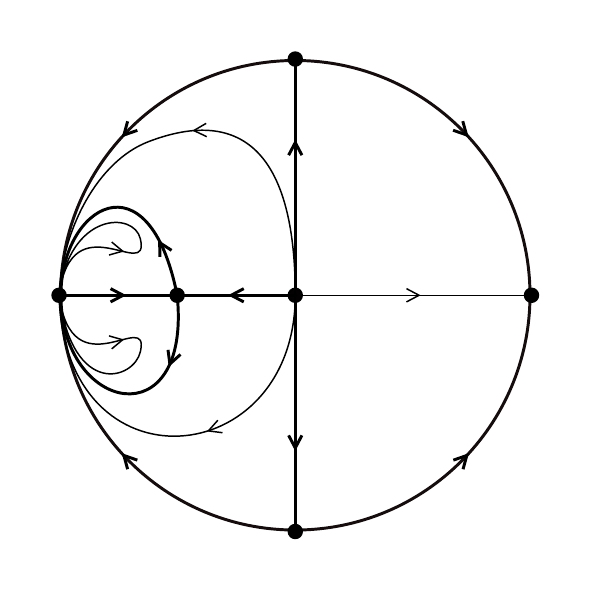}
		\caption*{\tiny (G102) [R=5, S=16]}
	\end{subfigure}
\caption{Global phase portraits of system \eqref{system} in the Poincar\'e disc.}
\label{fig:global_sis2.1}
\end{figure}

\begin{remark}
	Global phase portraits (G43), (G48), (G57), (G84) and (G92) can appear both under condition $c_1=0$ or under $c_1\neq0$ so, as we are interested in the topological classification we represent only the non-symmetric case respect to $z$-axis, but also the symmetric is possible.
	The same situation occurs with the global phase portraits
	(G9), (G13)--(G16), (G18), (G19), (G23)--(G27), (G31)--(G36), (G41), (G45)--(G49), (G55), (G56), (G76)--(G83), (G92)--(G102),
	%(G9), (G13), (G14), (G15), (G16), (G18), (G19), (G23), (G24), (G25), (G26), (G27), (G31), (G32), (G33), (G34), (G35), (G36), (G41), (G45), (G46), (G47), (G48), (G49), (G55), (G56), (G76), (G77), (G78), (G79), (G80), (G81), (G82), (G83), (G92), (G93), (G94), (G95), (G96), (G97), (G98), (G99), (G100), (G101) and (G102),
	which can appear under the condition $c_3=0$ or $c_3\neq0$ so they can present the symmetric or the non-symmetric form respect to $x$-axis, altought we only represent one of them because we are only interested on the topological classification.
\end{remark}

\subsection*{Acknowledgements}

The first and third authors are partially supported by the Ministerio de Economía, Industria y Competitividad, Agencia Estatal de Investigación (Spain), grant MTM2016-79661-P (European FEDER support included, UE) and the Consellería de Educación, Universidade e Formación Profesional (Xunta de Galicia), grant ED431C 2019/10 with FEDER funds. The first author is also supported by the Ministerio de Educacion, Cultura y Deporte de España, contract FPU17/02125.

The second author is partially supported by the Ministerio de Ciencia, Innovaci\'on y Universidades, Agencia Estatal de Investigaci\'on grants MTM2016-77278-P (FEDER), the Ag\`encia de Gesti\'o d'Ajuts Universitaris i de Recerca grant 2017SGR1617, and the H2020 European Research Council grant MSCA-RISE-2017-777911.

\end{document}